\documentclass[10pt]{article}

\usepackage[top=1in,bottom=1in,left=1.25in,right=1.25in]{geometry}

\usepackage[absolute,overlay]{textpos}
\usepackage{amssymb}
\usepackage{graphicx}
\usepackage{ragged2e}
\usepackage{tabto}

\usepackage{amsmath}
\usepackage{amsthm}
\usepackage{amsfonts}

\usepackage{float}
\usepackage{multirow}
\usepackage{tikz}
\usepackage{graphics}
\usepackage{color}
\usepackage{url}
\usepackage{afterpage}
\usepackage{multicol}
\usepackage[font=footnotesize]{caption}
\usepackage[font=footnotesize]{subcaption}
\usepackage{marginnote}
\usepackage{tcolorbox}

\usepackage{lmodern}

\usepackage[sf,bf,medium]{titlesec}
\usepackage{pgfplots}

\usepackage{bm}
\usepackage{booktabs}
\usepackage{siunitx}
\usepackage{isomath}

\usepackage[plainpages=false, colorlinks=true, citecolor=blue, filecolor=blue,
   linkcolor=blue, urlcolor=blue]{hyperref}
\usepackage{enumitem}

\usepackage{multicol}
\newcommand{\lp}{\left(}
\newcommand{\rp}{\right)}

\usepackage[sort,compress]{cite}

\DeclareMathOperator\erf{Erf}

\DeclareMathOperator\Th{th}

\newcommand\bbR{\mathbb R}
\newcommand\bbC{\mathbb C}
\newcommand\bbN{\mathbb N}

\newcommand\astar{\alpha_k}

\newcommand\bx{\boldsymbol{x}}

\newcommand\by{\boldsymbol{y}}

\newcommand\Int{\operatorname{int}}

\newcommand\oQ{\overline{Q}}

\newcommand\restrictedto{\upharpoonright}

\newtheorem{theorem}{\sffamily Theorem}
\newtheorem{remark}{\sffamily Remark}
\newtheorem{definition}{\sffamily Definition}

\newtheorem{lemma}{\sffamily Lemma}
\newtheorem{proposition}{\sffamily Proposition}
\newtheorem{corollary}{\sffamily Corollary}

\newcommand{\dd}{{\rm d}}

\newcommand{\cD}{\mathcal D}
\newcommand{\cS}{\mathcal S}

\newcommand{\cI}{\mathcal I}

\newcommand{\cC}{\mathcal C}

\newcommand{\cK}{\mathcal K}

\newcommand{\cG}{\mathcal G}
\newcommand{\cA}{\mathcal A}

\newcommand{\fB}{\mathfrak B}
\newcommand{\GammaC}{\Gamma_{\bbC}}
\newcommand{\GammaL}{\Gamma_{L}}
\newcommand{\GammaR}{\Gamma_{R}}
\newcommand{\GammaM}{\Gamma_{M}}
\newcommand{\tGamma}{\widetilde{\Gamma}}
\newcommand{\hGamma}{\widehat{\Gamma}}

\newcommand{\ttau}{\widetilde\tau}
\newcommand{\tsigma}{\widetilde\sigma}

\numberwithin{equation}{section}

\newcommand{\pa}{\partial}

\newcommand{\figref}[1]{\figurename~\hyperref[#1]{\ref{#1}}}
\newcommand{\tableref}[1]{\tablename~\hyperref[#1]{\ref{#1}}}

\newcommand{\lt}{L_{C}}

\newcommand{\Id}{\cI}
\newcommand{\qand}{\quad\text{and}\quad}

\binoppenalty=\maxdimen
\relpenalty=\maxdimen

\DeclareMathOperator{\sign}{sign}
\DeclareMathOperator{\supp}{supp}

\newcommand{\eps}{\varepsilon}

\newcommand{\ximin}{\xi_{{\rm min}}}
\newcommand{\xiset}{\Xi}
\newcommand{\ls}{\rho}
\newcommand{\causal}{\mathcal B}

\usepackage{setspace}
\usetikzlibrary{calc}
\usepackage{abstract}

\begin{document}

\begin{titlepage}

  \hrulefill\hskip 10mm \today 
  
  \raggedright
 \begin{textblock*}{\linewidth}(1.25in,2in) 
   {\LARGE \sffamily \bfseries Complex scaling for open waveguides}
  \end{textblock*}

\normalsize
\vspace{2in}
\vspace{\baselineskip}

Charles L. Epstein\\
\emph{\small Center for Computational Mathematics\\
Flatiron Institute\\
New York, NY 10010}\\
\texttt{cepstein@flatironinstitute.org}

\vspace{\baselineskip}

Tristan Goodwill\\
\emph{\small Department of Statistics and CCAM\\
University of Chicago\\
Chicago, IL 60637}\\
\texttt{tgoodwill@uchicago.edu}

\vspace{\baselineskip}

Jeremy Hoskins\\
\emph{\small Department of Statistics and CCAM\\
University of Chicago\\
Chicago, IL 60637}\\
\texttt{jeremyhoskins@uchicago.edu}

\vspace{\baselineskip}

Solomon Quinn\\
\emph{\small Center for Computational Mathematics\\
Flatiron Institute\\
New York, NY 10010}\\
\texttt{squinn@flatironinstitute.org}

\vspace{\baselineskip}

Manas Rachh\\
\emph{\small Center for Computational Mathematics\\
Flatiron Institute\\
New York, NY 10010}\\
\texttt{mrachh@flatironinstitute.org}


  
\end{titlepage}

\begin{abstract}
  \noindent In this work we analyze the complex scaling method applied to the
  problem of time-harmonic scalar wave propagation in junctions between
   `leaky,' or open dielectric waveguides. In
  \cite{epstein2023solvinga,epstein2023solvingb,epstein2024solving,goodwill2024numerical},
  it was shown that under suitable assumptions the problem can be reduced to a
  system of Fredholm second-kind integral equations on an infinite interface,
  transverse to the waveguides. Here, we show that the kernels appearing in the
  integral equation admit a rapidly decaying analytic continuation on certain
  natural totally real submanifolds of $\mathbb{C}^2.$ We then show that for
  suitable, physically-meaningful, boundary data the resulting solutions to the
  integral equations themselves admit analytic continuation and satisfy related
  asymptotic estimates. By deforming the integral equation to a suitable
  contour, the decay in the kernels, density, and data enable
  straightforward discretization and truncation, with an error that decays
  exponentially in the truncation length. We illustrate our results with several
  representative numerical examples. \\
  
  \noindent {\bfseries Keywords}: \\
  Open waveguides, scattering, Fredholm integral equations, complexification, infinite boundary, outgoing solution
  
  \noindent {\bfseries AMS subject classifications}: 65N80, 35Q60, 35C15, 35P25,
  45B05, 31A10, 30B40, 65R20

\end{abstract}
\tableofcontents

\normalsize

\section{Introduction}
Waveguides are important and ubiquitous tools for modeling acoustic,
electromagnetic, and fluidic problems. Their role in long distance transmission
of signals makes them ideal for use in a wide variety of applications, \textit{inter alia},
photonic circuits
\cite{kim2011graphene,bogaerts2020programmable,thylen2004photonic,meng2021optical,yang2018bridging,li2008harnessing,roeloffzen2013silicon,kopp2010silicon,ma2002polymer},
medical devices
\cite{shabahang2018light,wang2020optical,yan2019advanced,zhang2019multifunctional,nicholson1992waveguides,guo2019soft,chen2019double},
fiber optics \cite{argyros2013microstructures,eggleton2001microstructured,
  russell2003photonic,selvaraja2018review,kao1966dielectric,pal1992fundamentals},
and loudspeakers
\cite{makivirta2017acoustic,kuroda2020parametric,bezzola2018numerical,zangeneh2018acoustic,iuganson2023development}. In
many contexts the waveguides are modeled as closed systems --- not coupled in
any way to the surrounding material. Still, in a number of situations,
particularly in photonic devices, coupling with the surrounding medium is non-trivial
and an important source of signal degradation. This coupling poses significant
challenges both analytically and numerically. Essentially, the radiated field in
the surrounding medium decays slowly in space, which, when combined with the
non-compact nature of the waveguides, renders approaches based on truncation
difficult or impossible.

In two dimensions the basic mathematical model is the following partial differential equation (PDE)
\begin{equation}\label{eqn:the_pde}
   [ \Delta  + k^2(1 + \chi_{(-\infty,0)}(x_1) q_l(x_2) +\chi_{[0,\infty)}(x_1) q_r(x_2))] u(\bx) = f(\bx),  
\end{equation}
where $\bx = (x_1,x_2),$ $k\in \mathbb{R}$ denotes the wavenumber of the ambient
material, and the functions $q_{l},q_r$ are compactly supported piecewise smooth
functions that encode the material properties of the left and right waveguide
respectively. For uniqueness,
the above PDE must be supplemented with suitable conditions at infinity, with
corresponding constraints placed on the source term $f$. Though fairly
intuitive, the full description of these radiation conditions is rather
involved, and we refer the interested reader to~\cite{epstein2024solving}.

In~\cite{epstein2023solvinga,epstein2023solvingb,epstein2024solving} the
equation \eqref{eqn:the_pde} is analyzed by reducing it to a system of boundary
integral equations (BIEs). More precisely, the plane is divided into two halves,
with their interface chosen so as to cut the waveguides orthogonally along their
intersection, i.e., the line $\{x_1 = 0\}$. By introducing suitable layer
potentials generated by densities on the interface, it was shown that the
resulting boundary integral equation (BIE) for these densities has unique
solutions in suitable function spaces. In principle one can use this formulation
directly to solve \eqref{eqn:the_pde}. In practice the decay of the densities on
this interface, and the decay of the kernels of the integral operators, is too
slow to make such an approach practical. In the subsequent
paper~\cite{goodwill2024numerical}, building on the method of
\cite{lu2018perfectly}, it was shown empirically that by introducing {\it
  complex scaling} along the interface, i.e., replacing the line $\{x_1 = 0\}$
by any curve in a family of totally real submanifolds of
$\{0\}\times\bbC\subset\bbR\times\bbC,$ this boundary integral equation can be
solved accurately and efficiently on a bounded domain. Here, extending the
method of~\cite{epstein2024coordinate}, and the analytic approach of
\cite{epstein2024coordinate}, we give a rigorous proof of the validity of this
approach. In particular, we show that on suitably chosen deformed contours, the
analytic continuation of the BIEs
of~\cite{epstein2023solvinga,epstein2023solvingb,epstein2024solving} admit
unique, and rapidly decaying solutions. Additionally, when these contours are
truncated, as is necessary for numerical computations, the resulting
approximations to the solutions of \eqref{eqn:the_pde} converge rapidly to the
true solution. We further extend the method of~\cite{goodwill2024numerical} to
account for waveguides with material parameters that are piecewise smooth in
the transverse direction.

One proposed approach for solving problems of this class is via domain
decomposition, discretizing the problem partially in the Fourier
domain~\cite{dhia2024complex}. For simple problems one can compute an operator
which takes `right going' to `left going' waves for each half-space. The solution
to the full problem can then be obtained by enforcing continuity across the
interface. Though interesting analytically, these methods have yet to see
practical numerical realization, and are relatively inflexible in the presence
of perturbations and non-perpendicular waveguides. An alternate philosophy is
instead to discretize only the boundary of the waveguide. Though the analytic
structure of the equation \eqref{eqn:the_pde} prevents straightforward
truncation of the resulting BIEs, several methods have been developed to produce
accurate solutions on suitably truncated problems. One approach, referred to as
{\it windowed Green's functions}, uses partitions of unity to discretize the
waveguides near the boundary in physical space, and the infinite flat part of
the geometry in Fourier space~\cite{bruno2016windowed,lai2018new}. Though these
methods are empirically effective in an array of contexts, they do not directly
address the solvability of the integral equation on the untruncated domain, nor
its truncated counterpart.

A related set of methods involve introduction of perfectly matched layers (PMLs), which frequently arise in the restriction of PDEs defined on infinite domains to compact regions. These methods use the outgoing nature of the solutions to extend them to a two dimensional totally real submanifold of $\mathbb{C}^{2}$ where the solution is exponentially decaying and can be effectively truncated~\cite{berenger1994perfectly,chew19943d,dyatlov2019mathematical}. Recently there has also been work on integral equation based PML methods \cite{lu2018perfectly}, and our method is of this general character. As yet, it is still an open problem to apply these techniques to BIE formulations of non-compact waveguide problems such as \eqref{eqn:the_pde}.

In this paper, as in~\cite{goodwill2024numerical}, we consider a different approach, solving on the fictitious interface as in~\cite{epstein2023solvinga,epstein2023solvingb,epstein2024solving} rather than on the boundary of the waveguide itself. By `complexifying' the coordinates of this artificial interface, we prove that one obtains an invertible Fredholm  integral equation of second kind. The solutions of this new integral equation can then be used to recover the solution of the original equation \eqref{eqn:the_pde}. Significantly, the solutions of this `complexified' BIE decay exponentially, enabling straightforward truncation to be applied in numerical implementations. For a given choice of complexification, which depends on the choice of a positive number $\lt,$ the solution of the BIE can then be used to calculate the solution $u(\bx)$ for $\{\bx: |x_2|\leq \lt\}.$

The remainder of the paper is as follows. In Section \ref{sec:real_IE}, we review the construction of the BIEs in ~\cite{epstein2023solvinga,epstein2023solvingb,epstein2024solving}. Following this, in Section \ref{sec:main_res} we state our main analytic result. Sections \ref{sec:op_prop}, and \ref{sec:IE_prop} contain the proofs of these main results. Finally, in Section \ref{sec:numerics} we briefly describe a numerical method for approximately solving the complexified BIEs, extending the work in~\cite{goodwill2024numerical}, to allow for piecewise smooth (in the transverse coordinate) waveguide parameters.

\section{Boundary integral equation for interface matching}\label{sec:real_IE}

Following~\cite{epstein2023solvinga,epstein2023solvingb,epstein2024solving}, we solve the PDE \eqref{eqn:the_pde} by first replacing  it with a transmission problem, which is then reduced to a system of boundary integral equations on a fictitious interface separating the two semi-infinite waveguides.  We assume that~$q_{l,r}$ are piecewise
continuous and supported on the interval~$[-d,d]$. For $\Im k>0$ we let
$G^{l,r}_k$ denote the kernels of the inverses  of the self adjoint partial differential operators
\begin{align}
    \Delta_x G^{l,r}_k(\bx;\by) + k^2(1+q_{l,r}(x_2)) G^{l,r}_k(\bx;\by) = \delta(\bx-\by).
\end{align}
 Self adjointness implies that for $k$ in the closed upper half plane  the fundamental solution satisfies the relation
$$G^{l,r}_{k}(\bx;\by)=\overline{G^{l,r}_{-\overline{k}}(\by;\bx)}.$$
Because the potential is real, these kernels are real for $k=it,$ which implies that they also satisfy
$$G^{l,r}_{k}(\bx;\by)=G^{l,r}_k(\by;\bx).$$ 
Henceforth, we let $k$ be a fixed positive number and set
\begin{equation}
    G^{l,r}(\bx;\by)=\lim_{\epsilon\to 0^+}G_{k+i\epsilon}^{l,r}(\bx;\by).
\end{equation}
That is, $G^{l,r}$ is the kernel of limiting absorption resolvent 
\begin{equation}
    \lim_{\epsilon\to 0^+}[\Delta_x + (k+i\epsilon)^2(1+q_{l,r}(x_2))]^{-1}.
\end{equation}
By continuity this kernel also satisfies the relation $G^{l,r}(\bx;\by)=G^{l,r}(\by;\bx).$

Using these Green's functions we can remove the volumetric source $f$, reducing the problem to an auxiliary transmission problem on the line $\{(0,x_2)\,|\,x_2 \in \mathbb{R}\}.$ In particular, if $\Omega_{l,r}$ denotes the left and right half-spaces then, setting
\begin{equation}
    u_i^{l,r}(\bx) = \int_{\bbR^2} G^{l,r}(\bx,\by)\,f(\by) \,{\rm d}\by \, ,\quad \textrm{for } \bx \in \Omega_{l,r} \,,
    \end{equation}
and $u^{l,r} = u - u_i^{l,r}$ on $\Omega_{l,r},$ then the functions $u^{l,r}$ satisfy
\begin{equation}
    \begin{split}
\label{eqn:the_reduced_pde}
    &\Delta_x u^{l,r}(\bx) + k^2(1+q_{l,r}(x_2)) u^{l,r}(\bx) = 0,\quad {\rm on}\, \Omega^{l,r},\\
    &\lim_{\epsilon \to 0^+} \left[u^{r}(\epsilon,x_2) - u^{l}(-\epsilon,x_2)\right] = g(x_2), \quad x_2 \in \mathbb{R},\\
        &\lim_{\epsilon \to 0^+} \left[\left.\frac{\partial u^{r}}{\partial x_1}\right|_{x_1 = \epsilon}(x_1,x_2) - \left.\frac{\partial u^{l}}{\partial x_1}\right|_{x_1 = - \epsilon}(x_1,x_2)\right] = h(x_2), \quad x_2 \in \mathbb{R},
\end{split}
\end{equation}
where 
\begin{align}
g(x_2) &:=-\lim_{\epsilon \to 0^+} \left[u_{i}^{r}(\epsilon,x_2) - u_{i}^{l}(-\epsilon,x_2)\right]\\
h(x_2) &:= -\lim_{\epsilon \to 0^+} \left[\frac{\partial u_i^{r}}{\partial x_1}(\epsilon,x_2) - \frac{\partial u_i^{l}}{\partial x_1}(-\epsilon,x_2)\right].
\end{align}
In the absence of either waveguide, the fundamental solution is given by
$$G_0(\bx;\by) = \frac{i}{4} H_0(k\|\bx-\by\|)\in H^{1-\epsilon}_{\rm loc}(\bbR^2),$$
for any $\epsilon>0,$ where $H_0$ is the zeroth order Hankel function. Defining
\begin{align}
    w^{l,r}(\bx;\by) = G^{l,r}(\bx;\by)- G_0(\bx;\by),
\end{align}
it follows immediately from standard elliptic regularity theory that, for fixed
$\by,$ $w^{l,r}(\cdot;\by)\in H^{3-\epsilon}_{\rm loc}(\bbR^2).$

We next introduce the following ansatz for the left and right solutions:
\begin{align}\label{eqn:int_rep}
    u^{l,r}(\bx) = \int_{-\infty}^\infty G^{l,r}(\bx;0,y_2) \tau(y_2) \,{\rm d}y_2 - \int_{-\infty}^\infty \left.\frac{\partial}{\partial y_1}\right|_{y_1=0}G^{l,r}(\bx;y_1,y_2) \sigma(y_2) \,{\rm d}y_2, 
\end{align}
where $\tau$ and $\sigma$ are two unknown densities. Restricting this representation to $\{x_1=0\}$, and applying the transmission boundary conditions in~\eqref{eqn:the_reduced_pde}, leads to the integral equations for $[\sigma,\tau],$
\begin{align}\label{eqn:first_int}
\sigma(x_2) - \int_{-\infty}^{\infty} &\frac{\partial}{\partial y_1}\left[w^{r}(0,x_2;0,y_2)-w^{l}(0,x_2;0,y_2)\right] \sigma(y_2)\,{\rm d}y_2+ \\
&\int_{-\infty}^{\infty} [w^{r}(0,x_2;0,y_2)-w^{l}(0,x_2;0,y_2)] \tau(y_2)\,{\rm d}y_2 = g(x_2) \\
\tau(x_2) +\int_{-\infty}^{\infty} &\frac{\partial}{\partial x_1}\left[w^{r}(0,x_2;0,y_2)-w^{l}(0,x_2;0,y_2)\right] \tau(y_2)\,{\rm d}y_2- \\
\int_{-\infty}^{\infty}&\frac{\partial^2}{\partial x_1\partial y_1}\left[w^{r}(0,x_2;0,y_2)-w^{l}(0,x_2;0,y_2)\right] \sigma(y_2)\,{\rm d}y_2 = h(x_2). 
\end{align}
Here we have used the decomposition of $G^{l,r}$ into $w^{l,r}$ and $G_0$, together with the standard jump relations of $G_0$~\cite{colton2013integral}. We can make a further simplification by observing that $w^{l,r} (x_1,x_2;0,y_2)$ and $w^{l,r} (0,x_2;y_1,y_2)$ are even functions of $x_1$ and $y_1$, respectively. In particular, the terms involving the first derivatives of $w^{l,r}$ in \eqref{eqn:first_int} vanish. 

For ease of notation, we define the operators $C$ and $D$ by
\begin{align}
    D f(x_2) &:= \int_{-\infty}^{\infty} k_{D}(x_{2},y_{2})f(y_{2}) {\rm d}y_2 =  \int_{-\infty}^\infty  [w^{r}(0,x_2;0,y_2)-w^{l}(0,x_2;0,y_2)] f(y_2)\,{\rm d}y_2 \label{eqn:Ddef}\\
        C f(x_2) &= -\int_{-\infty}^{\infty} k_{C}(x_{2},y_{2})f(y_{2}) {\rm d}y_2 = -\int_{-\infty}^\infty  \frac{\partial^2}{\partial x_1\partial y_1}\left[w^{r}(0,x_2;0,y_2)-w^{l}(0,x_2;0,y_2)\right] f(y_2)\,{\rm d}y_2.\label{eqn:Cdef}
\end{align}
Using these definitions, we arrive at the following system of integral equations for the densities $\sigma$ and $\tau,$
\begin{align}
    \begin{bmatrix} \Id & D\\ C & \Id\end{bmatrix} \begin{bmatrix}
        \sigma \\ \tau
    \end{bmatrix}(x_2)  = \begin{bmatrix}
        g\\ h
    \end{bmatrix}(x_2), \quad {\rm for}\,\,{\rm all}\,\, x_2 \in \mathbb{R}.\label{eq:real_IE}
\end{align}
 Here $\Id$ denotes the identity operator.

\section{Main results}\label{sec:main_res}

In this section we present the central analytical results of this paper. Before stating the main theorems, we first introduce some necessary notation and results related to layer potentials on complex contours. We begin by defining a certain subset of the complex plane, in which the complexification of the integral equation~\eqref{eq:real_IE} is defined. Following that, we define complex extensions of the integral operators $D$ and $C$ from \eqref{eqn:Ddef} and \eqref{eqn:Cdef}, respectively, together with suitable function spaces on which they act. We conclude the section with Theorem \ref{thm:IE_solvable}, which gives existence and uniqueness of the solution to the complexification of~\eqref{eq:real_IE}, Theorem \ref{thm:IE_outgoing}, which connects the solutions of this complexified integral equation to solutions of \eqref{eqn:the_pde} and Theorem~\ref{thm:trunc_sol}, which bounds the error in computing the solutions of the complexified integral equation on truncated contours.

In our method and analysis we consider the following subsets of $\mathbb{C}$, which are convenient for defining suitable function spaces of `outgoing' data. 
\begin{definition}\label{def:subsets}
    Fix~$\lt>d$. Define the sets~$\Gamma_L,\Gamma_M, \Gamma_R\subset\bbC$,  given by
    \begin{equation*}
        \begin{aligned}
            \Gamma_L&:=\left\{z :\, \Re z<-\lt,\;\Im z\leq 0 \right\},\\
            \Gamma_M&:=\left\{z :\,|\Re z|\leq \lt,\;\Im z= 0 \right\}, \quad{\rm and}\\
            \Gamma_R&:=\left\{z :\, \Re z>\lt,\;\Im z\geq 0 \right\}.
        \end{aligned}
    \end{equation*}
    We define~$\GammaC$ by
    \begin{equation*}
        \GammaC:= \Gamma_L\cup\Gamma_M\cup\Gamma_R\,,
    \end{equation*}
    see Figure~\ref{fig:GammaC} for an illustration.
\end{definition}

\begin{remark}
  Below we show that the kernel functions $k_C(x_2,y_2),\, k_D(x_2,y_2)$ have
  analytic continuations to $\Gamma_{\bbC}\times\Gamma_{\bbC}.$ 
\end{remark}

\begin{figure}
    \centering
    \includegraphics[width = 0.6\textwidth]{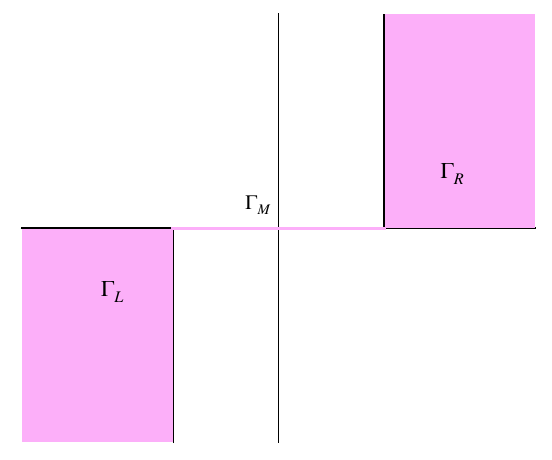}
    \caption{Plot showing $\GammaC$ in pink.}
    \label{fig:GammaC}
\end{figure}
Ultimately, we wish to formulate complex extensions of~\eqref{eq:real_IE} to admissible contours lying in $\GammaC$. The following definition provides sufficient conditions for a contour in $\GammaC$ to be compatible with our method. 
\begin{definition}
    Let~$\cG$ denote the set of all smooth curves~$\tGamma$ in~$\GammaC$ that satisfy the following two conditions.
    \begin{enumerate}[label=(\roman*)]
        \item The curve~$\tGamma$ is monotonic in the sense that if~$x_{2},y_{2}\in \tGamma$, then
        \begin{equation}\label{eqn3.1.30}
            \Im(x_2-y_2) \geq 0 \quad\text{implies}\quad \Re(x_2-y_2) \geq 0.
        \end{equation}
        \item For all~$t\in\bbR$ there is an~$x_2\in \tGamma$ with~$\Re(x_2) = t$. 
    \end{enumerate}
\end{definition}
We typically consider curves given by a parameterization $x_{2}(t)$ with ~$x_2(t) =
t+i\psi(t)$, where~$\psi$ is a smooth monotonically increasing function, which
is zero for $t\in [-\lt,\lt]$. We now define extensions of our integral operators
over curves in~$\cG$. In the following definitions we assume that
$(\sigma(y_2),\tau(y_2))$ are defined for $y_2\in\tGamma,$ and have sufficient
decay for the integrals to be defined. Precise hypotheses are given below.

\begin{definition}[Complex difference potentials]
\label{def:complex-diff}
Let~$\tGamma\in \cG$. For every~$x_{2}\in \GammaC$, the complexified operators on~$\tGamma$ are given by the complex contour integrals
    \begin{equation}
    C_{\tGamma}[\sigma](x_{2}) = \int_{\tGamma} k_{C}(x_{2}, y_{2})
    \sigma(y_2){\rm d}y_2
    \end{equation}
    and
    \begin{equation}
    D_{\tGamma}[\tau](x_{2}) = \int_{\tGamma}k_{D}(x_{2},y_{2}) \tau(y_2){\rm d}y_2
    \end{equation}
    where~$k_{C,D}$ are the analytic continuations of the kernels defined in~\eqref{eqn:Cdef} and~\eqref{eqn:Ddef}, respectively. In the following it will also be convenient to define the system matrix $\mathcal{K}_{\tilde{\Gamma}}$
     by
    \begin{equation}
    \cK_{\tGamma} = \begin{bmatrix}
        0 & D_{\tGamma}\\ C_{\tGamma}&0
    \end{bmatrix}.
\end{equation}
\end{definition}

We analyze the behavior of these operators on the following family of Banach spaces of piecewise analytic functions.
\begin{definition}
    For any real numbers~$\alpha$ and~$\beta$, we define the norm
    \begin{equation}
        \|f\|_{\alpha,\beta}:=\sup_{x\in \overline{\GammaC}} e^{\beta|\Im x |}\lp 1+|x|\rp^\alpha |f(x)|.
    \end{equation}
    The space~$\cC^{\alpha,\beta}$ is the set of all continuous functions~$f$ on~$\GammaC$, which are analytic in  ${\rm int}\Gamma_{L} \cup {\rm int} \Gamma_{R}$ and satisfy~$\|f\|_{\alpha,\beta}<\infty$.
\end{definition}

\begin{remark}
    Throughout the paper, unless stated otherwise, we  assume that $0<\alpha <\frac 12$ and $\beta > -k.$ 
\end{remark}

It is useful to define analogous  spaces for functions defined on individual curves $\tGamma\in\cG.$
\begin{definition} If $\tGamma\in\cG,$ then for any real numbers~$\alpha$, and~$\beta$, we define the norm
    \begin{equation}
        \|f\|_{\alpha,\beta,\tGamma}:=\sup_{x\in\tGamma} e^{\beta|\Im x |}\lp 1+|x|\rp^\alpha |f(x)|.
    \end{equation}
    The space~$\cC^{\alpha,\beta}_{\tGamma}$ is the set of all continuous functions~$f$ on~$\tGamma$, which  satisfy~$\|f\|_{\alpha,\beta,\tGamma}<\infty$.
\end{definition}
\noindent
The fact that the~$\cC^{\alpha,\beta}$ and $\cC^{\alpha,\beta}_{\tGamma}$  are Banach spaces is essentially proved in~\cite{epstein2024coordinate}.  Note that
\begin{equation}
    \cC^{\alpha,\beta}\restrictedto_{\tGamma\,}\subset \,\cC^{\alpha,\beta}_{\tGamma},
\end{equation}
and the spaces $\cC^{\alpha,\beta}_{\bbR}$ coincide with the spaces $\cC^{\alpha}(\bbR),$ which are defined in~\cite{epstein2023solvinga}.

In Section~\ref{sec:op_prop}, we prove that if the densities lie in~$\cC^{\alpha,\beta}$ then the value of the layer potential is independent of the path~$\tGamma$ used in their definition. 
In Section~\ref{sec:IE_prop}, we then use this path independence to extend the existence and uniqueness theorems in~\cite{epstein2023solvingb,epstein2024solving} to prove that the analytic continuation of~\eqref{eq:real_IE} is well-posed. This result is summarized in the following theorem. 

\begin{theorem}\label{thm:IE_solvable} Let $0<\alpha<\frac 12,$ and $-k<\beta\leq k.$
    If~$[f_D,f_N]\in \cC^{\alpha,\beta}\oplus \cC^{\alpha+\frac12,\beta},$ then there exists a unique pair of functions~$[\sigma,\tau]\in \cC^{\alpha,\beta}\oplus \cC^{\alpha+\frac12,\beta},$ such that for any $\tGamma\in\cG,$ 
    \begin{equation}
        \lp\Id + \cK_{\tGamma} \rp\begin{bmatrix}
            \sigma\\\tau
        \end{bmatrix}(x_2)
        =
        \begin{bmatrix}
            f_D\\f_N
        \end{bmatrix}(x_2) \quad\forall x_2\in\tGamma.\label{eq:comp_IE}
    \end{equation}
    Indeed, this equation actually holds for every~$x_2\in \GammaC$.
\end{theorem}

When solving equation~\eqref{eq:comp_IE} numerically, the densities $\sigma, \tau$
are obtained by solving the integral equation for $x_{2} \in \tGamma$ for some
fixed contour $\tGamma$. Given these densities defined on $\tGamma$, using the
analyticity of $f_{D}, f_{N}$, and the analytic properties of the complex
difference potentials discussed in Definition~\ref{def:complex-diff}, we see
that  $\sigma,\tau$ extend analytically to $\Gamma_{\bbC}$ and for $x_{2} \in
\GammaC$ can be computed via the formula:
\begin{equation}\label{eqn:eval_sigtau}
\begin{bmatrix}
\sigma\\
\tau
\end{bmatrix} (x_{2}) 
= \begin{bmatrix}
f_{D} \\
f_{N}
\end{bmatrix}(x_{2}) - \cK_{\tGamma} \begin{bmatrix}
\sigma\\
\tau
\end{bmatrix}(x_{2}).
\end{equation}
The integral operator in the last term in the equation above only depends on the
values of $\sigma$ and $\tau$ on $\tGamma$.

Surprisingly, data that is incoming with specified exponential growth at
$\infty$ in $\GammaC$ is admissible for the solution of
equation~\eqref{eq:comp_IE}. This allowed growth is related to the maximum
frequency of oscillations of `incoming' fields in the Isozaki
sense~\cite{epstein2024solving,isozaki1994generalization}, and `incoming' waves
oscillating with frequency greater than $-k$ are permissible. It should be noted
that while the integral equation is uniquely solvable for such incoming data,
additional restrictions may be necessary to show that the corresponding solutions to the PDE are outgoing.

In Section~\ref{sec:IE_prop}, we show that the solution of~\eqref{eq:comp_IE} can be used to construct a solution of the transmission problem~\eqref{eq:real_IE} and thus leads to a solution of the original PDE~\eqref{eqn:the_pde}. Before doing so, we define analytic continuations of the layer potential operators used to define $u^{l,r}.$
The decay properties of the analytic continuation of $G^{l,r}$ for evaluation on $x_{1} \neq 0$ are quite different from those of the kernels $k_{C}$ and $k_{D}$ on $x_{1} = 0$. This introduces the further restriction of $\beta > 0$. In conjunction with the restriction $|x_{2}| < \lt$, the analytic continuation of $G^{l,r}(\bx;0,y_{2})$ decays algebraically as a function of $|y_{2}|$ with a uniform rate of algebraic decay in $x_{1}$. This decay combined with the decay of the density implies that the layer potentials on $\tGamma$ are absolutely convergent and thus allows us to define the analytic extensions of the layer potentials. 
\begin{definition}[Complex layer potentials]    Let~$\tGamma\in\textit{} \cG$.  For every~$\bx \in \mathbb{R}^{2} \cap \{|x_{2}|< \lt\}$, the complexified operators with densities supported on~$\tGamma,$ are given by the complex contour integrals
    \begin{equation}
    \cS^{l,r}_{\tGamma}[\tau](\bx) = \int_{\tGamma}G^{l,r}(\bx;0,y_2) \tau(y_2){\rm d}y_2
    \end{equation}
    and
    \begin{equation}
    \cD^{l,r}_{\tGamma}[\sigma](\bx) = \int_{\tGamma} \partial_{y_1}G^{l,r}(\bx;0,y_2) \sigma(y_2){\rm d}y_2
    \end{equation}
    where~$G^{l,r}$ are the analytic continuations of the Green's functions of the left and right waveguide operators.  
\end{definition}

It turns out that physically relevant scattering data arising from volumetric or
point sources, and, after appropriate modifications, waveguide
modes (see~\eqref{eqn5.35.30}) result in boundary data in
$\cC^{\alpha,\beta}\oplus \cC^{\alpha+\frac12,\beta}$ with $\beta > 0$. Furthermore, for data arising from the waveguide modes, the contours
$\tGamma$ must also satisfy the following slope condition in $\GammaC$. $\tGamma
\in \cG$ is admissible if there exist $c,C>0$ and $R_{0}$ such that for all
$y_{2} \in \tGamma$ with $|y_{2}|>R_{0}>\lt$, we have
\begin{equation}
\label{eq:slop-cond}
c|\Re{y_{2}}| \leq |\Im{y_{2}}| \leq C |\Re{y_{2}}| \,.
\end{equation}
From a numerical efficiency perspective, the contours $\tGamma$ are anyway chosen to satisfy such criterion, and thus these conditions are effectively redundant.

The following result establishes the connection between the solution of the complexified integral equations and the solutions of the original PDE for certain choices of data.
\begin{theorem}\label{thm:IE_outgoing}
    Suppose $\beta\geq 0$ and the pair of functions $[f_D,f_N]\in \cC^{\frac12,\beta}\oplus \cC^{\frac32,\beta}$ satisfies
    \begin{equation}
       \begin{aligned}
        f_D(x_{2}) &\sim \frac{e^{ik|x_2|}}{|x_2|^{\frac 12}}\sum_{l=0}^{\infty}\frac{a^{\pm}_l}{|x_2|^l} \quad \text{and}\\
        f_N(x_{2}) &\sim \frac{e^{ik|x_2|}}{|x_2|^{\frac 32}}\sum_{l=0}^{\infty}\frac{b^{\pm}_l}{|x_2|^l} ,\label{eq:outgoing}
        \end{aligned}
    \end{equation}
    as~$x_2\to \pm \infty$ with~$x_2\in \bbR$.
    If the pair~$[\sigma_{\tGamma},\tau_{\tGamma}]\in \cC^{\alpha,\beta}_{\tGamma}\oplus \cC^{\alpha+\frac 12,\beta}_{\tGamma}$ solves equation~\eqref{eq:comp_IE} on some~$\tGamma \in \cG$, then this pair has an analytic continuation to an element of $[\sigma,\tau]\in\cC^{\frac 12,\beta}\oplus \cC^{\frac32,\beta}.$
    
    The pair of functions~$u^{l,r}=\cS^{l,r}_{\bbR}[\tau] + \cD^{l,r}_{\bbR}[\sigma]$, solves the the transmission problem in~\eqref{eqn:the_reduced_pde}, and satisfies  the outgoing radiation conditions given in~\cite{epstein2024solving}. Moreover, for any curve $\tGamma \in \mathcal{G}$ satisfying the slope condition in equation~\eqref{eq:slop-cond}, the solution constructed using the representation $u^{l,r}(\bx)=\cS^{l,r}_{\tGamma}[\tau](\bx)+\cD^{l,r}_{\tGamma}[\sigma](\bx)$ agrees with this solution for $\bx \in \mathbb{R}^{2} \cap \{|x_{2}|\leq \lt\}$. 
\end{theorem}

Practical numerical implementation requires truncation. In the following theorem we bound the error in the computed solution, when the contour used to find the density and compute the solution is truncated.
\begin{theorem}\label{thm:trunc_sol}
    Let  $0<\alpha<\frac 12,\, 0<\beta\leq k$ and~$\tGamma$ be a contour
    in~$\cG$ satisfying the slope condition in equation~\eqref{eq:slop-cond}. Let~$\tGamma_{\epsilon}$
    be its truncation given by $\tGamma_{\epsilon} = \tGamma \cap \{ x_{2} \in \GammaC: \, e^{-\beta |\Im x_{2}|} \geq \epsilon \}$.
    Suppose that~$f_D$ and~$f_N$ satisfy the conditions of the previous
    theorem. For sufficiently small $\epsilon>0$ there is a unique pair $[\sigma_{\epsilon},\tau_{\epsilon}]\in\cC^{\alpha,\beta}\oplus
    \cC^{\alpha+\frac12,\beta}$ that solves
    \begin{equation}
        \lp\Id + \cK_{\tGamma_{\epsilon}}\rp\begin{bmatrix}
            \sigma_{\epsilon}\\ \tau_{\epsilon}
        \end{bmatrix}(x_2)=\begin{bmatrix}
            f_D \\ f_N
        \end{bmatrix}(x_2)\quad \forall x_2\in \tGamma_{\epsilon}.
    \end{equation}
    Further, there exists a~$C>0,$
    independent of $\epsilon$ and the data, such that
    \begin{equation}\label{eqn3.15.10}
        \left|\lp \cS^{l,r}_{\tGamma_{\epsilon}}[\tau_{\epsilon}](\bx) {-} \cD^{l,r}_{\tGamma_{\epsilon}}[\sigma_{\epsilon}](\bx)\rp - u^{l,r}(\bx) \right| <C \epsilon \left( \|f_{D}\|_{\cC^{\alpha,\beta}} + \|f_{N}\|_{\cC^{\alpha+\frac 12,\beta}} \right)
    \end{equation}
    for all~$\bx\in \bbR^{2} \cap \{|x_{2}|\leq \lt \}$, 
    where~$u^{l,r}$ is the solution of~\eqref{eqn:the_reduced_pde} which satisfies the outgoing radiation conditions.
\end{theorem}


\begin{remark}
The restrictions that $\beta>0$, and $\tGamma$ satisfies the slope condition
in Theorems~\ref{thm:IE_outgoing} and~\ref{thm:trunc_sol} can be relaxed if one
restricts $x_1$ to also lie in a fixed interval $[-L_0,L_0],$ for any
$L_{0}\in (0,\infty)$. In this case we only
need to require that $\beta>-k$ in order for the solution constructed via the
layer potentials to agree with the solution of the transmission problem in the
bounded region $(x_{1}, x_{2}) \in [-L_{0}, L_{0}]\times[-\lt,\lt].$  This also leads to estimates like those
in~\eqref{eqn3.15.10}, provided $|x_1|<L_0.$
\end{remark}

The analytic continuation of a waveguide mode satisfies an estimate of the form $|v(x_2)|\leq Ce^{-c|\Re x_2|},$ for $c, C$ positive constants. Because the exponential decay is in the $\Re x_2$ it is  not in  $\cC^{\alpha,0}$ space for any $\alpha>0$. Nonetheless  restricted to a curve, $\tGamma,$ satisfying the slope condition this data does belong to $\cC^{\alpha,\beta}_{\tGamma}$ for some positive $\alpha$ and $\beta.$
\begin{corollary}\label{cor:mode_soln}
    Suppose~$u(x_1,x_2) = v(x_2)e^{\pm i \xi x_1}$ is a waveguide mode for either~$q^l$ or~$q^r$ and the boundary data~$[f_D,f_N] = [v, \pm i\xi v]$. If~$\tGamma\in\cG$ satisfies the slope conditions in \eqref{eq:slop-cond}, then there is a unique $[\tsigma,\ttau]\in \cC^{1/2,k}\oplus \cC^{1,k}$ such that
    \begin{equation}\label{eq:mode_comp_IE}
        \lp\Id + \cK_{\tGamma} \rp\begin{bmatrix}
            \tsigma+f_D\\\ttau+f_N
        \end{bmatrix}(x_2)
        =
        \begin{bmatrix}
            f_D\\f_N
        \end{bmatrix}(x_2) \quad\forall x_2\in\tGamma.
    \end{equation} 
    Further, the pair of functions 
    \begin{equation}
        u^{l,r}(\bx) = \cS_\bbR^{l,r}[\ttau +f_N](\bx)+\cD_{\bbR}^{l,r}[\tsigma +f_N](\bx)\label{eq:mode_real_rep}
    \end{equation}
    solve~\eqref{eqn:the_reduced_pde};  moreover, for $\bx\in \mathbb{R}^{2} \cap \{|x_{2}|\leq \lt\},$ we can represent the solution by integrating along $\tGamma$
    \begin{equation}
        u^{l,r}(\bx) = \cS_{\tGamma}^{l,r}[\ttau +f_N](\bx)+\cD_{\tGamma}^{l,r}[\tsigma +f_N](\bx).\label{eq:mode_comp_rep}
    \end{equation}

    For any sufficiently small~$\epsilon>0$, there exists an $0<\eta\leq \epsilon$, and a constant $C$ independent of $\epsilon$ such that the following holds. There is a unique $[\tsigma_\eps,\ttau_\epsilon]\in \cC^{1/2,k}\oplus \cC^{1,k}$ such that
        \begin{equation}
        \lp\Id + \cK_{\tGamma_\eta} \rp\begin{bmatrix}
            \tsigma_\epsilon+f_D\\\ttau_\epsilon+f_N
        \end{bmatrix}(x_2)
        =
        \begin{bmatrix}
            f_D\\f_N
        \end{bmatrix}(x_2) \quad\forall x_2\in\tGamma_\eta.
    \end{equation} 
Furthermore,
    \begin{equation}\label{eq:mode_rep_err}
        \left|\lp \cS^{l,r}_{\tGamma_{\eta}}[\ttau_\epsilon+f_N](\bx) {-} \cD^{l,r}_{\tGamma_{\eta}}[\tsigma_\epsilon+f_D](\bx)\rp - u^{l,r}(\bx) \right| <C \epsilon \left( \|f_{D}\|_{\cC^{\frac12,0}_{\bbR}} + \|f_{N}\|_{\cC^{1,0}_{\bbR}} \right)
    \end{equation}
    for all~$\bx\in \bbR^{2} \cap \{|x_{2}|\leq \lt \}$.
\end{corollary}

\section{Mapping Properties of operators}\label{sec:op_prop}\label{sec:path-indep}
In this section we establish several properties of the operators $C_{\tGamma}$ and $D_{\tGamma}.$ In particular, we show that  the operators are independent of the choice of $\tGamma \in \cG$ and are bounded between $\cC^{\alpha,\beta}$-spaces for appropriate choices of $\alpha$ and $\beta$. 

We begin by defining the following convenient class of kernels.
\begin{definition}
    Let~$\alpha>0$ and~$\cA_\alpha$ denote the class of kernels~$\tilde k$ that have the following properties:
    \begin{itemize}
        \item The kernel $\tilde k$ is continuous on~$(x_{2},y_{2}) \in (\GammaC \times \GammaC)\setminus \{x_2=y_2, |x_2|\leq d, |y_2|\leq d \}$.
        \item The kernel $\tilde k$ is weakly singular on the diagonal in~$\Gamma_M \times \Gamma_M,$ that is
          \begin{equation}
            \int_{\GammaM}|k(x_2,y_2)|\dd x_2+ \int_{\GammaM}|k(x_2,y_2)|\dd y_2<\infty.
          \end{equation}
        \item The kernel $\tilde k$ is separately analytic in each variable, $x_2, y_2,$ in the interior of~$\GammaC$.
        \item There exist constants~$C_{\pm,\pm}$ and~$K_{1}$ such that     \begin{equation}
        \label{eq:cc-est}
        \left|\tilde k(x_{2},y_{2}) - C_{\operatorname{sign}(\Re x_2),\operatorname{sign}(\Re y_2)} \frac{e^{ik(\operatorname{sign}(\Re x_2) x_2+ \operatorname{sign}(\Re y_2) y_2)}}{(\operatorname{sign}(\Re x_2) x_2+ \operatorname{sign}(\Re y_2) y_2)^{\alpha}} \right| 
        \leq \frac{K_{1}e^{-k\lp|\Im(x_2)|+ |\Im(y_2)|\rp}}{(1+|x_2|+|y_2|)^{\alpha+\frac 12}}
    \end{equation}
    if~$x_{2},y_{2}\in \Gamma_L\cup\Gamma_R$.
    \item There exists a constant~$K_{2}$ and bounded continuous functions~$f_{\pm}$ such that
    \begin{equation}
    \label{eq:cm-est}
        \left|\tilde k(x_{2},y_{2}) - f_{\operatorname{sign}(\Re x_2)}(y_2) \frac{e^{ik\operatorname{sign}(\Re x_2) x_2}}{(\operatorname{sign}(\Re x_2) x_2)^{\alpha}} \right| 
        \leq \frac{K_{2}e^{-k|\Im(x_2)|}}{(1+|x_2|)^{\alpha+\frac 12}}
    \end{equation}
    if~$x_{2}\in \Gamma_L\cup\Gamma_R$ and~$y_{2}\in\Gamma_M$ and a similar statement holds with~$x_{2}$ and~$y_{2}$ swapped.
    \end{itemize}
\end{definition}
\begin{remark}
    Note that both $\operatorname{sign}(\Re x_2) x_2+ \operatorname{sign}(\Re
    y_2) y_2$ and $\operatorname{sign}(\Re x_2) x_2$ lie in the first quadrant,
    so their $\alpha^{\Th}$-power is unambiguously defined by setting
    $z^{\alpha}=e^{\alpha\log z},$ where $\log z$ is defined in $\bbC\setminus
    (-\infty,0]$ by setting $\log 1=0.$
\end{remark}
The following theorem shows that the kernels~$k_C$ and~$k_D$ of the operators $C$ and $D$ can be analytically continued to~$\GammaC\times \GammaC$ and lie in these kernel classes.
\begin{theorem}\label{thm:kern_decay}
    The kernels~$k_C$ and~$k_D$ can be analytically continued to~$\GammaC\times \GammaC$ and 
    are in~$\cA_{\frac 32}$ and~$\cA_{\frac 12}$, respectively. Further, the kernel $k_{C}(x_2,y_2)$ has a $\log|x_2-y_2|$-singularity on the diagonal inside $[-d,d]\times[-d,d],$ and is continuously differentiable on the rest of~$\GammaC\times\GammaC.$ The kernel $k_D(x_2,y_2)$ has a $|x_2-y_2|^2\log|x_2-y_2|$-singularity in $[-d,d]\times[-d,d],$ and is continuously differentiable on the rest of~$\GammaC\times\GammaC$.
\end{theorem}
\begin{proof}
The weak singularities of $k_{C}$, and $k_{D}$ along the diagonal in~$\GammaM\times\GammaM$ are derived in~\cite{epstein2023solvinga}. Proving that the kernels are analytic in $\GammaC$ and that they satisfy~\eqref{eq:cc-est} and~\eqref{eq:cm-est} is straightforward, but tedious; and is left to Corollary~\ref{cor:analytic-kcd} and Lemma~\ref{lem:asymp-est-kcd} in Appendix~\ref{sec:kern_est}. 
\end{proof}

\begin{remark}
    The remainder of the results in this section generalize to any pair of kernels belonging to~$\mathcal{A}_{\frac 12}$ and $\mathcal{A}_{\frac 32}.$ To obtain analogs of the main results, all that is required is a uniqueness result for~\eqref{eq:real_IE} on the real line.
\end{remark}

We begin by verifying that the complexified layer potentials exist when~$\sigma$ and~$\tau$ are in the correct spaces.
\begin{lemma}\label{lem1.102}
    For $0<\alpha<\frac 12$ and $\beta>-k$, if~$\sigma\in \cC^{\alpha,\beta}$ and~$\tau\in \cC^{\alpha+\frac 12,\beta}$, then for any~$x_2 \in \GammaC$ and any~$\tGamma\in \cG$, the values of~$C_{\tGamma}[\sigma](x_2)$ and~$D_{\tGamma}[\tau](x_2)$ are given by convergent integrals.    
\end{lemma}
\begin{proof}
    All of the kernels are locally integrable, so it is enough to verify that the integrand decays sufficiently rapidly at infinity.
    Since~$k_C\in \cA_{\frac 32}$ and~$k_D\in \cA_{\frac 12}$, the assumptions on~$\sigma$ and~$\tau$ imply that the integrands of~$C[\sigma](x_2)$ 
    and~$D[\tau](x_2)$ decay as~$O\lp|y_2|^{-\frac32-\alpha} e^{-(\beta+k)|\Im y_2|}\rp$, and~$O\lp|y_2|^{-1-\alpha}e^{-(\beta+k)|\Im y_2|}\rp$ respectively. Since~$\beta>-k$, these are both absolutely and uniformly integrable.
\end{proof}

We now prove that $C_{\tGamma}$ and $D_{\tGamma}$ do not depend on the curve
used to define them, which justifies denoting them simply by $C$ and $D.$

\begin{proposition}\label{prop:path_indep}
    Let~$\sigma\in \cC^{\alpha,\beta}$ and~$\tau\in \cC^{\alpha+\frac 12,\beta}$. For any~$x_2 \in \GammaC$, $C_{\tGamma}[\sigma](x_2)$ and~$ D_{\tGamma}[\tau](x_2)$ are independent of the choice~$\tGamma\in \cG$.
\end{proposition}
\begin{proof}
Let~$\tGamma$ and~$\hGamma$ be two curves in~$\cG$ and~$x_2$ be any fixed point in~$\GammaC$. Let~$\tGamma_\rho$ and~$\hGamma_\rho$ be the truncations of the contours~$\tGamma$ and~$\hGamma$ at~$|\Re y_2|=\rho$. Let~$\gamma^{\pm}_\rho$ be the vertical lines connecting these truncations (see~\figref{fig:path_indep}). From the estimates in the proof of Lemma~\ref{lem1.102}, it follows that
\begin{equation}
    C_{\tGamma}[\sigma](x_2) = \lim_{\rho\to\infty}C_{\tGamma_\rho}[\sigma](x_2),
\end{equation}
with a similar statement for~$\hGamma$. It is therefore enough to study the difference between~$C_{\tGamma_\rho}$ and~$C_{\hGamma_\rho}$ and show that it vanishes as~$\rho\to\infty$.

By Lemma~\ref{lem:kern_analytic} and the assumptions on $(\sigma,\tau)$,  both integrands are analytic in the region between~$\tGamma$ and~$\hGamma$. We therefore can apply Cauchy's theorem to conclude that
\begin{equation}
    C_{\tGamma_\rho}[\sigma](x_2) - C_{\hGamma_\rho}[\sigma](x_2)=C_{\gamma_\rho}[\sigma](x_2)=\int_{\gamma^+_\rho\cup\gamma^-_\rho} k_C(x_2,y_2)\sigma(y_2) {\rm d}y_2,
\end{equation}
where~$\gamma^\pm_\rho$ are oriented so that they start at~$\hGamma_\rho$ and end at~$\tGamma_\rho$. We first consider $\gamma^+_\rho;$ since $-k<\beta,$ the integral can be bounded by
\begin{equation}
    |C_{\gamma^+_\rho}[\sigma](x_2)|\leq K \frac{\|\sigma\|_{\alpha,\beta}}{(\min_{\gamma^+_\rho}|y_2|)^\alpha} \int_{0}^{\infty}e^{-(\beta+k)t}{\rm d}t \leq \frac{K}{\beta+k} \frac{\|\sigma\|_{\alpha,\beta}}{\rho^\alpha}.
\end{equation}
As~$\alpha>0$, this bound implies that~$|C_{\gamma^+_\rho}[\sigma](x_2)|\to 0$ as~$\rho\to \infty$. A similar argument shows that ~$|C_{\gamma^-_\rho}[\sigma](x_2)|\to 0$ as~$\rho\to \infty$ and thus have shown that~$C_{\tGamma}[\sigma](x_2) = C_{\hGamma}[\sigma](x_2)$ for any~$x_2\in \GammaC$ and any pair of curves~$\tGamma,\hGamma\in\cG$.

An essentially identical proof shows that~$D_{\tGamma}[\tau](x_2)$ does not depend on~$\tGamma$.
\end{proof}
\begin{figure}
    \centering
    \includegraphics[width = 0.6\textwidth]{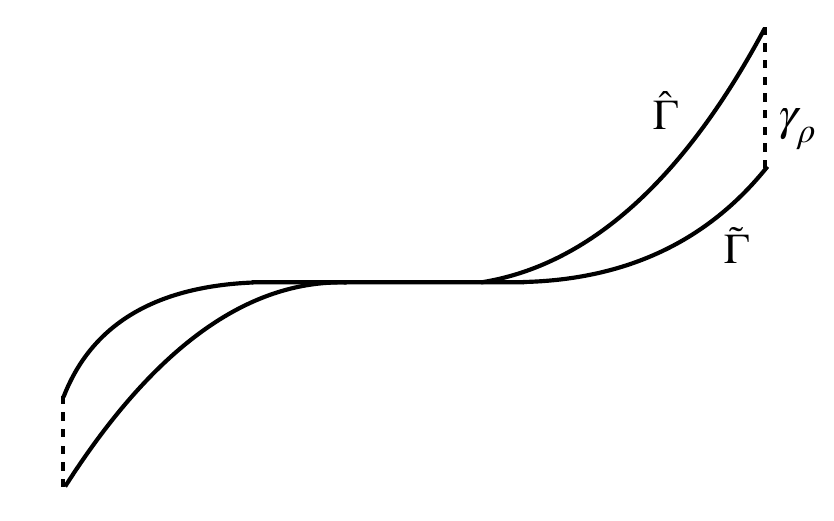}
    \caption{Contours for Proposition~\ref{prop:path_indep}}
    \label{fig:path_indep}
\end{figure}

Since $C$ and $D$ are independent of the choice of~$\tGamma$,  in the sequel we suppress the dependence on the contour of integration $\tGamma$, except to emphasize quantities that explicitly depend on it. We now turn to proving that~$C$ and~$D$ are bounded on~$\cC^{\alpha,\beta}$. To do this, we need the following integral inequality; it is essentially Lemma 1 in~\cite{epstein2023solvinga}.
\begin{lemma}\label{lem:basic_integral}
    If~$0\leq a< 1$ and~$a+b>1$, then 
    \begin{equation}
        I=\int_0^\infty \frac{1}{\lp |x_2|+t\rp^b}\frac{1}{(1+t)^{a}} {\rm d}t \leq  K |x_2|^{1-a-b} \,,
    \end{equation}
    where $K$ depends on $a$ and $b$ but not on $|x_{2}|$.
\end{lemma}
\begin{proof}
 We make the substitution~$t=|x_2|s$ to get
    \begin{equation}
        I = \int_0^\infty \frac{1}{|x_2|^{b}(1+s)^{b}}\frac{1}{(1+|x_2|s)^a}
        |x_2|ds\leq \frac{1}{|x_2|^{a+b-1}}\int_0^\infty \frac{1}{(1+s)^{b}}\frac{1}{s^a}ds.
    \end{equation}
    Since  $0<a<1,$ and $1<a+b$ it is clear that this integral is finite.
\end{proof} 

In order to establish the decay of $C[\sigma]$ at infinity, we need to handle the contribution from inside the channel
carefully owing to the logarithmic singularity of the kernel there~\cite{epstein2023solvinga}.
In what follows it is helpful to introduce a partition of unity to isolate this singularity. Recall that $\lt>d.$
\begin{definition}\label{def7.102}
    For $\epsilon = (\lt-d)/3$, let~$\psi \in C^\infty$ be compactly supported on~$[-\lt+\epsilon,\lt-\epsilon]$, identically one on~$[-d-\epsilon,d+\epsilon]$, and $0\leq \psi(y_{2}) \leq 1$ for all $y_{2} \in \bbR$; set
    \begin{equation}
        k_{C_{0}}(x_2,y_2) = \psi(y_2) k_{C}(x_2,y_2) \quad\text{and}\quad k_{C_{1}}= k_{C}-k_{C_{0}}.
    \end{equation}
    Let~$C_0$ and~$C_1$ be the integral operators with kernels~$k_{C_{0}}$ and~$k_{C_{1}}$ respectively.
\end{definition}

The following lemma follows from Theorem~\ref{thm:kern_decay}.
\begin{lemma}
    The kernel~$k_{C_{0}}$ is compactly supported in~$y_2$ with a logarithmic singularity on the diagonal. 
    \begin{equation}
        |k_{C_{0}}(x_2,y_2)| \leq K
        \psi(y_2)\frac{[1+\psi(x_2)|\log|x_2-y_2||]e^{-k|\Im
            x_2|}}{(1+|x_2|)^{\frac 32}}.
    \end{equation}
\end{lemma}

Using this split of the kernel of $C$, and Lemma~\ref{lem:basic_integral}, we establish the boundedness of the operators~$\sigma\mapsto C[\sigma]$ and~$\tau\mapsto D[\tau]$ between appropriate $\cC^{\alpha,\beta}$ spaces in the following three lemmas.
We begin by showing that~$C[\sigma]$ and~$D[\tau]$ have the required continuity and analyticity.
\begin{lemma}
    If~$\sigma\in \cC^{\alpha,\beta}$ and~$\tau\in \cC^{\alpha+\frac 12,\beta}$, then~$C[\sigma]$ and~$D[\tau]$ are continuous on~$\GammaC$.
\end{lemma}
\begin{proof}
    Since~$k_{C_{1}}(x_2,y_2)\sigma(y_2)$ and~$k_D(x_2,y_2)\tau(y_2)$ are continuous and can be bounded independent of~$x_2$, the functions~$C_1[\sigma]$ and~$D[\tau]$ are clearly continuous by the dominated convergence theorem. 
    The kernel~$k_{C_{0}}$ is weakly singular and supported on~$|y_2|<\lt$. Standard arguments for compact domains (e.g. Theorem 1.11 in~\cite{colton2013integral}) therefore give that~$C_0[\sigma]$ is continuous.
\end{proof}
\begin{lemma}
    If~$\sigma\in \cC^{\alpha,\beta}$ and~$\tau\in \cC^{\alpha+\frac 12,\beta}$, then~$C[\sigma](x_2)$ and~$D[\tau](x_2)$ are analytic in the interior of~$\GammaC$.
\end{lemma}
\begin{proof}
    The proof follows from Morera's theorem. Let~$\gamma$ be a closed contour in the first quadrant. Since,~$k_C\in \cA_{\frac 32},\sigma\in \cC^{\alpha,\beta},$ and~$\gamma$ is bounded we can use Fubini's theorem to see that
    \begin{equation}
        \int_\gamma C[\sigma](x_2){\rm d}x_2 = \int_{\tGamma} \int_\gamma k_C(x_2,y_2) \sigma(y_2)  {\rm d}x_2 {\rm d}y_2.
    \end{equation}
    The inner integral is zero because~$k_C(x_2,y_2)$ is an analytic function of~$x_2$ and so we have that~$C[\sigma]$ is analytic. An identical proof can be used to show that~$D[\tau]$ is also analytic.
\end{proof}
\begin{remark}\label{rmk3.102}
    From the proof of this theorem it is clear that it is not
    necessary for $\sigma$ or $\tau$ to be analytic. Given a curve
    $\tGamma\in\cG,$ all that is required for
    $C_{\tGamma}[\sigma](x_2),$ $D_{\tGamma}[\tau](x_2)$ to have
    analytic continuations to $\GammaL\cup\GammaR$ is for the
    integrals defining $C_{\tGamma}[\sigma](x_2)$
    and $D_{\tGamma}[\tau](x_2)$ to be absolutely convergent,
    uniformly for $x_2$ in compact subsets of $\Int
    \GammaL\cup\GammaR.$ This property holds for $\sigma\in
    \cC^{\alpha,\beta}_{\tGamma}$ and $\tau\in \cC^{\alpha+\frac
      12,\beta}_{\tGamma}$ respectively, but one cannot generally
    change the contour of integration in the $y_2$-integral, unless
    the data have analytic continuations from $\tGamma.$
\end{remark}
Finally, we  prove the norm estimates for the operators $C$ and $D$ in the lemma below.
\begin{lemma}\label{lem:op_bd}
    Suppose that~$\sigma\in \cC^{\alpha,\beta}$ and $\tau\in \cC^{\alpha+\frac 12,\beta},$ with $0<\alpha<1/2$ and $\beta>-k.$ Then there exists a~$K>0$ such that
    \begin{equation}
        \|C[\sigma]\|_{\alpha+\frac 12,k} \leq K\|\sigma\|_{\alpha,\beta}\label{eq:C_bound}
    \end{equation}
    and
    \begin{equation}
        \|D[\tau]\|_{\alpha,k} \leq K\|\tau\|_{\alpha+\frac 12,\beta}.\label{eq:D_bound}
    \end{equation}
\end{lemma}
\begin{proof}
    We  prove this bound using the choice~$\tGamma=\bbR$. We begin by proving the bound for
    \begin{equation}
        C_1[\sigma](x_2)=\int_{-\infty}^\infty k_{C_{1}}(x_2,t)\sigma(t) {\rm d}t.
    \end{equation}
    For ~$t\in\bbR$,
        $|\sigma(t)| \leq \|\sigma\|_{\alpha,\beta}/(1+|t|)^\alpha$.
    Using this bound, and Theorem~\ref{thm:kern_decay}, we have that
    \begin{equation}
        |C_1[\sigma](x_2)|\leq  \int_0^\infty \frac{Ke^{-k|\Im x_2|}}{(1+|x_2|+t)^{\frac 32}} \frac{ \|\sigma\|_{\alpha,\beta}}{(1+t)^\alpha} {\rm d}t.\label{eq:c1_bd}
    \end{equation}
    If we factor out the terms independent of~$t,$ then Lemma~\ref{lem:basic_integral} with~$a=\alpha$ and~$b=\frac32$ gives
    \begin{equation}
        \left|C_1[\sigma](x_2)\right| \leq \frac{K'e^{-k \Im x_2}}{(1+|x_2|)^{\frac12+\alpha}}\|\sigma\|_{\alpha,\beta}.\label{eq:c1_decay}
    \end{equation}
    The  bound on~$|C_1[\sigma](x_2)|$ for $|x_2|<1$ is obvious, and therefore we see that there is a constant $C$ so that
    $\|C_1[\sigma]\|_{\alpha+\frac 12,k}\leq C\|\sigma\|_{\alpha,\beta}$.
    
    To bound the contribution of~$C_0$, we note that
    \begin{equation}
        |C_0[\sigma](x_2)|\leq \int_{-{L_C}}^{{L_C}}   \frac{2K|\log|x_2-t||e^{-k|\Im x_2|}}{(1+|x_2|)^{\frac 32}} \|\sigma\|_{\alpha,\beta} {\rm d}t.
    \end{equation}
    Since the integral is over a finite domain, we can bound the integral of the logarithm by~$C\sqrt{|x_2|}$ to see
    \begin{equation}
        |C_0[\sigma](x_2)|\leq \frac{K'e^{-k|\Im x_2|}}{1+|x_2|}\|\sigma\|_{\alpha,{\beta}},
    \end{equation}
 which gives~\eqref{eq:C_bound}.

    A similar argument to the one above gives
    \begin{equation}
        \left|D[\tau](x_2) \right|\leq \int_0^\infty \frac{Ke^{-k|\Im x_2|}}{(1+|x_2|+t)^{\frac 12}} \frac{ \|\tau\|_{\alpha+\frac12,\beta}}{(1+t)^{\alpha+\frac12}} {\rm d}t.
    \end{equation}
    Applying Lemma~\ref{lem:basic_integral} with~$a=\alpha+\frac12$ and~$b=\frac12$ gives~\eqref{eq:D_bound}.
\end{proof}

\begin{remark}\label{rmk4.103}
    By carrying out the estimate on another choice of contour $\tGamma\in\cG$, it is clear that the conclusion of Lemma~\ref{lem:op_bd} actually holds for any such curve and functions $\sigma_{\tGamma}\in \cC^{\alpha,\beta}_{\tGamma}$ and $\tau_{\tGamma}\in \cC^{\alpha+\frac 12,\beta}_{\tGamma},$ in that there exists a~$K>0$ such that,
    \begin{equation}
        \|C_{\tGamma}[\sigma_{\tGamma}]\|_{\alpha+\frac 12,k} \leq K\|\sigma_{\tGamma}\|_{\alpha,\beta,\tGamma}\label{eq:C_bound_restr}
    \end{equation}
    and
    \begin{equation}
        \|D_{\tGamma}[\tau_{\tGamma}]\|_{\alpha,k} \leq K\|\tau_{\tGamma}\|_{\alpha+\frac 12,\beta,\tGamma}.\label{eq:D_bound_restr}
    \end{equation}
    Here $C_{\tGamma}[\sigma_{\tGamma}],\, D_{\tGamma}[\tau_{\tGamma}]$ are the analytic continuations described in Remark~\ref{rmk3.102}.
\end{remark}
Combining the above lemmas gives the following theorem which proves the boundedness of $C$ and $D$ between
appropriate $\cC^{\alpha,\beta}$ spaces.
\begin{theorem}\label{thm:op_maps}
    For any~$0<\alpha<\frac12$ and~$-k<\beta$,  the following maps are bounded
    \begin{equation}
        \begin{matrix}
            C: \cC^{\alpha,\beta} \to \cC^{\alpha+\frac12,k} & \text{and}& D: \cC^{\alpha+\frac 12,\beta} \to \cC^{\alpha,k}.
     \end{matrix}
    \end{equation}
\end{theorem}
Since~$\cC^{\alpha',k}$ can be continuously embedded into~$\cC^{\alpha',\beta}$ for any~$\alpha'>0$ and~$\beta\leq k$, this theorem also implies the weaker statement that 
    \begin{equation}
        \begin{matrix}
            C: \cC^{\alpha,\beta} \to \cC^{\alpha+\frac12,\beta} & \text{and}& D: \cC^{\alpha+\frac 12,\beta} \to \cC^{\alpha,\beta}
     \end{matrix}
    \end{equation}
    are bounded operators for~$0<\alpha<\frac12$ and~$-k<\beta\leq k$.
 

\section{Analysis of the integral equation system}\label{sec:IE_prop}
We now turn our attention to the analysis of~\eqref{eq:comp_IE}. We begin by proving uniqueness, and then show that the operator $\Id+\cK$ is a Fredholm operator of index zero. This combined with uniqueness implies that~\eqref{eq:comp_IE} is solvable for arbitrary data in $\cC^{\alpha,\beta}\oplus\cC^{\alpha+\frac 12,\beta},$ with $0<\alpha<\frac 12, -k<\beta\leq k.$ 

\subsection{Uniqueness}
The uniqueness of the solution to~\eqref{eq:comp_IE} is essentially immediate from the uniqueness result for $(\Id+\cK_{\bbR})$ proved in~\cite{epstein2024solving}, and analyticity. 
\begin{lemma}[Uniqueness]
\label{lem:uniqueness} For $0<\alpha<\frac 12,\,-k<\beta\leq k,$
    if~$[f_D,f_N]\in \cC^{\alpha,\beta}\oplus \cC^{\alpha+\frac12,\beta}$, then there is at most one solution~$[\sigma,\tau]\in \cC^{\alpha,\beta}\oplus \cC^{\alpha+\frac12,\beta}$ of~\eqref{eq:comp_IE}. 
\end{lemma}
\begin{proof}
    Since the equation is linear, we need to show that the only solution $[\sigma,\tau]\in \cC^{\alpha,\beta}\oplus \cC^{\alpha+\frac12,\beta}$ to 
    \begin{equation}\label{eqn.null_sp}
    (\Id+\cK)\left[\begin{matrix}\sigma\\ \tau\end{matrix}\right]=\left[\begin{matrix}0\\ 0\end{matrix}\right]\end{equation}
    is $\sigma=\tau=0.$ From Proposition~\ref{prop:path_indep} it follows that if $[\sigma,\tau]$ solves~\eqref{eqn.null_sp}, then $(\Id+\cK_{\bbR})\left[\begin{matrix}\sigma\\ \tau\end{matrix}\right](x_2)=\left[\begin{matrix}0\\ 0\end{matrix}\right],$ for $x_2\in\bbR.$ 
    
    The restriction $[\sigma,\tau]_{\bbR}\in \cC^{\alpha}(\bbR)\oplus \cC^{\alpha+\frac12}(\bbR)$ and therefore Theorem 3 in~\cite{epstein2024solving} shows that $[\sigma,\tau]_{\bbR}=[0,0].$ The proof then follows from~\eqref{eqn:eval_sigtau} with $\tGamma = \bbR$. 
\end{proof}

This result has the following useful corollary
\begin{corollary}\label{cor5.1.102}
     If, for any~$\tGamma\in \cG,$ there is a
     pair~$[\sigma_{\tGamma},\tau_{\tGamma}]\in
     \cC^{\alpha,\beta}_{\tGamma}\oplus \cC^{\alpha+\frac12,\beta}_{\tGamma}$
     such that
     \begin{equation}\label{eqn5.2.111}
       (\Id+\cK_{\tGamma})\begin{bmatrix} \sigma_{\tGamma}
       \\ \tau_{\tGamma}\end{bmatrix}(x_2)=0\text{ for
       }x_2\in\tGamma,
       \end{equation}
       then
        $[\sigma_{\tGamma},\tau_{\tGamma}]$ has an
       analytic continuation to an element $[\sigma,\tau]\in
       \cC^{\alpha,\beta}\oplus \cC^{\alpha+\frac12,\beta}$ and
       $[\sigma,\tau]\equiv [0,0].$
\end{corollary}
\begin{proof} Let $[\sigma_{\tGamma},\tau_{\tGamma}]$ be a solution to~\eqref{eqn5.2.111}.
    As noted in Remark~\ref{rmk4.103}, the components of
    $\cK_{\tGamma}\begin{bmatrix} \sigma_{\tGamma}
      \\ \tau_{\tGamma}\end{bmatrix}$ extend analytically from $\tGamma$ to
    $\GammaL\cup\GammaR,$ to define an element of $\cC^{\alpha,\beta}\oplus
    \cC^{\alpha+\frac12,\beta}.$ Since
    \begin{equation}
    \begin{bmatrix} \sigma_{\tGamma}
       \\ \tau_{\tGamma}\end{bmatrix}(x_2)=-  \cK_{\tGamma}\begin{bmatrix} \sigma_{\tGamma}
       \\ \tau_{\tGamma}\end{bmatrix}(x_2),
    \end{equation}
for $x_2\in\tGamma,$ it follows that $[\sigma_{\tGamma},\tau_{\tGamma}]$ also has an
analytic continuation belonging to $\cC^{\alpha,\beta}\oplus
    \cC^{\alpha+\frac12,\beta}.$ This pair satisfies
\begin{equation}
  (\Id+\cK)\begin{bmatrix} \sigma
       \\ \tau\end{bmatrix}(x_2)=0\text{ for }x_2\in\GammaC.
\end{equation}
    Lemma~\ref{lem:uniqueness} implies that $[\sigma,\tau]\equiv [0,0].$
\end{proof}

\subsection{The Fredholm Property and Existence of Solutions}
We now show that the operator defining  the integral equation~\eqref{eq:comp_IE} is Fredholm of index 0. The argument is nearly identical to the discussion of the Fredholm property of~\eqref{eq:real_IE} in~\cite{epstein2023solvinga}, but is given here for completeness.

For integral equations on bounded domains, compactness is usually proved by noting that the operators are smoothing and using the Arzela-Ascoli theorem. Since the Arzela-Ascoli theorem is only applicable on bounded domains, this proof cannot be directly generalized to our situation. Instead, we must prove that the integral operator of interest, in addition to smoothing, maps~$\cC^{\alpha,\beta}$ into a space of functions with faster decay. 

\begin{proposition}\label{prop:compact}
    Suppose~$0 < \alpha < \alpha'$ and $\beta \in \mathbb{R}$. Let~$A:\cC^{\alpha,\beta}
    \to \cC^{\alpha',\beta}$ be a bounded linear operator such that~$\partial_{x_2}A:\cC^{\alpha,\beta}
    \to \cC^{0,\beta}$ boundedly. Then $A:\cC^{\alpha,\beta}
    \to\cC^{\alpha,\beta}$ is a compact operator.
\end{proposition}
\begin{proof}
    Let $\{f_n\}$ be a bounded sequence in $\cC^{\alpha, \beta}$. Since $\{A f_n\}$ is bounded in $\cC^{\alpha', \beta}$ and $\alpha < \alpha'$, for any $\eps > 0$ there exists a bounded subset~$B_\eps$ of~$\GammaC$ such that 
    \begin{equation}\label{eqn5.2.110}
    \|A f_n\|_{\cC^{\alpha,\beta} (\Gamma_\mathbb{C} \setminus B_\eps)} < \eps
    \end{equation}
    for all $n$. Since~$Af_n$ and~$(\partial_{x_2} A f_n)$ are uniformly bounded
    in~$\cC^{0,\beta}$, the set~$\{Af_n\}$ is uniformly bounded and uniformly
    equi-continuous on~$B_\epsilon\cap \GammaC$. By the Arzela-Ascoli theorem,
    there exists a Cauchy subsequence of~$Af_n$ in~$\cC^0(B_\epsilon\cap
    \GammaC)$, and so also in~$\cC^{\alpha,\beta} (\Gamma_\mathbb{C} \cap
    B_\eps)$. By repeating this for $\epsilon\in\{\frac 1m:\:m\in\bbN\}$ we can
    use~\eqref{eqn5.2.110} in a standard diagonalization argument to construct a
    subsequence of~$\{Af_n\}$ that converges in $\cC^{\alpha,\beta}.$ As this
    holds for all bounded sequences, it follows that $A:\cC^{\alpha,\beta}
    \to\cC^{\alpha,\beta}$ is a compact operator.
\end{proof}

Unfortunately, this proposition cannot be applied directly to~$\cK$, since $D$ does not increase the algebraic rate of decay. Instead, we follow~\cite{epstein2023solvinga} and observe that
\begin{equation}
   \begin{bmatrix}
     \Id  & -D \\ 0 & \Id
    \end{bmatrix}   \begin{bmatrix}  \Id  & D \\ C & \Id
    \end{bmatrix}   \begin{bmatrix} \Id  & 0 \\ -C & \Id
    \end{bmatrix}=  \begin{bmatrix}
        \Id -DC & 0 \\ 0 & \Id
    \end{bmatrix}.\label{eq:precond_IE}
\end{equation}
Since the left and right factors on the left-hand side~of~\eqref{eq:precond_IE} are bounded and invertible,
equation~\eqref{eq:comp_IE} is Fredholm of index 0 if and only
if  the right hand side of~\eqref{eq:precond_IE} is. In what follows, we prove that~$DC$ is a compact
map from~$\cC^{\alpha,\beta}$ to itself. The advantage of using the composition
is that the oscillations of the kernels cause the kernel of the composite
map~$DC$ to decay faster than is predicted by Theorem~\ref{thm:op_maps} and
allow us to use Proposition~\ref{prop:compact}.

To aid in the proof, we split the kernels into far field and near field components. Specifically, let~$\psi$ be the bump function specified in Definition~\ref{def7.102} and set
\begin{equation}
    k_{C_{f}}(x_2,y_2) = (1-\psi(x_2))(1-\psi(y_2))k_C(x_2,y_2) \qand k_{C_{n}}(x_2,y_2)=k_{C}(x_2,y_2)-k_{C_{f}}(x_2,y_2).
\end{equation}
we define~$k_{D_{n}}$ and~$k_{D_{f}}$ similarly. This splitting is ensures that~$k_{C_{f}}$ and~$k_{D_{f}}$ are zero if either~$x_2$ or~$y_2$ is in the waveguide, which gives them a simple asymptotic form. The operators with kernels~$k_{C_{n}},k_{C_{f}},k_{D_{n}},$ and~$k_{D_{f}}$ are denoted~$C_n,C_f,D_n$ and~$D_f$, respectively. Using these definitions, we can split the product~$DC$:
\begin{equation}
    DC = D_nC_n + D_fC_n+D_nC_f+D_fC_f\label{eq:split_DC}
\end{equation}
The proof of Theorem~\ref{thm:op_maps} shows that all of these are bounded from $\cC^{\alpha,\beta}$ to itself.

We now prove that $C_n$ and~$D_n$ are compact.
\begin{lemma}\label{lem:n_compact}
    The operator $\begin{bmatrix}
        0&D_n\\C_n&0
    \end{bmatrix}$ is compact from~$\cC^{\alpha,\beta}\oplus\cC^{\alpha+\frac12,\beta}$ to $\cC^{\alpha,\beta}\oplus\cC^{\alpha+\frac 12,\beta}$
\end{lemma}
\begin{proof}
We expand the definition of~$k_{C_{n}}:$
\begin{equation}
    k_{C_{n}}(x_2,y_2) = \psi(x_2)k_{C}(x_2,y_2) + \psi(y_2)k_{C}(x_2,y_2) - \psi(x_2)\psi(y_2)k_{C}(x_2,y_2).
\end{equation}
The first and third terms define smoothing integral operators, mapping into functions which are compactly supported on~$\Gamma_M.$ They therefore trivially satisfy the assumptions of Proposition~\ref{prop:compact} and so are compact from~$\cC^{\alpha,\beta}$ to~$\cC^{\alpha+\frac12,\beta}$. 

For the second term, we let
\begin{equation}
    \tilde{k}_{C_n}(x_2,y_2) = (1-\psi(x_2))\frac{e^{ik \operatorname{sign}(\Re x_2) x_2}}{(\operatorname{sign}(\Re x_2) x_2)^{3/2}}f_{C,\operatorname{sign}(\Re x_2)}(y_2)\psi(y_2)
\end{equation}
be the asymptotic form of~$k_{C}(x_2,y_2)\psi(y_2)$ identified by Theorem~\ref{thm:kern_decay}. Theorem~\ref{thm:kern_decay} guarantees that
\begin{equation}
    \left| k_{C}(x_2,y_2)\psi(y_2) -\tilde{k}_{C_n}(x_2,y_2)\right|\leq \frac{Ke^{-k|\Im x_2|}}{|x_2|^{2}}\psi(y_2).
\end{equation}
Since~$k_{C}(x_2,y_2)\psi(y_2)-\tilde k_{C_n}(x_2,y_2)$ decays more rapidly than~$k_C$ and has decaying derivatives, it defines a compact operator from~$\cC^{\alpha,\beta}$ to~$\cC^{\alpha+1/2,\beta}$ by Proposition~\ref{prop:compact} and a proof similar to that of Lemma~\ref{lem:op_bd}. The operator with kernel~$\tilde k_{C_n}(x_2,y_2)$, being rank one and bounded, is also a compact operator. Moreover, it trivially maps functions in $\cC^{\alpha, \beta}$ to $\cC^{\alpha+1/2, \beta}$, for any $\alpha \leq1$, and $\beta \leq k$. The operator defined by~$k_C(x_2,y_2)\psi(y_2)$ is thus the sum of a compact operator and a rank one operator and so compact from~$\cC^{\alpha,\beta}$ to~$\cC^{\alpha+1/2,\beta}$.

The operator~$D_n$ can be seen to be compact using a nearly identical proof.
\end{proof}

Since the  operators $C_f,C_n,D_f,D_n$ are bounded, these lemmas show that all terms in~\eqref{eq:split_DC} are compact, but for $D_fC_f,$ which we now consider.
\begin{lemma} \label{lem:f_compact}
    The operator $D_fC_f$ is compact from~$\cC^{\alpha,\beta}$ to itself.
\end{lemma}
\begin{proof}
We showed previously  that, for $|x_2|>d$ and $|y_2|>d$ the kernels~$k_{D}$ and~$k_{C}$ have leading order asymptotics:
\begin{equation}
    \tilde{k}_{D}(x_2,y_2)=\frac{C_{\pm,\pm}e^{ ik( \pm x_2\pm y_2)}}{(\pm x_2+\pm y_2)^{\frac 12}}\qand \tilde{k}_{C}(x_2,y_2)=\frac{C'_{\pm,\pm}e^{ik( \pm x_2\pm y_2)}}{(\pm x_2+\pm y_2)^{\frac 32}}.
\end{equation}
where the choice of $\pm$ depends on the signs of~$\Re x_2$ and~$\Re y_2$. In analogy with above, we can define the far field components of $\tilde{k}_{D}$ and $\tilde{k}_{C},$ which we denote by $\tilde{k}_{D_f}$ and $\tilde{k}_{C_f},$ respectively.
By Theorem~\ref{thm:kern_decay}, 
    \begin{equation}
        \left|k_{D_{f}}(x_2,y_2) - \tilde{k}_{D_{f}}(x_2,y_2) \right| \leq \frac{Ke^{-k(|\Im x_2|+ |\Im y_2|)}}{|x_2|+|y_2|}. \label{eq:fast-decay}
    \end{equation}
 We have a similar expression for~$k_{C_{f}}$, except that the denominator is~$(|x_2|+|y_2|)^2$. Since~$k_{D_{f}}-\tilde{k}_{D_{f}}$ and~$k_{C_{f}}-\tilde{k}_{C_{f}}$ decay more rapidly and have decaying derivatives, a proof similar to the proofs of Lemma~\ref{lem:op_bd} and Proposition~\ref{prop:compact} shows that they define compact operators.

We now study the behavior of~$\tilde{D}_f\tilde{C}_f$, the product of the operators defined by the asymptotic kernels. We denote the kernel of this operator by~$\tilde k$. For notational convenience, in the following we set $\psi^c(z) = 1-\psi(z).$ To bound $\tilde{k}$, we suppose that~$x_2,y_2\in\Gamma_R$ and integrate over the real line
    \begin{equation}
        \tilde k(x_2,y_2) = \int_{-\infty}^\infty \tilde{k}_{D_{f}}(x_2,z_2)\tilde{k}_{C_{f}}(z_2,y_2) \dd z_2.
    \end{equation}
 We split this integral into the regions~$\{z_2>d\}$ and~$\{z_2<-d\},$ and first consider the integral over $[d,\infty),$
    \begin{equation}
        \tilde k_+(x_2,y_2) = C \int_d^\infty \frac{e^{ik (x_2+ z_2)}}{(x_2+z_2)^{\frac 12}}\frac{\psi^c(z_2)e^{ik (z_2+ y_2)}}{(z_2+y_2)^{\frac 32}}\dd z_2 = C e^{ik(x_2+y_2)}\int_{d}^\infty \frac{\psi^c(z_2)e^{2ik z_2}}{(x_2+z_2)^{\frac 12}(z_2+y_2)^{\frac 32}}\dd z_2.
    \end{equation}
  We need to estimate the integral
    \begin{equation}
        I=\int_{d}^\infty \frac{\psi^c(z_2)e^{2ik z_2}}{(x_2+z_2)^{\frac 12}(z_2+y_2)^{\frac 32}}dz_2.
    \end{equation}
    If we integrate by parts twice, we find that
    \begin{multline}
        I = -\frac{1}{4k^2}\int_{d}^\infty \psi^c(z_2)e^{2ik z_2}\left[\frac{15}{4(x_2+z_2)^{\frac 12}(z_2+y_2)^{7/2}} \right.\\
        +\left.\frac{3}{2(x_2+z_2)^{\frac 32}(z_2+y_2)^{\frac 52}}+\frac{3}{4(x_2+z_2)^{\frac 52}(z_2+y_2)^{\frac 32}} \right]\dd z_2 \\
        +\frac{1}{4k^2} \int_d^\infty {\psi^c}'(z_2) e^{2ik z_2}\left[\frac{3x_2 + y_2 + 4z_2}{2(x_2+z_2)^{\frac{3}{2}}(z_2+y_2)^{\frac{5}{2}}}\right]\dd z_2\\
        -\frac{1}{4k^2}\int_d^\infty {\psi^c}''(z_2) e^{2ik z_2}\frac{1}{{(x_2+z_2)^{\frac 12}(z_2+y_2)^{\frac 32}}}\dd z_2
    \end{multline}
   Observing that ${\psi^c}'$ and ${\psi^c}''$ are compactly supported in $[-\lt,-d]\cup [d,\lt],$ each of these terms is easily  seen to be bounded in absolute value by~$\frac{C}{(|x_2+d|)^{\frac 12}(|y_2+d|)^{\frac 32}},$ for some appropriately chosen constant $C.$
    Plugging this into our expression for~$\tilde k$ gives
    \begin{equation}
        \left|\tilde k_+(x_2,y_2)\right| \leq\frac{Ce^{-k\Im (x_2+y_2)}}{(|x_2+d|)^{\frac 12}(|y_2+d|)^{\frac 32}}.
    \end{equation}
     A similar proof can be applied to bound the integral over the region~$\{z_2<-d\}$. 
     
     Repeating the derivation for~$x_2$ and~$y_2$ in different regions of $\GammaL\cup\GammaR$ then gives
    \begin{equation}
        \left|\tilde k(x_2,y_2)\right| \leq\frac{Ce^{-k(|\Im x_2|+|\Im y_2|)}}{(|\pm x_2+d|)^{\frac 12}(|\pm y_2+d|)^{\frac 32}}.
    \end{equation}
    A similar calculation shows that, 
        \begin{equation}
        \left|\partial_{x_2}\tilde k(x_2,y_2)\right| \leq\frac{Ce^{-k(|\Im x_2|+|\Im y_2|)}}{(|\pm x_2+d|)^{\frac 32}(|\pm y_2+d|)^{\frac 32}}.
    \end{equation}    
    If we repeat the proof from Section~\ref{sec:path-indep}, then these estimates are enough to prove that
    \begin{equation}
        \tilde D_{f}\tilde C_f :\cC^{\alpha,\beta} \to \cC^{\frac 12,\beta}  \qand \partial_{x_2}\tilde D_{f}\tilde C_f :\cC^{\alpha,\beta} \to \cC^{\frac 32,\beta} .
    \end{equation}
    By Proposition~\ref{prop:compact} we thus have that~$\tilde D_{f}\tilde C_f$ is a compact map from~$\cC^{\alpha,\beta}$ to itself. Thus
    \begin{equation}
        D_fC_f = \tilde D_{f}\tilde C_f + (D_f-\tilde D_{f})\tilde C_f+\tilde D_{f}(C_f-\tilde C_f)+(D_f-\tilde D_{f})(C_f-\tilde C_f)
    \end{equation}
    is compact.
\end{proof}
Assembling the above results gives the following theorem.
\begin{theorem}\label{thm:fred} For $0<\alpha<\frac 12,\,-k<\beta\leq k,$
    the operator~$DC$ is a compact operator from~$\cC^{\alpha,\beta}$ to itself. Equation~\eqref{eq:precond_IE} is a second kind Fredholm integral equation,  and therefore of index 0. Equation~\eqref{eq:comp_IE} is therefore also Fredholm with index zero.
\end{theorem}

Coupled with the uniqueness result, Lemma~\ref{lem:uniqueness}, we have the following basic existence result.
\begin{corollary}[Existence]
     For $0<\alpha<\frac 12,\,-k<\beta\leq k,$ given $[f_D,f_N]\in \cC^{\alpha,\beta}\oplus \cC^{\alpha+\frac 12,\beta}$ the equation $(\Id+\cK)\begin{bmatrix}\sigma\\\tau\end{bmatrix}=\begin{bmatrix}f_D\\\ f_N\end{bmatrix}$  has a unique solution in $\cC^{\alpha,\beta}\oplus \cC^{\alpha+\frac 12,\beta}.$
\end{corollary}

By restriction this implies an existence result on any~$\tGamma\in\cG$.
\begin{lemma}
\label{lem:existence}
    Let~$\tGamma$ be any curve in~$\cG$.
    If~$[f_D,f_N]\in \cC^{\alpha,\beta}\oplus \cC^{\alpha+\frac12,\beta}$, then there exists a solution of~$[\sigma,\tau]\in\cC^{\alpha,\beta}\oplus \cC^{\alpha+\frac12,\beta}$ of $(\Id+\cK_{\tGamma})\begin{bmatrix}\sigma\\\tau\end{bmatrix}=\begin{bmatrix}f_D\\\ f_N\end{bmatrix}_{\tGamma}$
\end{lemma}
\begin{proof}
The solution of the equation on $\tGamma$ is simply the solution given in the corollary restricted to $\tGamma.$
\end{proof}

Theorem~\ref{thm:IE_solvable} then follows from the corollary and
Lemma~\ref{lem:existence}. We also observe the following fact about the
relationship between any solution of equation~\eqref{eq:comp_IE}  on a curve
$\tGamma\in\cG,$ and the analytic extension of the solution $[\sigma,\tau],$
given data $[f_D,f_N]\in\cC^{\alpha,\beta}\oplus\cC^{\alpha+\frac 12,\beta}.$
This is an extension of the argument used to prove Corollary~\ref{cor5.1.102}.
\begin{lemma}\label{lem11.102} If $[f_D,f_N]\in\cC^{\alpha,\beta}\oplus\cC^{\alpha+\frac 12,\beta},$ and $[\sigma_{\tGamma},\tau_{\tGamma}]\in\cC^{\alpha,\beta}_{\tGamma}\oplus\cC^{\alpha+\frac 12,\beta}_{\tGamma}$ solves 
\begin{equation}\label{eqn5.18.102}
(\Id+\cK_{\tGamma})\begin{bmatrix}\sigma_{\tGamma}\\\tau_{\tGamma}\end{bmatrix}=\begin{bmatrix}f_D\\\ f_N\end{bmatrix}_{\tGamma},
\end{equation}
then $[\sigma_{\tGamma},\tau_{\tGamma}]$ extends to $\GammaC,$ where it defines an element of $\cC^{\alpha,\beta}\oplus\cC^{\alpha+\frac 12,\beta}$ that solves~\eqref{eq:comp_IE}, and hence agrees with the solution found in Lemma~\ref{lem:existence}.
\end{lemma}
\begin{proof}
    Rewriting the equation we see that, for $x_2\in\tGamma,$
    \begin{equation}
       \begin{bmatrix}\sigma_{\tGamma}\\\tau_{\tGamma}\end{bmatrix}(x_2)=\begin{bmatrix}f_D\\\ f_N\end{bmatrix}(x_2)-\cK_{\tGamma}\begin{bmatrix}\sigma_{\tGamma}\\\tau_{\tGamma}\end{bmatrix}(x_2). 
    \end{equation}
    From the hypothesis of the lemma, and Remarks~\ref{rmk3.102} and~\ref{rmk4.103}, both terms on the right hand side have analytic extensions to $x_2\in\GammaC,$ which define an element of $[\sigma,\tau]\in\cC^{\alpha,\beta}\oplus\cC^{\alpha+\frac 12,\beta},$ analytically extending $[\sigma_{\tGamma},\tau_{\tGamma}].$ The lemma is immediate from this observation.
\end{proof}

\begin{remark}\label{rmk10.11}
Taken together Corollary~\ref{cor5.1.102}, Lemma~\ref{lem:existence}, and
Lemma~\ref{lem11.102} show that, for data in
$\cC^{\alpha,\beta}\oplus\cC^{\alpha+\frac 12,\beta},$ the problem of solving
equation~\eqref{eq:real_IE} on the real line, with appropriate estimates, can be
replaced with that of solving~\eqref{eqn5.18.102} for a convenient choice of
contour, $\tGamma$.  This is the basis for the complexification technique that we employ for
our numerical examples. Note that the solution to the latter problem restricted to
$\tGamma\cap\bbR$ agrees with the solution to~\eqref{eq:real_IE} on this
subset.

In fact, the arguments presented above can be used to show that
  $(\Id-D_{\tGamma}C_{\tGamma}) :$
  $\cC^{\alpha,\beta}_{\tGamma}\to\cC^{\alpha,\beta}_{\tGamma}$
  is Fredholm of index 0.  Corollary~\ref{cor5.1.102} then shows that this operator has a
  trivial null-space and is therefore invertible.
\end{remark}

\section[Solution to PDE 1.1]{Solution to PDE~\eqref{eqn:the_pde}}
  \label{sec:PDEsol}
We now check that the solution of~\eqref{eq:comp_IE} gives a solution of the transmission problem~\eqref{eqn:the_reduced_pde}. This is proved in~\cite{epstein2023solvinga} for data in $\cC^{\alpha}(\bbR)\oplus\cC^{\alpha+\frac 12}(\bbR),$ for $0<\alpha<\frac 12.$ If the data also satisfies the outgoing conditions in~\eqref{eq:outgoing}, then the solution to the transmission problem and PDE satisfies the outgoing radiation condition given in~\cite{epstein2024solving}. In this section we consider the extent to which the representations of $u^{l,r}(\bx),$ for $\bx\in\bbR^2,$ as an integral over the interface $\{(0,x_2)\,|\, x_2 \in \mathbb{R}\}$ can be replaced with integrals over a curve $\tGamma\in\cG,$ assuming the data belongs to $\cC^{\alpha,\beta}\oplus\cC^{\alpha+\frac 12,\beta}.$  Because the analytic continuations of the fundamental solutions of the bi-infinite waveguide problems do not necessarily decay at a uniform exponential rate along curves in $\cG,$ in order to compute $u^{l,r}$ using our representation \eqref{eqn:int_rep}, we impose the additional restriction $0\leq\beta.$

To deform the contour of integration, we consider the single and double layer
contributions separately. The Green's functions are written as
$G^{l,r}(\bx;\by)=G_0(\bx;\by)+w^{l,r}(\bx;\by);$ as the two terms have rather different analytic continuation properties, we treat them separately. Since the kernels $w^{l,r}(x_1,x_2;y_{1},y_2)$ depend only on $x_{1} -y_{1}$, it is convenient to represent them in terms of their partial Fourier transform, i.e.
\begin{equation}
     w^{l,r}(x_1,x_2;y_1,y_2) =\int_{\gamma_{F}} e^{i\xi (x_1-y_1)}\tilde w^{l,r}(\xi;x_2,y_2)\dd\xi,
    \end{equation}
    where $\gamma_{F}$ is defined to be $\{x+iy:\: y=-\tanh(x)\}.$ In every bounded subset of $\bbR^2 \times \bbR^2$, this representation is the same as the one presented in~\cite{epstein2023solvinga}, see Appendix~\ref{app:a} for a discussion 
    of the properties of $\tilde w^{l,r}(\xi; x_{2}, y_{2})$.

    As shown in Lemma~\ref{lem:kern_analytic}, the analytic continuations of the
    kernels $w^{l,r}$ have the same asymptotic properties when $x_1\neq 0$ as they do when $x_1=0.$ Hence these terms can be analytically continued to any curve in $\cG,$ and we easily show that, for any  $(x_1,x_2)\in\bbR^2,$ $\tau\in\cC^{\alpha+\frac 12,\beta},$ and $\tGamma\in\cG$
    \begin{equation}
\int_{\bbR}w^{l,r}(x_1,x_2;0,y_2)\tau(y_2)\dd y_2=\int_{\tGamma}w^{l,r}(x_1,x_2;0,y_2)\tau(y_2)\dd y_2,
    \end{equation}
    with a  similar result for the $\sigma$-integral.

What requires more care is the analytic continuation of
the free-space contributions.  This is because the free-space kernel
on $\{y_1=0\}$ is a function of $\sqrt{x_1^2+(y_2-x_2)^2},$ so care
must be taken to avoid branch cuts of this function as $y_2$ moves
into the complex plane. For stability it is also important to keep
$\Im\sqrt{x_1^2+(y_2-x_2)^2}\geq 0.$ If we require that $x_2\in[-\lt,\lt],$
then these conditions hold for $y_2\in\GammaL\cup\GammaR$ and
arbitrary $x_1.$ As $(y_2-x_2)^2$ lies in $\{z:\:\Im z\geq
0\}\setminus (-\infty,0],$ it is easy to see that
\begin{equation}\label{eqn5.22.107}
    0\leq\Im\sqrt{x_1^2+(y_2-x_2)^2}\leq |\Im y_2|.
\end{equation}
If $(x_1,x_2)\in K,$ for a compact set $K\subset\bbR^2,$ then as $|\Im y_2|\to \infty,$ we see that
\begin{equation}
  \frac{\Im\sqrt{x_1^2+(y_2-x_2)^2}}{|\Im y_2|}\to 1,
\end{equation}
but note that as $|x_1|\to\infty,$ with $y_2$ bounded, the
$\Im\sqrt{x_1^2+(y_2-x_2)^2}\to 0^+.$ In light of this we now assume that
$\beta> 0$ in order to be certain that  various integrals we consider are
\emph{uniformly}, absolutely convergent.

From the standard asymptotic expansion for the free-space kernel we see that for~ $(x_1, x_2)\in \bbR,$ with $|x_2|<\lt$ and $y_2\in\GammaL\cup\GammaR,$ we obtain the uniform estimates
\begin{equation}
    \left| G_0(x_1,x_2;0;y_2) - \frac{C_{\pm}e^{ik \sqrt{x_1^2+(x_2-y_2)^2}}}
    {[x_1^2+(x_2-y_2)^2]^{\frac 14}}\right| \leq \frac{Ke^{-k \Im\sqrt{x_1^2+(x_2-y_2)^2}}}{\sqrt{|x_1^2+(x_2-y_2)^2|}};\label{eq:Hank_asymp}
\end{equation}
with a similar estimate for the $\pa_{y_1}$-derivatives of the free-space
kernel:
\begin{equation}
   \left| \pa_{y_1}G_0(x_1,x_2;0;y_2) - \frac{C'_{\pm}x_1e^{ik \sqrt{x_1^2+(x_2-y_2)^2}}}
    {[x_1^2+(x_2-y_2)^2]^{\frac 34}}\right| \leq \frac{K(|x_1|+1)e^{-k\Im \sqrt{x_1^2+(x_2-y_2)^2}}}{\sqrt{|x_1^2+(x_2-y_2)^2|^{\frac 54}}}.\label{eq:Hank1_asymp}
\end{equation}
Using these estimates we can deform the contour of integration for the free-space terms.
\begin{lemma}\label{lem:layer_contour}
  Let~$\sigma\in\cC^{\alpha,\beta}$ and~$\tau \in \cC^{\alpha+\frac 12,\beta}$ for  $0<\alpha<\frac 12,$ $0\leq\beta\leq k$. The values~$\cS^{l,r}_{\tGamma}[\tau](\bx)$ and~$\cD^{l,r}_{\tGamma}[\sigma](\bx)$ exist for every~$\bx\in\bbR\times[-\lt,\lt]$ and are independent of~$\tGamma\in \cG$.
\end{lemma}
\begin{proof}
    We use the splitting
    \begin{equation}
        G^{l,r} = G_0 + w^{l,r}
    \end{equation}
    to split the layer potentials into a free-space part and a waveguide correction:
    \begin{equation}
        \cS^{l,r}_{\tGamma}=\cS_{\tGamma} +\cS^{l,r}_{w,\tGamma}\quad\text{and}\quad\cD_{\tGamma}^{l,r}=\cD_{\tGamma} +\cD^{l,r}_{w,\tGamma}.
    \end{equation}
    The existence and path independence of the waveguide contributions~$\cS^{l,r}_w$ and~$\cD^{l,r}_w$ is proved  above. The proof of convergence and path independence of the free-space terms follows from~\eqref{eq:Hank_asymp},~\eqref{eqn5.22.107}, the analyticity properties of $H^{(1)}_0(z),$ and the decay properties of $(\sigma,\tau)$ 
\end{proof}
This lemma combined with Theorem~\ref{thm:IE_solvable}, Lemma~\ref{lem11.102}, and Theorem 3 in~\cite{epstein2024solving} completes the proof of Theorem~\ref{thm:IE_outgoing}. We conclude this section by showing that various physically relevant incoming fields give data that satisfies the conditions of Theorem~\ref{thm:IE_outgoing}.

\subsection{Admissible Data}
We consider various types of physically relevant data for our integral equations. Some of the data is shown to be admissible. For certain cases the data is not obviously admissible, but it is explained how the data can be modified to make it admissible.

\subsubsection*{Point sources}
If $\by$ is fixed  with $y_1\neq 0,$ and $|y_2|<\lt,$ then, as $x_2$ tends to infinity within $\GammaL\cup\GammaR,$ we have the estimate
\begin{equation}
    \sqrt{y_1^2+(x_2-y_2)^2}\sim (\sign\Re x_2) x_2.
\end{equation}
This observation, along with standard properties of Hankel functions show that the waveguide Green's function gives data in the desired~$\cC^{\alpha,\beta}$ spaces. The outgoing asymptotics for $w^{l,r}$ are derived in~\cite{epstein2023solvinga}. 
\begin{lemma}
    If~$\by \in (\bbR\setminus\{0\})\times[-\lt,\lt]$, then, as a function of $x_2,$
    \begin{equation}
        f_D(x_2)=G^{l,r}(0,x_2;\by)\in\cC^{\frac 12,k}\quad\text{and}\quad f_N(x_2)=\partial_{x_1}G^{l,r}(0,x_2;\by)\in\cC^{\frac 32,k}.
    \end{equation} 
    Further, these~$f_D$ and~$f_N$ have outgoing asymptotic expansions like those in~\eqref{eq:outgoing}.
\end{lemma}

\subsubsection*{Waveguide modes}
For any non-negative potential~$q,$ which is not identically zero, the
equation $(\Delta +k^2(q+1))u=0$  admits solutions of the form
\begin{equation}\label{eqn5.35.30}
  u(x_1,x_2)=  v(x_2)e^{\pm i\xi x_1}, \text{ with }v\in L^2(\bbR),
\end{equation}
for a finite collections of~$\xi \in (k,k\sqrt{M_q+1})$, where
$M_q=\sup\{q(x)\}.$ These solutions are known as waveguide
modes for the potential~$q$.  Let~$v(x_2)e^{\pm i\xi x_1}$ be a waveguide
mode of $q=q^l$ or $q=q^r.$ The corresponding boundary data for the
integral equation is given by $f_D(x_2)=v(x_2)$ and~$f_N(x_2)=\pm i\xi v(x_2).$
There are constants $C_{\pm}$ so that $v\in \cC^{1}(\bbR)$ is of the form
    \begin{equation}
        v(x_2) = C_{\pm}e^{\mp \sqrt{\xi^2-k^2}x_2}
    \end{equation}
    when~$\pm x_2>d,$ (see~\cite{epstein2023solvinga}). This data is clearly analytic on the interior of~$\GammaC$ and decays
    exponentially as $|\Re x_2|\to\infty,$ but not as $|\Im x_2|\to\infty,$ so does not belong to any $\cC^{\alpha,\beta}$-space. 
    This will prevent us from applying the truncation estimate in Theorem~\ref{thm:trunc_sol}. To
    overcome this defect we let $\tilde{f}_{D}$, and
    $\tilde{f}_{N}$ be given by
    \begin{equation}
    \begin{bmatrix}
    \tilde{f}_{D} \\
    \tilde{f}_{N}
    \end{bmatrix}
    = -\cK_{\bbR} \begin{bmatrix}
    f_{D} \\ 
    f_{N}
    \end{bmatrix} \,.
    \end{equation}
 Since $[f_D,f_N]\in\cC^{\alpha,\beta}_{\tGamma} \oplus \cC^{\alpha+1/2,\beta}_{\tGamma}$ for any~$\tGamma$
    which satisfy the slope condition in~\eqref{eq:slop-cond} and any~$\alpha>0$ and $\beta < \sqrt{\xi^2-k^2} C$, it follows from Remark~\ref{rmk4.103} that $[\tilde{f}_{D}, \tilde{f}_{N}] \in \cC^{\alpha,k} \oplus \cC^{\alpha+1/2,k}$, 
    for any $\alpha \leq1/2$.   Let $[\tilde{\sigma}, \tilde{\tau}]$ denote the solution to~\eqref{eq:comp_IE} with data
    $[\tilde{f}_{D}, \tilde{f}_{N}]$, then
    \begin{equation}
    \begin{bmatrix}
    \sigma \\
    \tau
    \end{bmatrix}
    = \begin{bmatrix}
    \tilde{\sigma} \\
    \tilde{\tau}
    \end{bmatrix}
    + 
    \begin{bmatrix}
    f_{D} \\
    f_{N}
    \end{bmatrix}
    \end{equation}
    satisfies~\eqref{eq:comp_IE} for all $\tGamma \in \cG$ which satisfy the slope condition in~\eqref{eq:slop-cond}. Further, it is also clear that 
    \begin{equation}
        \cS^{l,r}_{\tGamma}[f_N](\bx)+\cD^{l,r}_{\tGamma}[f_D](\bx)= \cS^{l,r}_{\bbR}[f_N](\bx)+\cD^{l,r}_{\bbR}[f_D](\bx),
    \end{equation}
    so Theorem~\ref{thm:IE_outgoing} can be extended to this case.
    
\subsection{Truncation error}
In this section, we show that the we can solve~\eqref{eq:comp_IE} on a truncated
contour with a controllable error. We fix a~$\tGamma\in
\cG$. Assuming that $\beta>0,$ we let~$\tGamma_{\epsilon}$ be the truncation of~$\tGamma$ to the
region~$\exp(-\beta|\Im x_2|)<\epsilon$. We can think of
$\cK_{\tGamma},\cK_{\tGamma_{\epsilon}}$ either as maps
\begin{equation}
  \cK_{\tGamma},\cK_{\tGamma_{\epsilon}}:
    \cC^{\alpha,\beta}_{\tGamma}\oplus \cC^{\alpha+\frac 12,\beta}_{\tGamma}\to
    \cC^{\alpha,\beta}\oplus \cC^{\alpha+\frac 12,\beta},
\end{equation}
or, after first composing with restriction to $\tGamma$ or
$\tGamma_{\epsilon},$ respectively, as  maps
\begin{equation}\label{eqn5.35.110}
  \cK_{\tGamma},\cK_{\tGamma_{\epsilon}}:
    \cC^{\alpha,\beta}\oplus \cC^{\alpha+\frac 12,\beta}\to
    \cC^{\alpha,\beta}\oplus \cC^{\alpha+\frac 12,\beta}.
\end{equation}
For the most part, the latter interpretation is more useful in our
applications.

We begin our analysis of truncation error by showing that, for
sufficiently small $0<\epsilon,$  the difference
between $\cK_{\tGamma}$ and $\cK_{\tGamma_{\epsilon}}$ can be bounded in norm by
$\epsilon$ with a prefactor depending on the potentials, $q_{l,r},$ and the choice of contour
$\tGamma$.
\begin{lemma}\label{lem:trunc_op} For $0<\alpha<\frac 12,\,  -k<\beta\leq k,$
  $0<\epsilon,$ and $\tGamma\in \cG$, let 
  $$\tGamma_{\epsilon}=\{x_2\in\tGamma:\: e^{-(\beta+k)|\Im  x_2|}<\epsilon\}. $$
      There exists a~$C>0$  such that
    \begin{equation}
        \left\| (\cK_{\tGamma} -
        \cK_{\tGamma_{\epsilon}})\begin{bmatrix}\sigma_{\tGamma}\\
          \tau_{\tGamma}\end{bmatrix}\right\|_{\cC^{\alpha,\beta}\oplus
          \cC^{\alpha+\frac 12,\beta}} < C\epsilon
        \left\|\begin{bmatrix}\sigma_{\tGamma}\\ \tau_{\tGamma}\end{bmatrix}\right\|_{\cC^{\alpha,\beta}_{\tGamma}\oplus
          \cC^{\alpha+\frac 12,\beta}_{\tGamma}}.
    \end{equation}
\end{lemma}
\begin{proof}
    We begin by looking at the operator~$C$.
    Letting~$\sigma$ be any element of~$\cC^{\alpha,\beta},$ we have
    \begin{equation}
       C_{\tGamma}[\sigma](x_2) - C_{\tGamma_{\epsilon}}[\sigma](x_2) = \int_{\tGamma\setminus \tGamma_{\epsilon}} k_C(x_2,y_2) \sigma(y_2) {\rm d}y_2.
    \end{equation}
    Using our bounds on~$k_C$ and~$\sigma$ gives
    \begin{multline}
        \left|C_{\tGamma}[\sigma](x_2) -
        C_{\tGamma_{\epsilon}}[\sigma](x_2)\right| \leq \int_{\tGamma\setminus
          \tGamma_{\epsilon}} \frac{Ke^{-k(|\Im x_2|+|\Im y_2|)}}{(1+|x_2|+|y_2|)^{\frac 32}}\frac{e^{-\beta|\Im y_2|}\|\sigma\|_{\alpha,\beta,\tGamma}}{(1+|y_2|)^\alpha} |{\rm d}y_2|\\
        \leq  \int_{\tGamma\setminus \tGamma_{\epsilon}}\frac{\epsilon Ke^{-k|\Im x_2|}\|\sigma\|_{\alpha,\beta,\tGamma}}{(1+|x_2|+|y_2|)^{\frac 32}(1+|y_2|)^\alpha} |{\rm d}y_2|.
    \end{multline}
    We let~$s(y_2)$ be the signed arc-length from~$0$ to~$y_2$ along~$\tGamma;$
    the monotonicity condition in~\eqref{eqn3.1.30} implies that
    that~$|y_2|>\frac{1}{\sqrt{2}}|s|$. This allows us to bound the integral by
    \begin{equation}
        \left|C_{\tGamma}[\sigma](x_2) - C_{\tGamma_{\epsilon}}[\sigma](x_2)\right|\leq\int_{-\infty}^\infty\frac{\epsilon \sqrt{2} Ke^{-k|\Im x_2|}\|\sigma\|_{\alpha,\beta,\tGamma}}{(1+|x_2|+\frac{|s|}{\sqrt{2}})^{\frac 32}(1+\frac{|s|}{\sqrt{2}})^\alpha} ds 
    \end{equation}
    Applying Lemma~\ref{lem:basic_integral} gives the estimate
    $$\|C_{\tGamma}[\sigma]- C_{\tGamma_{\epsilon}}[\sigma]\|_{\alpha+\frac 12,\beta}<
    C\epsilon\|\sigma\|_{\alpha,\beta,\tGamma}.$$
   A similar argument can be applied to~$D$ to show that
   \begin{equation}
     \|D_{\tGamma}[\tau]-
     D_{\tGamma_{\epsilon}}[\tau]\|_{\alpha,\beta}<
     C\epsilon\|\tau\|_{\alpha+\frac 12,\beta,\tGamma}.
   \end{equation}

 Combining these estimates  gives the desired result.
\end{proof}

For $(\sigma,\tau)\in \cC^{\alpha,\beta}\oplus\cC^{\alpha+\frac 12,\beta}$ we therefore have the estimate
\begin{equation}\label{eqn5.41.110}
    \left\| (\cK_{\tGamma} -
        \cK_{\tGamma_{\epsilon}})\begin{bmatrix}\sigma_{\tGamma}\\
          \tau_{\tGamma}\end{bmatrix}\right\|_{\cC^{\alpha,\beta}\oplus
          \cC^{\alpha+\frac 12,\beta}} < C\epsilon
        \left\|\begin{bmatrix}\sigma\\ \tau\end{bmatrix}\right\|_{\cC^{\alpha,\beta}\oplus
          \cC^{\alpha+\frac 12,\beta}}.
\end{equation}
Using this estimate, we can control the error caused by solving~\eqref{eq:comp_IE} on a truncated contour.
\begin{theorem}\label{thm:trunc_IE} For each $0<\alpha<\frac 12,\,
  -k<\beta\leq k,$ and $\tGamma\in\cG,$ there is an  $\epsilon_0>0,$ such that,
  for $\epsilon<\epsilon_0,$ the operator
  \begin{equation}
  \label{eq:ckspace}
    (\Id+\cK_{\tGamma_{\epsilon}}):
    \cC^{\alpha,\beta}\oplus
    \cC^{\alpha+\frac 12,\beta}\to \cC^{\alpha,\beta}\oplus
    \cC^{\alpha+\frac 12,\beta}
  \end{equation}
  is invertible and there exists a~$C>0,$ independent of $\epsilon,$ such that
    \begin{equation}
      \left\| \left[(\cI+\cK_{\tGamma})^{-1} -(\cI+\cK_{\tGamma_{\epsilon}})^{-1}
        \right]\begin{bmatrix}
        f_D\\ f_N\end{bmatrix}\right\|_{\cC^{\alpha,\beta}\oplus
          \cC^{\alpha+\frac 12,\beta}}< C\epsilon\left\|\begin{bmatrix}f_{D}\\f_{N}\end{bmatrix}\right\|_{\cC^{\alpha,\beta}\oplus \cC^{\alpha+\frac 12,\beta}}.
    \end{equation}
\end{theorem}
\begin{proof}

    We first note that
    \begin{equation}\label{eqn5.44.110}
        \cI+\cK_{\tGamma_{\epsilon}} =
        \cI+\cK_{\tGamma}+\lp\cK_{\tGamma_{\epsilon}} -\cK_{\tGamma}\rp =
        \lp\cI+\cK_{\tGamma}\rp\left[ \cI + \lp\cI+\cK_{\tGamma}\rp^{-1}
          \lp\cK_{\tGamma_{\epsilon}} -\cK_{\tGamma}\rp\right],
    \end{equation}
    where we have used the fact that $\cI+\cK_{\tGamma}$ is invertible is
    invertible on $\cC^{\alpha,\beta}\oplus \cC^{\alpha+\frac 12,\beta},$ which contains the range of $\cK_{\tGamma_{\epsilon}} -\cK_{\tGamma}$. The estimate in~\eqref{eqn5.41.110} shows that there is a
    constant, $C,$ so that the norm of the composition
    \begin{equation}
      \lp\cI+\cK_{\tGamma}\rp^{-1} \lp\cK_{\tGamma_{\epsilon}}
      -\cK_{\tGamma}\rp:
      \cC^{\alpha,\beta}\oplus \cC^{\alpha+\frac 12,\beta}\to \cC^{\alpha,\beta}\oplus \cC^{\alpha+\frac 12,\beta}
    \end{equation}
    is bounded by
    $$\eta=C\epsilon\left\|\lp\cI+\cK_{\tGamma}\rp^{-1}
    \right\|_{\cC^{\alpha,\beta}\oplus \cC^{\alpha+\frac 12,\beta}}.$$ For
    sufficiently small $0<\epsilon,$ we see that $\eta<1;$  we can therefore use a
    Neumann series to invert the term in brackets on the right hand
    side of~\eqref{eqn5.44.110}  to obtain the invertibility of
    $\Id+\cK_{\tGamma_{\epsilon}},$ as well as the expression
    \begin{equation}
        \lp\cI+\cK_{\tGamma_{\epsilon}}\rp^{-1} = \left[ \cI + \lp\cI+\cK_{\tGamma}\rp^{-1} \lp\cK_{\tGamma_{\epsilon}} -\cK_{\tGamma}\rp\right]^{-1}\lp\cI+\cK_{\tGamma}\rp^{-1}.
    \end{equation}
   From this formula and the Neumann series we easily show the
   existence of a constant $C'$ for which
    \begin{equation}
     \left \| \left[ \lp\cI+\cK_{\tGamma_{\epsilon}}\rp^{-1}
        -\lp\cI+\cK_{\tGamma}\rp^{-1}\right]
      \begin{bmatrix}
            f_D \\ f_N
        \end{bmatrix}\right\|_{\cC^{\alpha,\beta}\oplus
        \cC^{\alpha+\frac 12,\beta}}
      \leq C'\epsilon\left\|\begin{bmatrix}f_D\\f_N\end{bmatrix}\right\|_{\cC^{\alpha,\beta}\oplus \cC^{\alpha+\frac 12,\beta}},
    \end{equation}
    which is equivalent to the desired result.
\end{proof}
We now observe that it is enough to enforce the integral equation on the compact contour~$\tGamma_{\epsilon}$.
\begin{lemma}\label{lem:solve_compact}
Let~$\alpha,\beta,$ and~$\epsilon_0$ satisfy the assumptions of the previous theorem. If $\tGamma\in \cG$ and $[f_D,f_N]\in \cC^{\alpha,\beta}\oplus \cC^{\alpha+\frac 12,\beta}$, then for any~$\epsilon<\epsilon_0$ there exists a unique~$[\sigma_\epsilon,\tau_\epsilon]\in \cC^{\alpha,\beta}\oplus \cC^{\alpha+\frac 12,\beta}$ satisfying
\begin{equation}\label{eqn5.60.11}
        \lp\cI+\cK_{\tGamma_{\epsilon}}\rp\begin{bmatrix}
            \sigma_\epsilon\\ \tau_\epsilon
        \end{bmatrix}(x_2)=\begin{bmatrix}
            f_D \\ f_N
        \end{bmatrix}(x_2)\quad \forall x_2\in \tGamma_{\epsilon}.
   \end{equation}
Moreover, the solution to~\eqref{eqn5.60.11} agrees with
\begin{equation}
\lp\cI+\cK_{\tGamma_{\epsilon}}\rp^{-1} 
\begin{bmatrix}
f_{D} \\
f_{N}
\end{bmatrix} \,,
\end{equation}
with $\cI + \cK_{\tGamma_{\epsilon}}$ being the operator defined on the spaces in~\eqref{eq:ckspace}. 
\end{lemma}
\begin{proof}
    The existence of a solution is an immediate consequence of the previous theorem, i.e., a solution obtained by enforcing the integral equation for all $x_{2} \in \GammaC$ trivially satisfies the equation on $\tGamma_{\epsilon}$. Now let $[\sigma_{\epsilon}, \tau_{\epsilon}]$ denote any solution obtained by enforcing the integral equation for all $x_{2} \in \tGamma_{\epsilon}$. To prove uniqueness, we let
    \begin{equation}
        \begin{bmatrix}
            \tsigma_\epsilon\\ \ttau_\epsilon
        \end{bmatrix}:=\begin{bmatrix}
            f_D \\ f_N
        \end{bmatrix} - \cK_{\tGamma_{\epsilon}}\begin{bmatrix}
            \sigma_\epsilon\\ \tau_\epsilon
        \end{bmatrix},
    \end{equation}
for all $x_{2} \in \GammaC$. This function is in~$\cC^{\alpha,\beta}\oplus \cC^{\alpha+\frac 12,\beta}$ by~Lemma \ref{lem:op_bd} and is equal to $[\sigma_{\epsilon}, \tau_{\epsilon}]$ on $\tGamma_{\epsilon}$ by~\eqref{eqn5.60.11}. Moreover, it satisfies
    \begin{equation}
        \lp\cI+\cK_{\tGamma_{\epsilon}}\rp\begin{bmatrix}
            \tsigma_\epsilon\\ \ttau_\epsilon
        \end{bmatrix}=\begin{bmatrix}
            f_D \\ f_N
        \end{bmatrix}
    \end{equation}
    on~$\GammaC$. Theorem~\ref{thm:trunc_IE} gives that~$\tsigma_{\epsilon}$ and~$\ttau_{\epsilon}$ are both uniquely determined by~$f_D$ and~$f_N$. 
    The uniqueness of~$[\sigma_{\epsilon},\tau_\epsilon]$ therefore follows from the identity theorem for analytic functions.
\end{proof}


We now give a proof of Theorem~\ref{thm:trunc_sol},
which bounds the error in the potential $u_{\epsilon}$ obtained by solving for
densities $(\sigma_{\epsilon}, \tau_{\epsilon})$ on any truncated
curve $\tGamma_{\epsilon}$. Since we now include the contribution
of the free-space kernel, we again assume that $0<\beta.$
\begin{proof}[Proof of Theorem~\ref{thm:trunc_sol}]
   Recall that $0<\beta$ and $\tGamma^c_{\epsilon}=\{y_2\in\tGamma:\:
   e^{-\beta|\Im y_2|}<\epsilon\}.$ Let~$[\sigma_\epsilon,\tau_\epsilon]$ solve \eqref{eqn5.60.11} and~$[\sigma,\tau]$ solve the equation
        \begin{equation}
        \lp\cI+\cK_{\tGamma}\rp\begin{bmatrix}
            \sigma\\ \tau
        \end{bmatrix}=\begin{bmatrix}
            f_D \\ f_N
        \end{bmatrix}
    \end{equation}
    on all of~$\Gamma_C$. By Theorem~\ref{thm:trunc_IE} and Lemma~\ref{lem:solve_compact}, we have that
    \begin{equation}
        \|\sigma-\sigma_\epsilon\|_{\alpha,\beta} +\|\tau-\tau_\epsilon\|_{\alpha + \frac 12 , \beta} < C\epsilon\left(  \|f_D\|_{\alpha,\beta} + \|f_N\|_{\alpha+\frac 12,\beta}\right).
    \end{equation}
    We then write
    \begin{equation}
        \cS^{l}_{\tGamma_{\epsilon}}[\tau_\epsilon] -\cS^{l}_{\tGamma}[\tau]= \cS^{l}_{\tGamma}[\tau_{\epsilon}-\tau] + \lp\cS^{l}_{\tGamma_{\epsilon}}-\cS^{l}_{\tGamma}\rp[\tau_{\epsilon}].
    \end{equation}

    Using Lemma~\ref{lemma:w_decay} and~\eqref{eq:Hank_asymp} we can prove that
    \begin{equation}\label{eqn5.51.110}
    \begin{aligned}
        \left|\lp\cS^{l}_{\tGamma_{\epsilon}}-\cS^{l}_{\tGamma}\rp[\tau_{\epsilon}](\bx )\right|&\leq C\epsilon \|\tau_\epsilon\|_{\alpha+\frac 12,\beta} \\
        &\leq C'\epsilon \left(  \|f_D\|_{\alpha,\beta} + \|f_N\|_{\alpha+\frac 12,\beta}\right)\,,
        \end{aligned}
    \end{equation}
    with a similar proof to Lemma~\ref{lem:trunc_op}. Using the same
    steps as the proof of Lemma~\ref{lem:op_bd} and
    Theorem~\ref{thm:trunc_IE}, we can show that
    \begin{equation}
    \begin{aligned}
        |\cS^{l}_{\tGamma}[\tau_{\epsilon}-\tau] (\bx)|&\leq C \|\tau_{\epsilon}-\tau\|_{\alpha+\frac 12,\beta} \\
        &\leq C'\epsilon \left(  \|f_D\|_{\alpha,\beta} + \|f_N\|_{\alpha+\frac 12,\beta}\right)\,.
        \end{aligned}
    \end{equation}
    A similar proof shows that
    \begin{equation}
        |\cD^{l}_{\tGamma_{\epsilon}}[\sigma_\epsilon](\bx) -\cD^{l}_{\tGamma}[\sigma](\bx)|< C'\epsilon \left(  \|f_D\|_{\alpha,\beta} + \|f_N\|_{\alpha+\frac 12,\beta}\right)\,.
    \end{equation}
    Summing these two bounds and applying Theorem~\ref{thm:IE_outgoing} gives that
    \begin{equation}
    \begin{aligned}
        \left|\cS^{l}_{\tGamma_{\epsilon}}[\tau_\epsilon]-\cD^{l}_{\tGamma_{\epsilon}}[\sigma_\epsilon](\bx) -u^l(\bx)\right| &=  \left|\cS^{l}_{\tGamma_{\epsilon}}[\tau_\epsilon]-\cD^{l}_{\tGamma_{\epsilon}}[\sigma_\epsilon](\bx) -\cS^{l}_{\tGamma}[\tau]+\cD^{l}_{\tGamma}[\sigma](\bx)\right|\\
        &<2C\epsilon \left(  \|f_D\|_{\alpha,\beta} + \|f_N\|_{\alpha+\frac 12,\beta}\right).
        \end{aligned}
    \end{equation}
    Repeating the same argument for~$u^r$ gives the desired result.
\end{proof}

We conclude this section with the proof of Corollary~\ref{cor:mode_soln}, which addresses the problems associated with waveguide mode data.
\begin{proof}[Proof of Corollary~\ref{cor:mode_soln}]
    Rearranging \eqref{eq:mode_comp_IE} shows that~$[\tsigma,\ttau]$ must solve
    \begin{equation}\label{eq:mode_comp}
        \lp\Id + \cK_{\tGamma} \rp\begin{bmatrix}
            \tsigma\\\ttau
        \end{bmatrix}(x_2)
        =
        -\cK_{\tGamma}\begin{bmatrix}
            f_D\\f_N
        \end{bmatrix}(x_2) \quad\forall x_2\in\tGamma.
    \end{equation}
    We observe that the right hand side is in $\cC^{1/2,k} \oplus \cC^{3/2,k}$. The existence and uniqueness of~$[\tsigma,\ttau]$ then follows from Theorem~\ref{thm:IE_solvable}. The formula \eqref{eq:mode_comp_rep} then follows from Lemma~\ref{lem:layer_contour} and the analyticity of the waveguide modes outside the support of~$q^{l,r}$. Theorem 3 in \cite{epstein2024solving} then gives that the function~$u^{l,r}$ solves \eqref{eqn:the_reduced_pde}. 

    The error estimate \eqref{eq:mode_rep_err} cannot be derived from Theorem~\ref{thm:trunc_sol} because though $[f_{D}, f_{N}]$ are in $\cC^{0,0}\oplus \cC^{0,0}$, but they are not in any $\cC^{\alpha,0}$ for $\alpha>0$. Instead, we apply Theorem~\ref{thm:trunc_sol} with $\beta=k$ to see that 
    \begin{equation}\label{eqn6.42.40}
    \begin{aligned}
        \left|\lp \cS^{l,r}_{\tGamma_{\eta}}[\ttau_\epsilon](\bx)+\cD^{l,r}_{\tGamma_{\eta}}[\tsigma_\epsilon](\bx)\rp - \lp \cS^{l,r}_{\tGamma}[\ttau](\bx)+\cD^{l,r}_{\tGamma}[\tsigma](\bx)\rp\right| &< C\eta \left\| \cK_{\tGamma}\begin{bmatrix}
            f_D\\f_N
        \end{bmatrix}\right\|_{\cC^{1/2,k}\oplus \cC^{1,k}} \\
        &\leq C\epsilon \left\| \cK_{\tGamma}\begin{bmatrix}
            f_D\\f_N
        \end{bmatrix}\right\|_{\cC^{1/2,k}\oplus \cC^{1,k}}  \,.
        \end{aligned}
    \end{equation}

What remains is to bound the truncation error for $\cS^{l,r}_{\tGamma_{\eta}}[f_{N}]$ and $\cD^{l,r}_{\tGamma_{\eta}}[f_{D}]$. 
For any fixed $\tGamma$ satisfying the slope condition~\eqref{eq:slop-cond}, it is easy to show that the data $[f_{D}, f_{N}]$ restricted to $\tGamma$ is exponentially decaying, and therefore lies in $\cC^{\alpha,\beta}_{\tGamma}\oplus \cC^{\alpha+\frac 12,\beta}_{\tGamma}$ for some positive $\alpha$ and $\beta.$ The value of $\beta$ depends on the slope of the curve $\tGamma$ and the waveguide frequency $\xi$. The estimate for the single layer contribution is
\begin{equation}
\begin{split}
     \left| \cS^{l,r}_{\tGamma_{\eta}}[f_N](\bx) - \cS^{l,r}_{\tGamma}[f_N](\bx)\right| 
     &\leq \int\limits_{\tGamma\setminus\tGamma_{\eta}}
     |G^{l,r}(\bx;0;y_2)f_N(y_2)| |dy_2|\\
     &\leq C\int\limits_{\tGamma\setminus\tGamma_{\eta}} e^{-\beta|\Im y_2|}|dy_2|.
     \end{split}
\end{equation}
Let $\eta_{0}$ be such that the above integral is less than $\epsilon$. The parameter $\eta$ is then given by $\eta=\min(\eta_{0}, \epsilon)$. This issue only arises if $x_1$ is allowed to tend to infinity. It has nothing to do with either the solvability of the integral equation, or the truncation error in~\eqref{eqn6.42.40}. A similar argument applies to the double layer contribution.
\end{proof}

\section{Numerical experiments}\label{sec:numerics}
In this section, we illustrate the performance of our approach through several numerical examples. We set $k=3$ throughout and consider two different configurations: a) where $q^{a}_{l,r}(x_{2})$
are the smooth potentials given by
\begin{equation}
\begin{aligned}
    q^a_{l}(x) &=  2 e^{-2x^2} (3+2\cos (5\pi x+0.7)-\cos(2x+1.6))\, ,\\
    q^a_{r}(x) &= \frac{8}{3} e^{-x^2/0.245} (1.53 - \cos (1.05)+0.53 \cos (x+0.25))\,, 
    \end{aligned}
\end{equation}
and; b) where $q^{b}_{l} = q^{a}_{l}$ and 
$q^{b}_{r} = \chi_{[-3,3]}$. 
The parameters are chosen so that $q_l := q^{a,b}_{l}$ supports 
four waveguide modes with frequencies 
$$(\xi_{l,1}, \xi_{l,2}, \xi_{l,3}, \xi_{l,4}) = (3.7401, 4.6783, 6.2958, 7.5942),$$
and 
$q^{a,b}_r$ 
support two, respectively  five waveguide modes with frequencies
\begin{equation*}
\begin{aligned}
(\xi^a_{r,1}, \xi^a_{r,2}) &= (4.1378, 5.9297), \\
(\xi^b_{r,1}, \xi^b_{r,2},\xi^b_{r,3}, \xi^b_{r,4}, \xi^b_{r,5}) &= (3.2274,3.5503,3.8083,4.0022,4.2164).
\end{aligned}
\end{equation*}
For brevity, we have reported the first $5$ significant digits of the waveguide frequencies, however, they have been
computed to $10$ significant digits in the numerical computations.
The potentials and waveguide modes are plotted in Figure \ref{fig:wg_modes}. 

\begin{figure}[h!]
    \centering
    \includegraphics[width=\linewidth]{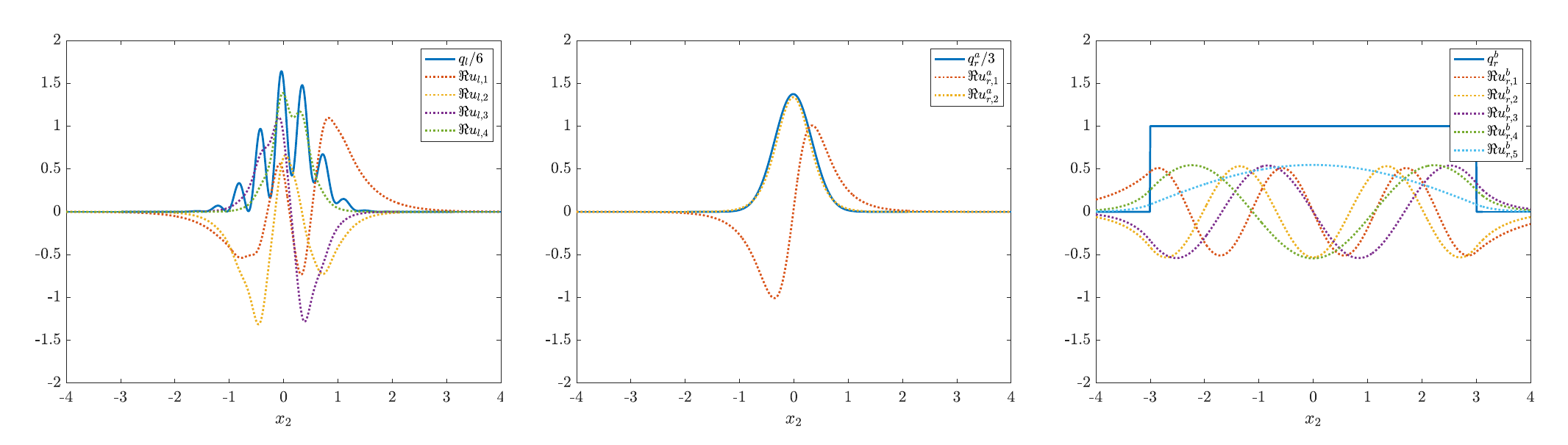}
    \caption{The potentials $q_l$ and $q_r^{a,b}$, shown in blue, and
      cross-sections of the corresponding waveguide modes.}
    \label{fig:wg_modes}
\end{figure}

The efficient evaluation of the Green's functions $G^{l,r}$ is fairly
involved. It is perhaps most natural to compute them via their partial Fourier
transform. There are three main technical issues: the slow decay of the Fourier
transform induced by the logarithmic singularity in real space for $x_{2},y_{2}
\in [-d,d]$; the presence of waveguide modes, which correspond to poles on the
real axis in the Fourier domain, and the oscillatory nature of the integrand for
large $|x_1-y_1|.$ Here we provide a brief sketch of the method used in this
paper for their computation. We refer the reader to
~\cite{cai2013computational,okhmatovski2004evaluation,wang2019fast,zhang2020exponential,paulus2000accurate,bruno2016windowed,bruno2017windowed,goodwill2024numerical},
for examples of other approaches.


As mentioned above, the Green's functions $G^{l,r}$ are evaluated by taking the Fourier transform in the $x_{1}$ variable. The Fourier transform is then evaluated at interpolation nodes on a complex contour, similar to the blue contour in Figure~\ref{fig_contours}. For each Fourier frequency $\xi$, depending on $y_{2}$, the solution for all $x_{2}$ and $x_{1}-y_{1}$ requires the solution of an ordinary differential equation (ODE) on $\bbR$ (or  more precisely $\tGamma$, though this does not require extra work). This ODE is solved using a shooting method with a Runge-Kutta discretization when $|\xi|$ is small, and using a WKB expansion of order $6$ when $|\xi|$ is large when $q$ is smooth, and a WKB expansion of order $1$ for the piecewise constant case. The discretization of the contour and the parameters of the ODE solver are chosen to guarantee $6$ relative digits of accuracy in the computed Green's function for $\bx, \by \in [-L_{0}, L_{0}] \times \GammaC$.

\begin{figure}[h!]
    \centering
    \includegraphics[width=\linewidth]{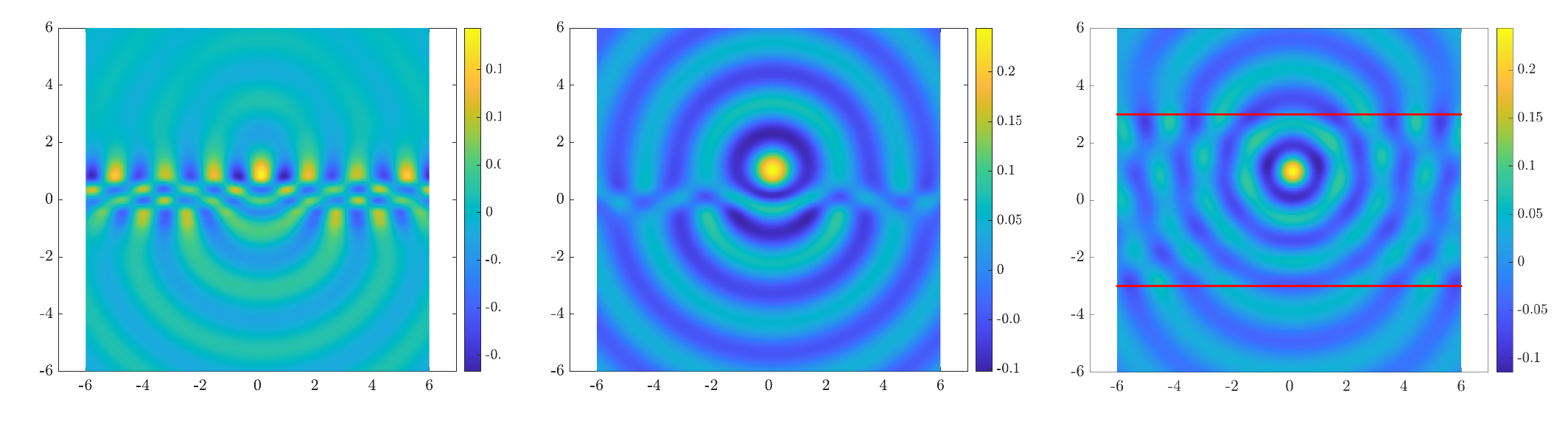}
    \caption{Imaginary part of the Green's functions $G_q (\bx, \by_0)$ with $q = q_l, q_r^a, q_r^b$ (left to right), and $\by_0 := (0.1, 1)$.}
    \label{fig:greenfun} 
\end{figure}
In Figure~\ref{fig:greenfun}, we plot the Green's functions for the three different media for $\by_{0} = (0.1,1)$. 
The accuracy of the Green's function evaluator is measured in two ways. In Figure~\ref{fig:greenfun-err} (left panel), we plot the residual $|(\Delta + k^2(1+q))G|$ which should be $0$ away from the point source $\by_{0}$. We also verify the accuracy of the Green's function by using Green's identities. Suppose $u(\bx) = G_{q}(\bx, \by_{0})$, where $G_{q}$ is the Green's function with material property $q$, then $u$ for $|x_{2}|<\lt$ satisfies
\begin{equation}\label{eq:gi}
        \int_{\tGamma_{\epsilon}} G_{q}(\bx; 0,y_{2}) \frac{\partial u}{\partial y_{1}}(0,y_{2}) dy_{2} - \int_{-\tGamma_{\epsilon}} \frac{\partial G_{q}}{\partial y_{1}}(\bx; 0, y_{2}) u(0,y_{2}) dy_{2} = \begin{cases}
        u(\bx) \quad & x_{1} < 0 \,, \\
        u(\bx)/2 \quad & x_{1} = 0 \,,\\
        0 \quad & x_{1} > 0\,.
   \end{cases}
\end{equation}
Here $\tGamma \in \cG$ is the complex contour given by
\begin{equation}
    \tGamma = t + i \psi(t) \, ,
\end{equation}
with
\begin{equation}
\psi(t) = 20 \left(\erf \left(\frac{40+t}{5}\right) - \erf \left(\frac{40-t}{5}\right)\right)\,,
\end{equation}
with the curve $\tGamma_{\epsilon}$ truncated at $|\Im(x_{2})| = 39.527$.
The error in Figure~\ref{fig:greenfun-err} (right panel) is the absolute deviation in the validity of the above identity.
\begin{figure}[h!]
    \centering
    \includegraphics[width=0.9\linewidth]{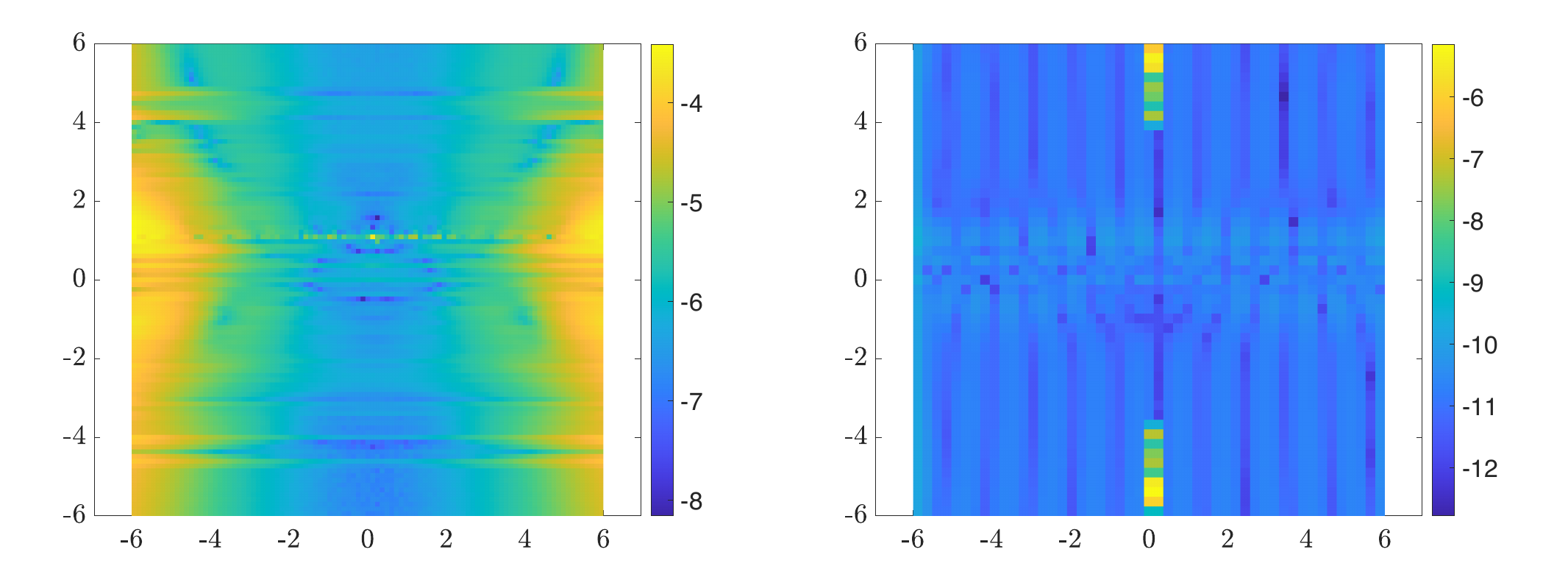}
    \caption{Left: $\log_{10}$ of the residual $(\Delta_h + k^2 (1+q_l(x_2))) G_{q_l} (\bx, \by_0)$ associated with the computed Green's function for the left material, with $\by_0 = (0.1,1)$, and 
    $\Delta_h$ the standard finite-difference approximation to the Laplacian with step-size $h = 10^{-4}$. Right: $\log_{10}$ of the computed difference between the left- and right-hand sides of \eqref{eq:gi} for $q = q_l$.}
    \label{fig:greenfun-err} 
\end{figure}

The layer potentials above and the ones appearing in the numerical solution of integral equations in the remainder of this section are discretized using the chunkIE package~\cite{chunkIE}. This package solves integral equations using a modified Nystr\"om method. It discretizes~$\tGamma_{\eps}$ by splitting it into 16th order Gauss-Legendre panels and provides utilities for the accurate evaluation of integral operators with logarithmic singularities on surface. The smooth quadrature rule is used to evaluate the integrals off-surface.

Next we turn our attention to transmission problem. Throughout, we represent our solution $u$ to the PDE \eqref{eqn:the_pde} as
\begin{align*}
    u (\bx) = 
        u_i^{l,r} (\bx)+ u^{l,r} (\bx) \, ,\quad \textrm{for } \bx \in \Omega_{l,r},
\end{align*}
where $u_i^{l,r}$ are the prescribed incoming fields, and the unknown functions $u^{l,r}$ satisfy \eqref{eqn:the_reduced_pde}. We then use the above Green's function evaluators to discretize~\eqref{eqn3.15.10}, which allows us to obtain the densities $\sigma_{\epsilon}$ and $\tau_{\epsilon}$. The functions $u^{l,r}$ are evaluated by discretizing
$$
u^{l,r}(\bx) = \cS^{l,r}_{\tGamma_{\epsilon}}[\sigma_{\epsilon}](\bx) - \cD^{l,r}_{\tGamma_{\epsilon}}[\tau_{\epsilon}](\bx) 
$$
using the smooth Gauss-Legendre quadrature rule for $x_{2} \neq 0$.

To test the accuracy, we set
\begin{equation*}
    u_i^{l,r}(\bx) = \int_{\bbR^2} G^{l,r}(\bx,\by)\,f^{l,r}(\by) \,{\rm d}\by \, ,\quad \textrm{for } \bx \in \Omega_{l,r} \,,
    \end{equation*}
where $f^{l,r}(\by) := \delta (\by - \by^{l,r})$, with $\by^l \in \Omega_r$ and $\by^r \in \Omega_l$. These fictitious point sources produce an outgoing solution $u \in \mathcal{C}^1 (\mathbb{R}^2)$ to 
$\Delta_x u^{l,r}(\bx) + k^2(1+q_{l,r}(x_2)) u^{l,r}(\bx) = 0$ in $\Omega_{l,r}$, meaning that $u \equiv 0$. 
The right panel of Figure \ref{fig:ps_24} illustrates the error via the computed absolute value of $u$ when $\by^l = (1, 0.5)$ and $\by^r = (-1.5, 1)$. The corresponding incoming fields $u_i^{l,r}$ are plotted in the left panel.

Having tested our solver, we now use it to simulate scattering with a few choices of physically relevant data
. In Figure~\ref{fig:ps_23}, we plot the solution to the PDE \eqref{eqn:the_pde} due to a point source $f (\bx) := \delta (\bx - \by_0)$ located at $\by_{0} = (-1.5,1)$. In Figure~\ref{fig:wg_21}, we plot the solution of the PDE corresponding to 
waveguide-mode incident fields, 
\begin{align*}
    (u_i^{l} (\bx), u_i^r (\bx)) = (v_{l,j} (x_2) e^{i\xi_{l,j} x_1}, 0)\qquad 
    \text{or}
    \qquad 
    (u_i^{l} (\bx), u_i^r (\bx)) = (0, v_{r,j} (x_2) e^{-i\xi_{r,j}^b x_1})
\end{align*}
for various waveguide mode frequencies $\xi_{l,j}$ and $\xi_{r,j}^b$.

\begin{figure}
    \centering
    \includegraphics[width=0.9\linewidth]{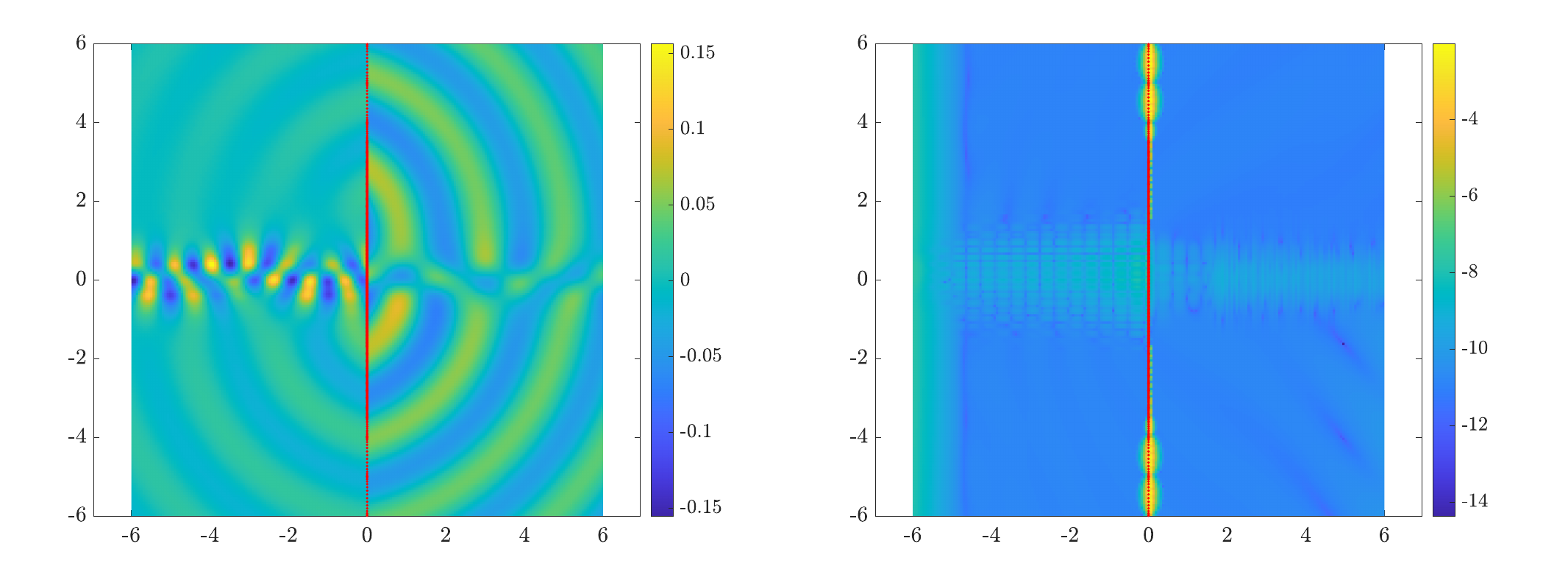}
    \caption{The left panel shows the incoming field, with $u_i^{l,r}$
      respectively corresponding to point sources at $(1,0.5)$ and $(-1.5,1)$
      for the analytic solution test. The absolute error of the solution
      measured by $\log_{10}(|u|)$ is plotted in the right panel.}
    \label{fig:ps_24}
\end{figure}

\begin{figure}
    \centering
    \includegraphics[width=0.9\linewidth]{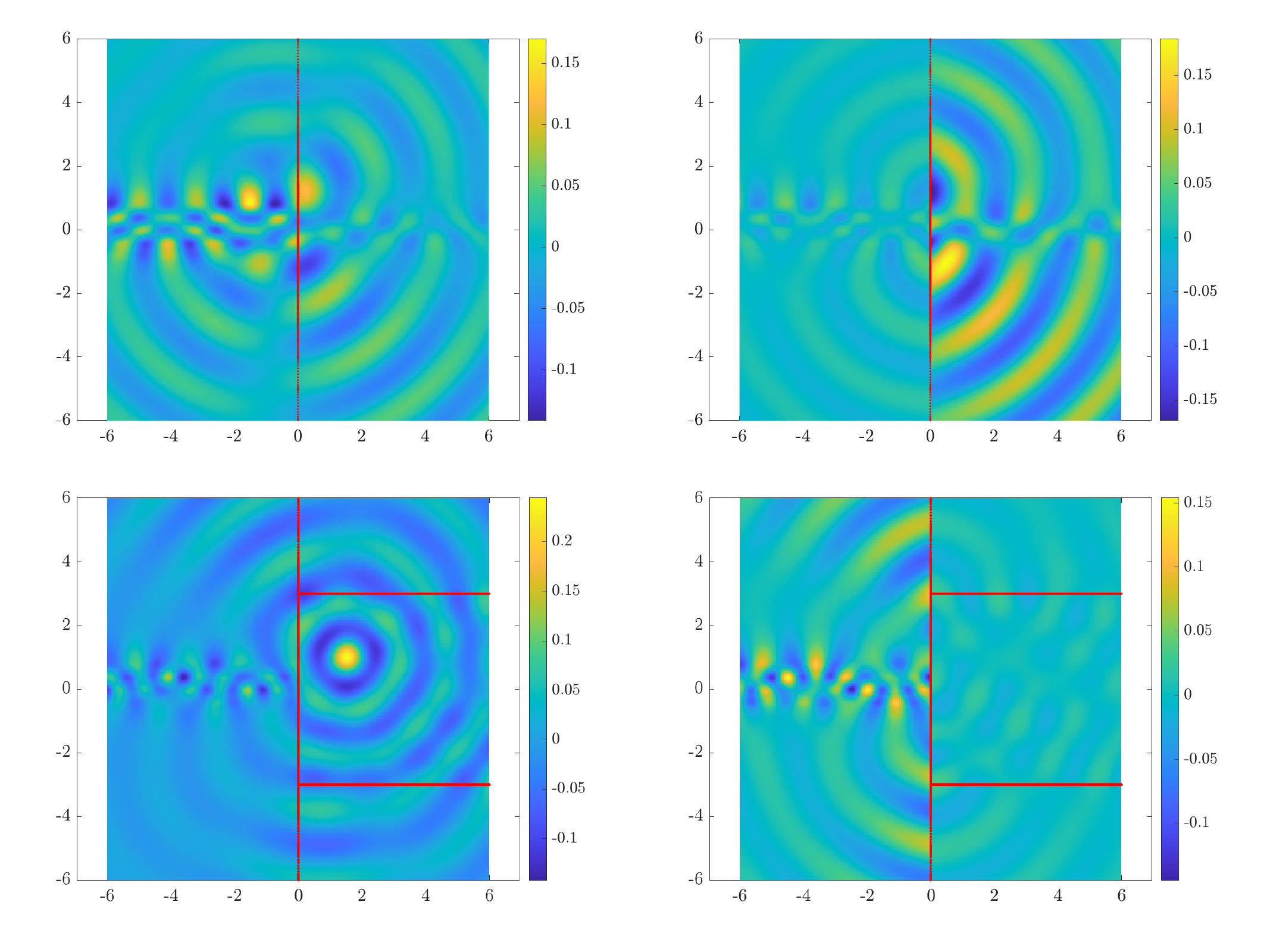}
    \caption{Top: Imaginary parts of the total and scattered fields
      corresponding to a point source at $\by_0 = (-1.5,1)$, with potential $q_r
      = q_r^a$ in the right half-plane. Bottom: Imaginary parts of the total and
      scattered fields corresponding to a point source at $\by_0 = (1.5,1)$,
      with potential $q_r = q_r^b$ in the right half-plane.}
    \label{fig:ps_23}
\end{figure}

\begin{figure}
    \centering
    \includegraphics[width=0.9\linewidth]{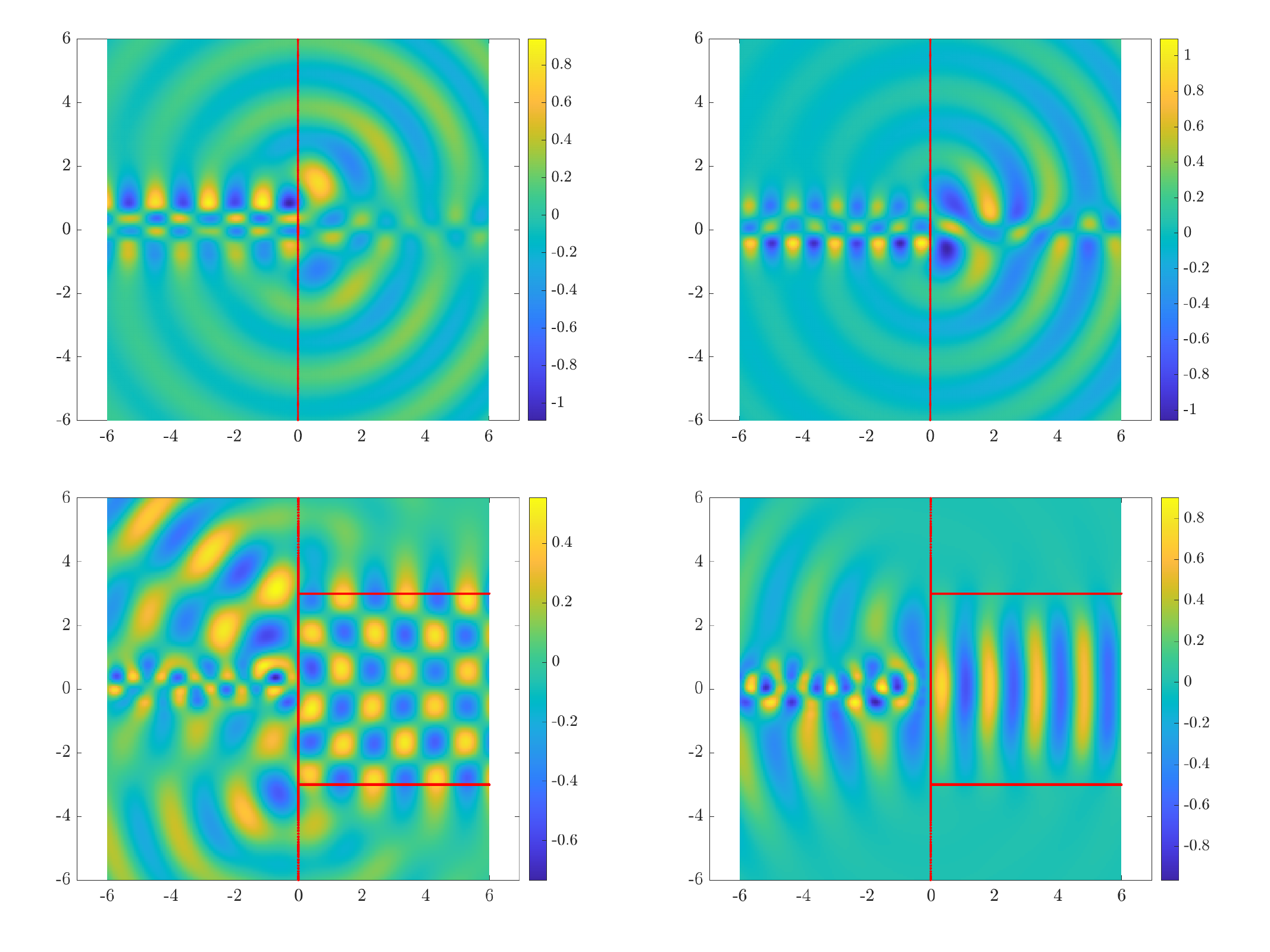}
     \caption{Imaginary parts of the solution corresponding to various incoming waveguide
       modes. Top: the potential in the right half-plane is $q_r
       = q_r^a$. The incoming waveguide mode is supported in the left
       half-plane, and has frequency $\xi_{l,1},\, \xi_{l,2}$ in the left and
       right panels, respectively.  Bottom: the potential in the right
       half-plane is $q_r = q_r^b$. The incoming waveguide mode is supported in
       the right half-plane, and has frequency $\xi_{r,1}^b,\, \xi_{r,5}^b$ in
       the left and right panels, respectively.  }
    \label{fig:wg_21}
\end{figure}

\section{Concluding remarks}
In this work, we analyze the complex scaling approach applied to boundary
integral equations for the solution of time harmonic acoustic wave scattering
from junctions of dielectric `leaky' waveguides proposed
in~\cite{epstein2023solvinga,epstein2023solvingb,epstein2024solving}. The method
in these papers involved the solution of a certain integral equation with
unknown densities supported on an infinite fictitious interface separating the
two waveguides. This led to a system of integral equations which, while Fredholm
on the interface, have slowly decaying solutions, posing significant numerical
challenges for their solution. In~\cite{goodwill2024numerical}, it was shown
through extensive numerical experiments that complex scaling is an effective
approach for solving integral equations of this type. This work provides a
rigorous justification for the use of the complex scaling approach in this
context. In particular, we showed that the outgoing nature of both the Green's
functions and the data can be used to analytically extend integral operators to
appropriately chosen totally real submanifolds of $\mathbb{C}^{2}$ on which the
the data and the solutions to the integral equations decay rapidly. Moreover, we
also showed that these contours lead to efficient discretizations of the
integral equation and allow for truncation at $O(\log{(1/\epsilon)})$
wavelengths outside the support of the waveguides to obtain an $\epsilon$
accurate solution to the PDE.

There are many natural extensions of this work, both numerically and
analytically. Firstly, in the numerical examples presented in this paper, the
dominant computational cost was the efficient evaluation of Green's functions
for bi-infinite waveguides. While there exist many methods for computing the
Green's functions for the piecewise constant case~\cite{cho2021adapting}, obtaining an
efficient evaluator for generic perturbations $q$ in the limits $x_{2} \to
y_{2}$, and $|x_{1}| \to \infty$ remains an open problem. Secondly, we believe
that the analytical framework presented here can be extended to dielectric leaky
waveguides for time-harmonic electromagnetic scattering, periodically varying
waveguides and scattering from gratings, and interface problems arising in the
study of topological insulators (see
\cite{dirac_waveguide,kleingordon_waveguide} for an alternate approach). In
principle, it should also be possible to extend it to three dimensions where
quadrature \cite{bowei} and fast algorithms are available. Finally, an
interesting extension, which the current work does not directly address, would
be the analysis of Fredholm properties of junctions that meet obliquely at the
interface $\{x_{1} = 0\}$; this would enable the analysis and the development of
numerical methods for photonic device components such as `Y-couplers'. Provided
that all the wave guides are asymptotic to parallel rays, this more realistic problem is a
relatively compact perturbation of the problems analyzed herein.

\section{Acknowledgments}
The authors would like to thank the American Institute of Mathematics and, in particular, John Fry for hosting them on
Bock Cay during the SQuaREs program, where parts of this work were completed. The Flatiron institute is a division of
the Simons foundation.

\appendix

\section{Properties of the partial Fourier transforms of $w^{l,r}$}
\label{app:a}
Recall that the kernels $w^{l,r}$ are functions of $x_{1} - y_{1}$ and thus their analytical properties can be understood by
studying their partial Fourier transform $\tilde{w}^{l,r}$, given by
\begin{equation}
w^{l,r}(\bx; \by) = \int_{\gamma} e^{i \xi (x_{1} - y_{1})} \tilde{w}^{l,r}(\xi; x_{2}, y_{2}) d\xi\,, 
\end{equation} 
for an appropriately chosen contour $\gamma$. In this section, we establish
several key analytical properties of the functions $\tilde{w}^{l,r}$. For
notational convenience, we drop the superscripts $l,r$. In particular, let $q$
be a piecewise continuous function with compact support in $[-d,d]$, let $G =
G_{0} + w$ denote the outgoing Green's function of $\Delta+k^2(1+q(x_2)),$ where
$G_{0}$ is, as before, the free-space Green's function for wave number $k^2.$
Finally let $\tilde{w}$ denote the partial Fourier transform of $w$ in the $x_1$-direction.

We denote the four open quadrants of the complex plane by $Q_1,\, Q_2,\, Q_3,\, Q_4,$ i.e., $Q_1=\{z:\: \Re z>0,\Im z>0\}$ and $Q_j=e^{\frac{i(j-1)\pi}{2}}Q_1,$ for $j=2,3,4.$ For $\theta <\pi/2,$ let $S_{\theta}$ denote the union of the sectors in $\oQ_{2} \cup \oQ_{4}$ given by
\begin{equation}
    S_{\theta} = \left\{\xi: \, \frac{\pi}{2} \leq \arg{\xi} \leq \frac{\pi}{2}+\theta \right\} \cup \left\{\xi: \, -\frac{\pi}{2}\leq \arg{\xi} \leq \theta - \frac{\pi}{2}\right\}. 
\end{equation}
The analysis of $\tilde{w}$ requires integrals involving the function $\astar(\xi):=-\sqrt{\xi^2-k^2}$ for $\xi \in \mathbb{C},$ which has branch points at $\xi = \pm k.$ The following definition fixes our choice of branch cuts.
\begin{definition}
    Let~$\sqrt{\cdot}$ be chosen to have the branch cut along the
    negative real axis, such that~$\Re \sqrt{z}\geq 0$.

    We define the functions $\astar$ by
    \begin{equation}    
        \astar(\xi) = -\sqrt{i(\xi-k)}\sqrt{-i(\xi+k)}\label{eq:w_c_formula},
    \end{equation}
    which has a branch cut pointing vertically into the lower half plane at $-k$ and a branch cut pointing vertically
    into the upper half plane at $k$.  Along the real axis where $|\xi|<k,$ $\astar(\xi)=i\sqrt{k^2-\xi^2},$ and where $|\xi|>k,$ $\astar(\xi)=-\sqrt{\xi^2-k^2}.$ $\astar(\xi)$ maps $\oQ_{2}\cup\oQ_{4}$ into $\oQ_{2}$ and in $S_{\theta}$ is equal to $i \sqrt{k^2 - \xi^2}$.
\end{definition}

Returning to $\tilde{w},$ we begin by observing that for fixed $y_{2}$ and $\xi$, $\tilde{w}$ is the unique outgoing solution to the ordinary differential equation (ODE)
\begin{equation}
L_{\xi}[\tilde{w}] := (\partial^{2}_{x_{2}} - \xi^2 + k^2 (1+q)) \tilde{w} = -\frac{1}{2\astar(\xi)} k^2 q(x_{2}) e^{\astar(\xi)|x_{2} - y_{2}|};
\end{equation}
the outgoing  condition implies that $\tilde{w}'(x_2) = {\rm sign}(x_2)
\alpha_k(\xi) \tilde{w}(x_2)$ for all $|x_2| >d.$ To analyze the behavior of
solutions to this ODE, we begin by constructing the outgoing resolvent of
$L_{\xi},$ denoted by $g(\xi; x_{2}, z_{2}),$ and defined so that the solution
of $L_{\xi}[v] = f$ is given by
\begin{equation}
\label{eq:out_res}
v(x_{2}) = \int_{-\infty}^{\infty} g(\xi; x_{2}, z_{2}) f(z_{2}) dz_{2} \, ,
\end{equation}
provided $f$ decays sufficiently rapidly at $\infty$. 
For $x_{2} > d$, $L_{\xi}$ reduces to $(\partial^{2}_{x_{2}} - \xi^2 + k^2)$ and its two linearly independent solutions are given by $e^{\pm \astar(\xi) x_{2}}$.  
Following the analysis in~\cite[Appendix A]{epstein2023solvinga}, the outgoing resolvent can be expressed explicitly as
\begin{equation}
g(\xi; x_{2}, z_{2}) := \frac{1}{W(\xi)} \left[ \tilde{u}_{+}(\xi; x_{2}) \tilde{u}_{-}(\xi; z_{2}) \chi_{z_{2} {\le} x_{2}} + \tilde{u}_{-}(\xi; x_{2}) \tilde{u}_{+}(\xi; z_{2}) \chi_{z_{2} > x_{2}}  \right] \,,
\end{equation}
where $\tilde{u}_{\pm}$ are particular solutions of $L_{\xi}$ satisfying
    \begin{align}\label{eq:basic}
    \begin{split}
        L_\xi \tilde{u}_\pm (\xi;x_2)= 0, \qquad &x_2 \in \mathbb{R}\\
        \tilde{u}_\pm (\xi;x_2) = e^{\pm\astar (\xi) x_2}, \qquad &\pm x_2 > d\,,
    \end{split}
    \end{align}
$W_{\xi}$ is their Wronskian given by
    \begin{align*}
        W(\xi) := \tilde{u}_- (\xi;x_2) \partial_{x_2} \tilde{u}_+ (\xi;x_2) - \tilde{u}_+ (\xi;x_2) \partial_{x_2} \tilde{u}_- (\xi; x_2) \,,
    \end{align*}
 and $\chi_{S}$  is the indicator function of the set $S$.
 
In the following lemma, we show that the functions $\tilde{u}_{\pm}$ are uniquely defined, and describe their region of analyticity.

 \begin{lemma}\label{lemma:basic}
    For every $\xi \in \mathbb{C}$, $\tilde{u}_{\pm}$ satisfying equation~\eqref{eq:basic} are uniquely defined. 
    Moreover, when viewed as a function of $\zeta = \astar(\xi)$, the functions $\tilde{u}_{\pm}$, and $\partial_{x_{2}} \tilde{u}_{\pm}$ are entire in $\zeta$.
\end{lemma}
\begin{proof}
The proof is standard, see \cite{deift1979inverse} for example. We sketch the
proof for $\tilde{u}_-.$ Let
$$v (\xi; x_2) := (e^{\astar(\xi)x_2}\tilde{u}_- (\xi;
x_2), \partial_{x_2} e^{\astar(\xi)x_2} \tilde{u}_- (\xi; x_2))^T$$
and
$$
B(\zeta; x_2) := \begin{pmatrix} 0 & 1\\ -k^2 q(x_2) & 2 \zeta
    \end{pmatrix} \, .$$ 
 Then \eqref{eq:basic} is equivalent to
    $$\partial_{x_2} v (\xi; x_2) =
    B(\astar(\xi); x_2) v (\xi; x_2), \quad v (\xi; -d) = \left(1, 0\right),$$
    which in integral form is given by
    \begin{align}\label{eq:vA}
    v(\xi; x_2) = v(\xi; -d) + \int_{-d}^{x_2} B(\astar (\xi); z) v(\xi; z) {\rm d} z.
    \end{align}
    As in\cite{deift1979inverse}, we iterate the above Volterra integral equation, and use the initial data, to obtain
    \begin{align}\label{eqn:neum}
        v(\xi; x_2) = (1,0)^T + \sum_{j=1}^\infty \fB_{\astar (\xi)}^{(j)}[(1,0)^T] (x_2),
    \end{align}
    with 
    \begin{align}
       \fB_{\astar (\xi)}^{(j)} [w] (x_2):= \int_{-d}^{x_2}\int_{-d}^{z_1}\cdots \int_{-d}^{z_{j-1}} B(\astar (\xi); z_1) \cdots B(\astar (\xi); z_j) w(z_j)\,\dd z_j \cdots \dd z_1.
    \end{align}
    Noting that $\|B(\zeta,z)\| \le \sqrt{4|\zeta|^2+k^2 \|q\|_\infty + 1}$ we find that 
    $$|\fB_{\astar (\xi)}^{(j)} [w] (x_2)| \le \frac{[(x_2+d)\,\sqrt{4|\zeta|^2+k^2 \|q\|_\infty + 1}]^j}{j!}\|w\|_\infty,$$
    where the bound holds for both components. It is then straightforward to show that the sum in (\ref{eqn:neum}) converges uniformly for $\zeta$ bounded with
    $$|v(\xi;x_2)| \le \exp((x_2+d)\,\sqrt{4|\zeta|^2+k^2 \|q\|_\infty + 1}).$$
    More refined bounds can be found for the region with $x_2 >d$ by restarting the Volterra integral equation at $x_2=d.$ For $\Re \zeta \le 0, \zeta \neq 0,$ at the expense of a multiplicative constant independent of $\zeta$ and $x_2,$ we can replace $x_2$ in the above estimate by $\min (x_2,d).$ Since each term in (\ref{eqn:neum}) is clearly analytic in $\zeta,$ and the convergence of the sum is uniform on compact sets, the analyticity of $v(\xi;x_2)$ follows immediately.

    Finally, we observe that $\tilde{u}_+ (\xi; x_2) = \tilde{u}_-^\sharp (\xi; -x_2)$, where $\tilde{u}_-^\sharp$ is the solution $\tilde{u}_-$ of \eqref{eq:basic} with $q(x_2)$ replaced by $q(-x_2)$. Thus we have established the analyticity of $\tilde{u}_\pm$ and $\partial_{x_2} \tilde{u}_\pm$, and the proof is complete.
\end{proof}

Since the outgoing resolvent $g$ has $W(\xi)$ in its denominator, the zeros of the Wronskian correspond to poles of $g$. In the following lemma, we show that the Wronskian has finitely many zeros on the real axis, and provide a uniform lower bound on $W$ outside of a ball in $\oQ_{2} \cup \oQ_{4}$. We denote this region by $\xiset :=\{\oQ_2 \cup \oQ_4 : |\xi| \ge \ximin\}$,  assuming that $q$ is not identically zero, we set $\ximin := k \sqrt{1+\|q\|_\infty}$.
\begin{lemma}
\label{lemma:w-est}
The Wronskian $W(\xi)$ has a finite number of zeros $\{\pm \xi_j\}$, all of which are real and have absolute value in the interval $[k, \ximin)$. 
Moreover, there exists a constant $c>0$ such that $|W(\xi)| \ge c|\xi|$ uniformly in $\xi \in \xiset$. 
\end{lemma}
The proof of this result is classical; we include it only for completeness. In particular, the bound \eqref{eq:bd_phi} follows immediately from \cite[Lemma 1 (i)]{deift1979inverse}.
\begin{proof}
The self-adjointness of the operator $\partial^2_{x_2} +k^2 (1 + q(x_2))$ implies that all the zeros of $W$ must lie on the real or imaginary axes.
    However, the imaginary axis and real interval $(-k, k)$ cannot contain zeros of $W$, as these zeros would correspond to nontrivial outgoing solutions $u$ to $L_\xi u = 0$ (see, e.g. \cite[Proposition 1.4]{tang2007potential}).
    For $\xi\in \mathbb{R}$ satisfying $|\xi| > k$, $W(\xi) = 0$ if and only if $\tilde{u}_+ (\xi;\; \cdot)$ is an eigenfunction of $\partial^2_{x_2} + k^2 q (x_2)$ with eigenvalue $\xi^2 - k^2$.
    Since $q$ is compactly supported and piecewise continuous, 
    the operator $\partial^2_{x_2} + k^2 q(x_2)$ has a finite number of positive eigenvalues, all of which are less than $k^2 \|q\|_\infty$. Indeed, by the Courant-Fischer min-max principle \cite{RS4}, each positive eigenvalue of $\partial^2_{x_2} + k^2 q(x_2)$ is bounded above by a distinct positive eigenvalue of $\partial^2_{x_2} + k^2 \|q\|_\infty \chi_{x_2 \in \supp q}$ (see also \cite[Corollary 9.43]{Teschl}). Thus we have shown that $W$ has a finite number of zeros, all of which lie in $(-\ximin, -k] \cup [k, \ximin) \subset \mathbb{R}$. 

    It remains to verify the growth condition for $W$. 
    Observe that for any $z<-d,$ $\tilde{u}_-(\xi;x_2)$ satisfies the integral equation 
    $$\tilde{u}_-(\xi;x_2) = e^{-\astar(\xi) x_2} - \int_{z}^{x_2} \frac{1}{\astar(\xi)}\sinh(\astar(\xi)(x_2-s))k^2q(s)\tilde{u}_-(\xi;s)\,{\rm d}s,$$
    where ${\chi}_{s>0} \sinh(\astar(\xi) s)/\astar(\xi)$ is the causal Green's function. In particular, setting $$\phi_-(\xi;x_2) := e^{\astar(\xi)x_2}\tilde{u}_-(\xi,x_2),$$ it follows straightforwardly that
    \begin{align}\label{eq:straightforward}
    \phi_-(\xi;x_2) = 1 - \int_{-\infty}^{x_2}\frac{e^{2\astar(\xi)(x_2-s)}-1}{2\astar(\xi)}k^2q(s)\phi_-(\xi;s)\,{\rm d}s =: 1 + \causal [\phi_- (\xi; \;\cdot)] (x_2).
    \end{align}
    Since $\Re \astar (\xi) \le 0$ for all $\xi \in \oQ_2 \cup \oQ_4$ and $\supp (q) \subseteq [-d,d]$,
    we have
    $$|\causal [\phi_- (\xi; \;\cdot)] (x_2)| \le \frac{C}{|\astar (\xi)|} \|\phi (\xi; \; \cdot) \|_{L^\infty[-d,d]}, \qquad \xi \in \oQ_2 \cup \oQ_4, \quad x_2 \in \mathbb{R}.$$
    Hence for all $\xi \in \oQ_2 \cup \oQ_4$ sufficiently large, $I-\causal$ is boundedly invertible on $L^\infty (\mathbb{R})$ with $\| (I - \causal)^{-1} - I \| \le C/|\xi|$. It then follows from \eqref{eq:straightforward} that
    \begin{align}\label{eq:bd_phi}
        \left|\phi_- (\xi; x_2) - 1\right| \le \frac{C}{|\xi|}, \qquad \xi \in \xiset, \quad x_2 \in \mathbb{R}.
    \end{align}
    Now, when $x_2 > d$,
    \begin{align}\label{eq:tildeu_minus}
        \tilde{u}_- (\xi;x_2) = a(\xi) e^{x_2 \astar (\xi)} + b(\xi) e^{-x_2 \astar (\xi)},
    \end{align}
    and thus
    \begin{align}\label{eq:W}
        W (\xi) = b(\xi) \astar (\xi), \qquad b(\xi) = \phi_- (\xi; x_2) - a(\xi) e^{2 x_2 \astar (\xi)}.
    \end{align}
    By \eqref{eq:bd_phi},
    \begin{align}\label{eq:b_bd}
        |b(\xi) - 1| \le \inf_{x_2 \in \mathbb{R}} \left( |\phi_- (\xi; x_2) - 1| + \left| a(\xi) e^{2 x_2 \astar (\xi)}\right|\right) \le \frac{C}{|\xi|} + |a(\xi)|\inf_{x_2 \in \mathbb{R}}\left|e^{2 x_2 \astar (\xi)}\right| = \frac{C}{|\xi|}
    \end{align}
    uniformly in $\xi \in \xiset \cap \{\Re \xi \ne 0\}$.
    Indeed, for any $\xi \in \xiset$ with $\Re\xi \ne 0$, we have 
    $\Re \astar (\xi) < 0$, meaning that the last infimum in the above equation vanishes. 
    Recalling Lemma \ref{lemma:basic} and the expression for $W$ in \eqref{eq:W}, we see that $b(\xi) = W(\xi)/\astar (\xi)$ is an analytic function of $\xi$ in the strip $\{|\Re \xi| < k/2\}$, hence the bound \eqref{eq:b_bd} extends to all $\xi \in \xiset$.
    Since $W$ does not vanish in $\xiset$, we have thus established the existence of a constant $c>0$ such that $|W (\xi)| \ge c |\xi|$ for all $\xi \in \xiset$.
\end{proof}

The following corollary is a straightforward consequence of Lemmas~\ref{lemma:basic} and~\ref{lemma:w-est}.
\begin{corollary}
The outgoing resolvent $g$ when viewed as a function $\zeta = \astar(\xi)$ is analytic in $\zeta$ with the exception of $\zeta_j = \astar(\xi_{j})$.
\end{corollary}

In the following lemma, we turn our attention to properties of solutions to $L_{\xi} v = f$. In the remainder of the section, let 
$\rho_{a,b}$, for $a<b$, denote the following piecewise linear function 
\begin{align}\label{eq:ls_def}
    \ls_{a,b}(y) := \begin{cases}
        0, & y < a\\
        y-a, & a \le y \le b\\
        b-a, & y > b
    \end{cases} \, .
\end{align}
\begin{lemma}\label{lemma:decay}
    Let $m \ge 0$ and suppose $f: \mathbb{C} \times \mathbb{R}^2 \to \mathbb{C}$ satisfies
    \begin{align}\label{eq:decay_g}
        |f(\xi; x_2, y_2)|\le \begin{cases}
            C x_2^m e^{\Re \astar (\xi) (x_2 - d + |y_2-d|)}, & x_2 > d\\
            C e^{\Re \astar (\xi) |x_2 - y_2|}, & -d \le x_2 \le d\\
            C|x_2|^m e^{\Re \astar (\xi) (|x_2 + d| + |y_2+d|)}, & x_2 < -d
        \end{cases}
    \end{align}
    uniformly in $(\xi; x_2, y_2) \in \xiset \times \mathbb{R}^2$. Then the solution $v$ to
    $
        L_\xi v (\xi;x_2,y_2) = f(\xi; x_2,y_2)
    $
    satisfies
    \begin{align}\label{eq:decay_v}
        |v(\xi; x_2, y_2)| \le
        \begin{cases}
            Cx_2^m e^{\Re \astar (\xi) (x_2 - d + |y_2-d|)} \left(\frac{1}{|\xi|^2}+\frac{\ls_{-d,d}(-y_2)+x_2-d}{|\xi|}\right), & x_2 > d\\
            Ce^{\Re \astar (\xi) |x_2-y_2|} \left(\frac{1}{|\xi|^2}+\frac{\ls_{x_2,d}(y_2)+\ls_{-x_2,d}(-y_2)}{|\xi|}\right), & -d \le x_2 \le d\\
            C|x_2^m| e^{\Re \astar (\xi) (|x_2 + d| + |y_2+d|)} \left(\frac{1}{|\xi|^2}+\frac{\ls_{-d,d}(y_2)+|x_2+d|}{|\xi|}\right), & x_2 < -d
        \end{cases}
    \end{align}
    uniformly in $(\xi; x_2, y_2) \in \xiset \times \mathbb{R}^2$. Moreover, 
    if $f(\xi; x_2, y_2) = \mathfrak{f} (\astar (\xi); x_2, y_2)$ for some function $\mathfrak{f} (\zeta; x_2, y_2)$ which is analytic in $\zeta \in \mathbb{C} \setminus \{\astar (\xi_j)\}$, then $v(\xi; x_2, y_2) = \mathfrak{v} (\astar (\xi); x_2, y_2)$ with $\mathfrak{v} (\zeta; x_2, y_2)$ analytic in $\zeta \in \mathbb{C}\setminus \{\astar (\xi_j)\}$.
\end{lemma}

\begin{proof}
The analyticity of $v$ in $\zeta$ follows from~\eqref{eq:out_res}, and the analyticity of $f$ and the outgoing resolvent $g$. 

It remains to verify the decay estimates \eqref{eq:decay_v}.
Observe that \eqref{eq:out_res} implies
    \begin{align*}
        |v (\xi; x_2, y_2)| &\le C\frac{e^{\Re \astar (\xi) x_2}}{|W (\xi)|} \Big(
        \int_{-\infty}^{-d} |z_2|^me^{-\Re \astar (\xi) z_2}e^{\Re \astar (\xi) (|z_2 + d| + |y_2+d|)} {\rm d} z_2\\
        &\qquad + \int_{-d}^d |\tilde{u}_- (\xi; z_2)|e^{\Re \astar (\xi) |z_2-y_2|} {\rm d} z_2 + \int_d^{x_2} z_2^m|\tilde{u}_- (\xi; z_2)|e^{\Re \astar (\xi) (z_2 - d + |y_2-d|)}{\rm d} z_2
        \Big)\\
        &+ C\frac{|\tilde{u}_- (\xi; x_2)|}{|W(\xi)|} \int_{x_2}^\infty z_2^m e^{\Re \astar (\xi) z_2}e^{\Re \astar (\xi) (z_2 - d + |y_2-d|)} {\rm d} z_2
    \end{align*}
    uniformly in $(\xi;x_2, y_2) \in \xiset \times (d,\infty)\times\mathbb{R}$. 
    By 
    \eqref{eq:bd_phi}, the above factors of $\tilde{u}_- (\xi; z_2)$ are well approximated by $e^{-z_2 \astar (\xi)}$, up to terms that decay faster in $|\xi|$. Therefore,
    \begin{align*}
        |v (\xi; x_2, y_2)| &\le C\frac{e^{\Re \astar (\xi) x_2}}{|W (\xi)|} \Big(
        \int_{-\infty}^{-d}|z_2|^m e^{-\Re \astar (\xi) z_2}e^{\Re \astar (\xi) (|z_2 + d| + |y_2+d|)} {\rm d} z_2\\
        &\qquad + \int_{-d}^d e^{-\Re \astar (\xi) z_2}e^{\Re \astar (\xi) |z_2-y_2|} {\rm d} z_2 + \int_d^{x_2}z_2^m e^{-\Re \astar (\xi) z_2}e^{\Re \astar (\xi) (z_2 - d + |y_2-d|)}{\rm d} z_2
        \Big)\\
        &+ C\frac{e^{-\Re \astar (\xi) x_2}}{|W(\xi)|} \int_{x_2}^\infty z_2^m e^{\Re \astar (\xi) z_2}e^{\Re \astar (\xi) (z_2 - d + |y_2-d|)} {\rm d} z_2
    \end{align*}
    uniformly in $(\xi;x_2, y_2) \in \xiset \times (d,\infty)\times\mathbb{R}$.
    These integrals are estimated by
    \begin{align}\label{eq:integral_bounds}
        \begin{split}
        &\int_{-\infty}^{-d} |z_2|^m e^{-\Re \astar (\xi) z_2}e^{\Re \astar (\xi) (|z_2 + d| + |y_2+d|)} {\rm d} z_2 \le 
            \frac{C}{|\xi|} e^{\Re \astar (\xi) d} e^{\Re \astar (\xi) |y_2+d|},\\
            &\int_{-d}^d e^{-\Re \astar (\xi) z_2}e^{\Re \astar (\xi) |z_2-y_2|} {\rm d} z_2 \\
            &\qquad =
        -\frac{1}{2\Re \astar (\xi)} (e^{-\Re \astar (\xi)d} e^{\Re \astar (\xi)|y_2-d|} - e^{\Re \astar (\xi)d}e^{\Re \astar (\xi)|y_2+d|})
        +
        \ls_{-d,d} (-y_2) e^{-y_2 \Re \astar (\xi)},\\
        &\int_d^{x_2}z_2^m e^{-\Re \astar (\xi) z_2}e^{\Re \astar (\xi) (z_2 - d + |y_2-d|)}{\rm d} z_2 \le x_2^m (x_2 - d)e^{-\Re\astar (\xi) d}e^{\Re\astar (\xi)|y_2-d|},\\
        &\int_{x_2}^\infty z_2^m e^{\Re \astar (\xi) z_2}e^{\Re \astar (\xi) (z_2 - d + |y_2-d|)} {\rm d} z_2 \le \frac{C}{|\xi|} x_2^m e^{-\Re \astar (\xi) d} e^{2 \Re \astar (\xi) x_2} e^{\Re\astar (\xi)|y_2-d|},
        \end{split}
    \end{align}
    which implies that
    \begin{align*}
        |v (\xi; x_2, y_2)| \le \frac{C}{|W(\xi)|} e^{\Re \astar (\xi)(x_2 - d + |y_2 - d|)} \left( \frac{x_2^m}{|\xi|} + \ls_{-d,d} (-y_2) + x_2^m (x_2 - d) \right)
    \end{align*}
    uniformly in $(\xi;x_2, y_2) \in \xiset \times (d,\infty)\times\mathbb{R}$. The bound in \eqref{eq:decay_v} for $x_2 > d$ then follows from the growth condition $|W(\xi)| \ge c |\xi|$ established in Lemma \ref{lemma:w-est}.

    Applying this result to the function $v(\xi; -x_2, -y_2)$, we obtain the desired bound for $x_2 < -d$.
    It remains to establish \eqref{eq:decay_v} when $-d \le x_2 \le d$. To this end, we again use \eqref{eq:out_res} to write
    \begin{align*}
        |v (\xi; x_2, y_2)| &\le C\frac{|\tilde{u}_+ (\xi; x_2)|}{|W (\xi)|} \Big(
        \int_{-\infty}^{-d} |z_2|^me^{-\Re \astar (\xi) z_2}e^{\Re \astar (\xi) (|z_2 + d| + |y_2+d|)} {\rm d} z_2\\
        &\hspace{5cm}+ \int_{-d}^{x_2} |\tilde{u}_- (\xi; z_2)|e^{\Re \astar (\xi) |z_2-y_2|} {\rm d} z_2 \Big)\\
        &+ C\frac{|\tilde{u}_- (\xi; x_2)|}{|W (\xi)|}\Big( \int_{x_2}^{d} |\tilde{u}_+ (\xi; z_2)|e^{\Re \astar (\xi) |z_2-y_2|}{\rm d} z_2
        \\
        &\hspace{5cm} + \int_{d}^\infty z_2^m e^{\Re \astar (\xi) z_2}e^{\Re \astar (\xi) (z_2 - d + |y_2-d|)} {\rm d} z_2\Big),
    \end{align*}
    and apply 
    \eqref{eq:bd_phi} and the analogous bound for $\tilde{u}_+$ to obtain
    \begin{align}\label{eq:v_bd_mid}
        \begin{split}
        |v (\xi; x_2, y_2)| &\le C\frac{e^{\Re \astar (\xi) x_2}}{|W (\xi)|} \Big(
        \int_{-\infty}^{-d} |z_2|^me^{-\Re \astar (\xi) z_2}e^{\Re \astar (\xi) (|z_2 + d| + |y_2+d|)} {\rm d} z_2\\
        &\hspace{5cm}+ \int_{-d}^{x_2} e^{-\Re \astar (\xi) z_2}e^{\Re \astar (\xi) |z_2-y_2|} {\rm d} z_2 \Big)\\
        &+ C\frac{e^{-\Re \astar (\xi) x_2}}{|W (\xi)|}\Big( \int_{x_2}^{d} e^{\Re \astar (\xi) z_2}e^{\Re \astar (\xi) |z_2-y_2|}{\rm d} z_2
        \\
        &\hspace{5cm} + \int_{d}^\infty z_2^m e^{\Re \astar (\xi) z_2}e^{\Re \astar (\xi) (z_2 - d + |y_2-d|)} {\rm d} z_2\Big)
        \end{split}
    \end{align}
    uniformly in $(\xi;x_2, y_2) \in \xiset \times [-d,d]\times\mathbb{R}$. A direct calculation reveals that
    \begin{align*}
        &\int_{-d}^{x_2} e^{-\Re \astar (\xi) z_2}e^{\Re \astar (\xi) |z_2-y_2|} {\rm d} z_2\\
        &\hspace{0.5cm} = -\frac{1}{2 \Re \astar(\xi)} e^{\Re \astar (\xi) y_2} \left( e^{-2 \Re \astar (\xi) \min \{x_2, y_2\}} - e^{2\Re \astar (\xi) d}\right)\chi_{(-d,\infty)} (y_2)+ e^{-\Re \astar (\xi) y_2} \ls_{-x_2,d}(-y_2),
    \end{align*}
    where 
    $\chi_I$ is the indicator function of the set $I$.
    Since $\ls_{-x_2,d} (-y_2)$ vanishes for all $y_2 > x_2$, it follows that
    \begin{align*}
        &e^{\Re \astar (\xi) x_2}\int_{-d}^{x_2} e^{-\Re \astar (\xi) z_2}e^{\Re \astar (\xi) |z_2-y_2|} {\rm d} z_2 \\
        &\hspace{0.5cm} = -\frac{1}{2 \Re \astar(\xi)} \left( e^{\Re \astar (\xi) |x_2-y_2|} - e^{\Re \astar (\xi) (x_2+y_2+2d)}\right)\chi_{(-d,\infty)} (y_2)+ e^{\Re \astar (\xi) |x_2-y_2|} \ls_{-x_2,d} (-y_2).
    \end{align*}
    The first integral in \eqref{eq:v_bd_mid} was already controlled in \eqref{eq:integral_bounds}, while the third and fourth integrals in \eqref{eq:v_bd_mid} evaluated at $(x_2, y_2)$ are the same as the first and second evaluated at $(-x_2, -y_2)$. 
    Combining these bounds and again using that $W(\xi) \ge c |\xi|$, 
    we obtain
    \begin{align*}
        |v (\xi; x_2, y_2)| &\le \frac{C}{|\xi|^2} \left(e^{\Re \astar (\xi) |x_2 -y_2|} + e^{\Re \astar (\xi) (x_2+d+|y_2+d|)}+ e^{\Re \astar (\xi) (-x_2+d+|y_2-d|)}\right)\\
        &\hspace{5cm} + \frac{C}{|\xi|}e^{\Re \astar (\xi) |x_2 -y_2|} (\ls_{x_2,d}(y_2) + \ls_{-x_2,d}(-y_2))
    \end{align*}
    uniformly in $(\xi;x_2, y_2) \in \xiset \times [-d,d]\times\mathbb{R}$.
    Since $\pm x_2 + d + |y_2 \pm d| \ge |x_2 - y_2|$ for all $-d \le x_2 \le d$, we have established the remaining bound in \eqref{eq:decay_v} and the result is complete.
\end{proof}

We now turn our attention to the main result of this appendix which summarizes estimates on the partial Fourier transform of the resolvent kernel.

\begin{theorem}\label{prop:suitable}
Suppose $q: \mathbb{R} \to [0,\infty)$ is piecewise continuous with $\supp (q)
  \subset [-d,d]$.  For fixed $(x_{2}, y_{2}) \in \bbR^2$, the partial Fourier transform $\tilde{w}(\xi; x_{2}, y_{2})$ of the resolvent kernel has a finite set of real
  poles denoted by $\{\pm\xi_{j}:\:j=1,2,\ldots n_{p}\},$ all of which have
  modulus greater than $k.$ Furthermore, it takes the form
\begin{equation}
      \tilde{w}(\xi;x_2,y_2) =
            e^{\astar(\xi)(|x_2\mp d|+|y_2 \mp d|)} A_{\pm} (\xi, y_2), \qquad A_{\pm} (\xi,y_2) = \begin{cases}
                A_{\pm,+}(\xi), & y_2 > d\\
                A_{\pm,0}(\xi, y_2), & |y_2| \le d\\
                A_{\pm,-}(\xi), & y_2 < -d
            \end{cases}       \label{eq:waveguide_def}
\end{equation}
for all $\pm x_2 > d$, where the $A_{\pm}$ are of the form
\begin{align}\label{eq:B_pm}
    A_\pm (\xi, y_2) = \frac{1}{\astar (\xi)}B_{\pm} (\astar (\xi), y_2), \qquad \xi \in \mathbb{C}, \quad y_2 \in \mathbb{R},
\end{align}
with the functions $B_\pm (\zeta, y_2)$ analytic in $\zeta \in \mathbb{C} \setminus \{ \astar (\xi_j) \}$. 
For any $j \ge 0$, the functions $\partial^{j}_\zeta B_\pm (\zeta, y_2)$ are uniformly bounded over compact subsets of $(\mathbb{C} \setminus \{\astar ( \xi_j)\})\times \mathbb{R}$.
Moreover, there exists a constant $\ximin > 0$ such that,
\begin{align}\label{eq:decay_A}
    |\partial_{\xi}^jA_{\pm}(\xi,y_2)| \leq \frac{C_j}{|\xi|^3} \left( 1 + |\xi| \ls_{-d,d} (\mp y_2) \right), \qquad \xi \in \xiset, \quad y_2 \in \mathbb{R}.
\end{align}
Finally, when $-d \le x_2 \le d$ and $-d \le y_2 \le d$,
\begin{align*}
    \tilde{w}(\xi;x_2,y_2) =
            e^{\astar(\xi)|x_2 - y_2|} A_{0,0} (\xi, x_2,y_2),
\end{align*}
where $A_{0,0} (\xi, x_2, y_2) = B_{0,0} (\astar (\xi), x_2, y_2)/\astar (\xi)$ with $B_{0,0} (\zeta, x_2, y_2)$ analytic in $\zeta \in \mathbb{C} \setminus \{\astar (\xi_j)\}$ and $\partial^j_\zeta B_{0,0} (\zeta, x_2, y_2)$ uniformly bounded over compact subsets of $(\mathbb{C} \setminus \{\astar ( \xi_j)\})\times \mathbb{R}$. For any $j \in \mathbb{N}_0$,
\begin{align}\label{eq:decay_A_00}
    |\partial_{\xi}^jA_{0,0}(\xi,x_2,y_2)| \leq \frac{C_j}{|\xi|^3} \left( 1 + |\xi| (\ls_{x_2,d} (y_2) + \ls_{-x_2, d}(-y_2))\right)
\end{align}
uniformly in $\xi \in \xiset$ and $(x_2, y_2) \in [-d,d]^2$.
\end{theorem}

\begin{proof}
    Recall that the function $\tilde{w} (\xi; x_2, y_2)$ satisfies $L_\xi \tilde{w}(\xi;x_2,y_2) = f(\xi; x_2, y_2)/\astar (\xi)$, where
    \begin{align*}
        f(\xi; x_2, y_2) := -\frac{1}{2} k^2 q(x_2) e^{\astar (\xi) |x_2-y_2|}
    \end{align*}
    satisfies the assumptions of Lemma \ref{lemma:decay} with $m=0$. It follows that $\astar (\xi) \tilde{w} (\xi; x_2, y_2)$  
    is analytic in $\astar (\xi) \in \mathbb{C} \setminus \{\astar (\xi_j)\}$.
    
    Since $q$ is compactly supported in $[-d,d]$, it follows from \eqref{eq:out_res} and the definition of $f$ that $\tilde{w} (\xi;x_2,y_2) e^{-(|x_2 \mp d| + |y_2 \mp d|)\astar (\xi)}$ is independent of $\pm x_2 > d$, and constant in $y_2$ over the intervals $y_2 > d$ and $y_2 < -d$. We conclude that 
    \begin{align}\label{eq:w_A}
        \tilde{w}(\xi;x_2,y_2) =
            e^{\astar(\xi)(|x_2\mp d|+|y_2 \mp d|)}A_{\pm}(\xi, y_2), \qquad \pm x_2 > d, \quad y_2 \in \mathbb{R},
    \end{align}
    where the $A_{\pm} (\xi, y_2)$ are constant in $y_2$ over the intervals $y_2 > d$ and $y_2 < -d$, and are 
    of the form \eqref{eq:B_pm} with the $\partial^{j}_\zeta B_\pm (\zeta, y_2)$ locally bounded by continuity. Lemma \ref{lemma:decay} then implies that
    \begin{align*}
        |A_{\pm}(\xi,y_2)| \leq \frac{C}{|\xi|^3} \left( 1 + |\xi| \ls_{-d,d} (\mp y_2) \right), \qquad    \xi \in \xiset , \quad y_2 \in \mathbb{R}.
    \end{align*}
    The analogous properties of $A_{0,0}$ follow from parallel arguments, thus we have verified
    that $\astar (\xi) \tilde{w} (\xi; x_2, y_2)$ satisfies the bound in \eqref{eq:decay_v}.
    
    It remains to verify that derivatives of $A_\pm$ and $A_{0,0}$ satisfy the same decay estimate. 
    With $$\tilde{w}^{(j)}(\xi;x_2,y_2) := \partial^j_\xi \tilde{w}(\xi;x_2,y_2),$$ we have that $L_\xi \tilde{w}^{(j)}(\xi;x_2,y_2) = f_j (\xi;x_2, y_2)/\astar (\xi)$, where
    \begin{align*}
        \frac{f_j (\xi;x_2, y_2)}{\astar (\xi)} := 
        \partial^j_\xi \frac{f(\xi; x_2, y_2)}{\astar (\xi)} + 2j\xi \tilde{w}^{(j-1)} (\xi;x_2, y_2) + 2 \binom{j}{2}\tilde{w}^{(j-2)} (\xi;x_2, y_2).
    \end{align*}
    Since $\partial^j_\xi (f(\xi; \; \cdot, y_2)/\astar (\xi))$ vanishes outside $[-d,d]$ and satisfies     
    \begin{align*}
        |\partial^j_\xi (f (\xi; x_2, y_2)/\astar (\xi))| \le C (1+|y_2|)^j\frac{1}{|\xi|}e^{-\Re \astar (\xi) |x_2 - y_2|}, \qquad (\xi; x_2, y_2) \in \xiset \times [-d,d]\times \mathbb{R},
    \end{align*}
    the function $(1+|y_2|)^{-j}|\xi|\partial^j_\xi (f (\xi; x_2, y_2)/\astar (\xi))$ is bounded by the right-hand side of \eqref{eq:decay_g} with $m=0$. 
  An inductive argument shows that $f_{j}/\astar(\xi)$ satisfies \eqref{eq:decay_g} with $m=j-1$, and thus $\tilde{w}^{(j)}$ satisfies \eqref{eq:decay_v} with $m = j -1$. 
        It then follows from \eqref{eq:w_A} that
    \begin{align*}
        |\partial_{\xi}^jA_{\pm}(\xi,y_2)| \leq \frac{C_j}{|\xi|^3} \left( 1 + |\xi| \ls_{-d,d} (\mp y_2) \right)(1+|y_2|)^j, \qquad \xi \in \xiset, \quad y_2 \in \mathbb{R}.
    \end{align*}
    Since $A_\pm (\xi, \cdot)$ is constant in the intervals $(-\infty,-d)$ and $(d,\infty)$, the factor $(1+|y_2|)^j$ can be absorbed into $C_j$. Similarly, the definition of $A_{0,0}$ implies \eqref{eq:decay_A_00}
    and the proof is complete.
\end{proof}

\begin{corollary}
\label{cor:tw-analyticity}
For all $\xi \neq \astar(\xi_{j})$, the function $\tilde{w}(\xi; x_{2}, y_{2})$ admits an analytic extension to $x_{2}, y_{2} \in \GammaC \times \GammaC$.
\end{corollary}
\begin{proof}
The result is an immediate consequence of equation~\eqref{eq:waveguide_def}.
\end{proof}

\section{Kernel estimates}\label{sec:kern_est}
Having established properties of the partial Fourier transforms
$\tilde{w}^{l,r}$, we now turn our attention to the analytical properties of
$w^{l,r}$. The purpose of the section is twofold: a) to establish that the
kernels $k_{C}$, and $k_{D}$ defined in terms of $w^{l,r}$ lie in the required
$\cA_{\alpha}$ spaces, and b) obtain bounds for the kernels
$w^{l,r}(x_{1},x_{2}; 0, y_{2})$ that are uniform in $x_{1} \in \mathbb{R}$,
$|x_{2}|\leq \lt$, and $y_{2} \in \GammaC$.  These estimates are essential for
the proofs of Theorems~\ref{thm:trunc_sol} and~\ref{thm:kern_decay}.

Recall that
\begin{equation}
\label{eq:IFT-appb}
w^{l,r}(\bx; \by) = \int_{\gamma_{F}} e^{i \xi (x_{1} - y_{1})} \tilde{w}^{l,r}(\xi; x_{2}, y_{2}) d\xi\,, 
\end{equation}
and that
\begin{equation}\label{eqnB.2.31}
\begin{aligned}
k_{D}(x_{2},y_{2}) &= w^{r}(0,x_2;0,y_2)-w^{l}(0,x_2;0,y_2) = \int_{\gamma_{F}} (\tilde{w}^{r}(\xi; x_{2},y_{2}) - \tilde{w}^{l}(\xi; x_{2}, y_{2})) d\xi \,, \\
k_{C}(x_{2},y_{2}) &= \frac{\partial^2}{\partial x_1\partial y_1} \left(w^{r}(0,x_2;0,y_2)-w^{l}(0,x_2;0,y_2) \right) =  \int_{\gamma_{F}} \xi^2(\tilde{w}^{r}(\xi; x_{2},y_{2}) - \tilde{w}^{l}(\xi; x_{2}, y_{2})) d\xi \,.
\end{aligned}
\end{equation}

In the following lemma, we establish the analyticity of $w^{l,r}$.
\begin{lemma}\label{lem:kern_analytic}
    The formula~\eqref{eq:IFT-appb} defines $w^{l,r}(x_1,x_2;y_1,y_2)$ for any~$x_2,y_2\in \GammaC$ and~$x_1,y_1\in \bbR$. Further, $w^{l,r}(x_1,x_2;y_1,y_2)$ and their~$x_1$ and~$y_1$ derivatives are  analytic functions of~$x_2$ and~$y_2$ in the interiors of~$\GammaC$.
%
\end{lemma}
\begin{proof}
    From Corollary~\ref{cor:tw-analyticity}, it is known that~$\tilde w^{l,r}(\xi;x_2,y_2)$ can be analytically extended as a function of $x_2$ and $y_2$  from the real line to~$\GammaC$.  The result then follows from a combination of dominated convergence theorem, Morera's theorem and the estimates in Theorem~\ref{prop:suitable}. 
%
\end{proof}

This lemma and~\eqref{eqnB.2.31} lead easily to the following corollary.
\begin{corollary}
\label{cor:analytic-kcd}
The kernels $k_{C}(x_2,y_2)$ and $k_{D}(x_2,y_2)$ extend analytically to $x_{2},
y_{2}$ in the interior of $\GammaC.$
\end{corollary}

Next we turn our attention to the analysis of the asymptotics of $k_{C}$, and
$k_{D}$. The absence of the $e^{i \xi (x_{1} - y_{1})}$ term in their
definitions simplifies this analysis in comparison to the analysis
of $w^{l,r}$.  In order to establish that $k_{D} \in \cA_{\frac{3}{2}}$, and
$k_{C} \in \cA_{\frac{1}{2}}$, all that remains to be shown are the asymptotic
estimates~\ref{eq:cc-est} and~\ref{eq:cm-est}. Both of these follow from a
steepest descent argument, which we formulate as follows.

\begin{lemma}[Steepest descent]\label{lem:steep_de}
    Let~$ A(\xi)$ have the following properties:
    \begin{enumerate}[label=\roman*)]
    \item $A$ is an even function of $\xi.$
        \item The function~$ A(\xi)$ is analytic on~$\bbC$, apart from the branch cuts of~$\astar$ and finitely many                             poles on the real axis, and~$\xi=0,\pm k$ are not poles.
        \item There exists $p \in \mathbb{N}$ and $c>0$ such that for each
          $\delta >0$ there is a constant $C_\delta$ such that $|{A}(\xi)| \le
          C_\delta \lp(1+|\xi|)^{(p+1)}+e^{c|\Im\xi|}\rp$ for all $\xi \in
          \mathbb{C} \setminus T_\delta$ where $T_\delta$ is the union of
          $\delta$-balls around each of the poles and branch points.
    \end{enumerate}
    We denote by~$\gamma_F$ the contour defined by~$\Im \xi = -\tanh \Re\xi$, see Figure~\ref{fig_contours}.
    For~$z$ in the first quadrant  let $I(z)$ be the function defined by the  contour integral
    \begin{equation}
        I(z) = \int_{\gamma_F} e^{\astar(\xi) z}A(\xi) d\xi.\label{eq:steep_des}
    \end{equation}
    Then, there exist constants~$\{C_j:\: j=0,1,\dots\}$  depending on~$A$  such that, for any~$n\in\bbN$,  there is a constant $K_{n}$ so that
    \begin{equation}
        \left|I(z) - \frac{e^{ik z}}{\sqrt{z}}\sum_{j=0}^n\frac{C_j }{z^{j}}  \right|\leq e^{-k \Im z}\frac{K_{n}}{|z|^{n+\frac 32}}.\label{eq:de_asym}
    \end{equation}
  as $z$ tends to infinity in the first quadrant with~$|z|>4c.$
\end{lemma}
\begin{figure}[h]
    \centering
    \includegraphics[width = 0.6\textwidth]{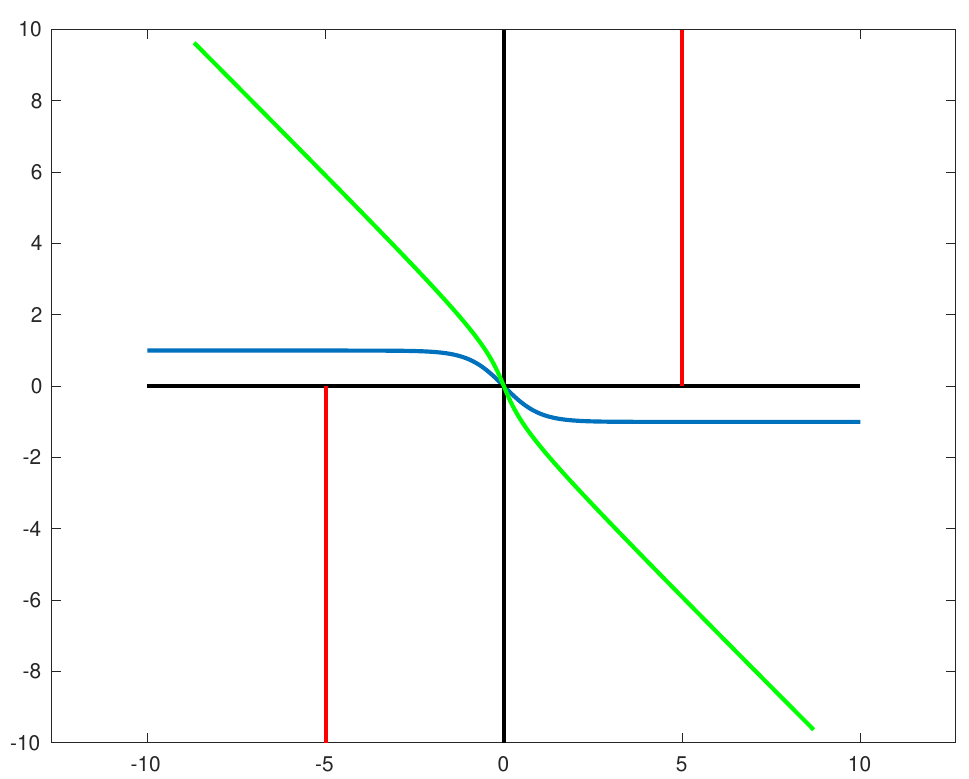}
    \caption{Plot in the $\xi$-plane showing $\gamma_F$ in blue and
      $w(t)$ in green. The cuts used in the definition of
      $\astar(\xi)$ are shown in red, for $k=5.$}\label{fig_contours}
\end{figure}
\begin{proof}
    To prove this expansion for $z$ in the first quadrant, recall that $\astar(\xi)$ maps $\oQ_{2} \cup \oQ_{4}$ into $\oQ_{2}$,  and therefore
    $\Re(\astar(\xi) z)\le 0$ for $\xi\in \oQ_2\cup \oQ_4$ and $z\in \oQ_1.$ 
    Suppose that $w$ is the contour given by
    \begin{equation}
    w(t) := kt \sqrt{2} e^{-3i\pi/8} \sqrt{1+\frac{t^2}{2} e^{i\pi/4} }  \,,
    \end{equation}
    for $t\in \bbR$. This contour is contained in the sector $\overline{S}_{\pi/4}$, is an odd function on the real axis, is such that both $w(t)$ and $A(w(t))$ are analytic functions of $t$ in a strip containing the real axis, and there exists a $\delta>0$ such that $w(t)$ and $A(w(t))$ have convergent Taylor expansions centered at $t=0$ in $|t|<\delta$.
    The purpose of this contour is three fold: to avoid the singularities of the integrand in $I(z)$, to ensure that the curve asymptotes at $\infty$ at an angle between $(\pi/2, \pi)$, and most importantly, so that the exponent $\astar(w(t))$ simplifies to
    \begin{equation}
    \begin{aligned}
    \astar(w(t)) &= i \sqrt{k^2 - w(t)^2} \,, \\
    &= ik - kt^2e^{-i\pi/4} \, .
    \end{aligned}
    \end{equation}
    It is not difficult to show, using the estimates above and Cauchy's theorem, that
    \begin{equation}
    \begin{aligned}
          I(z) &= \int_{w} e^{\astar z}A(\xi) d\xi\,\\
          &= e^{ikz}\int_{-\infty}^{\infty}
 e^{-kt^2 ze^{-i\pi/4} }  A(w(t)) w'(t) \, dt    \,.    
 \end{aligned}
    \end{equation}
    
    Letting $b(t) = A(w(t))w'(t)$ and using the properties of $A$, and $w$, we note that $b$ is an analytic
    function in a strip around the real axis, is an even function of $t$ and 
    there exist a constant $C_{0}$ such that
    \begin{equation}
      \begin{aligned}
       |w(t)| &\leq k(1+t^2)\,,\\
       |b(t)| &= |A(w(t)) w'(t)| \leq C_{\delta} \left( (1+|w|)^{(p+1)} + e^{c|\Im(w)|} \right) |w'(t)| \leq C_{0} e^{2ck(1+t^2)} \, ,
       \end{aligned}
    \end{equation}
    for $t \in \bbR$. Moreover, $b$ has a Taylor expansion centered at $t=0$
    valid in $|t_{0}|<\delta$, i.e. there exist coefficients $b_{j}$, $j=1,2,\ldots \infty$, and constants $R_{j}$ such that
    \begin{equation}
      \left| b(t) - \sum_{j=0}^{N} b_{j} t^{2j} \right| \leq R_{N} |t|^{2N+2} \, ,
    \end{equation} 
    for all $|t| < \delta$. To establish the asymptotic expansions
    in~\eqref{eq:de_asym}, we split the integral into two parts $|t|<\delta$ and
    the remainder.  For the infinite piece, we have
    \begin{equation}
    \begin{aligned}
   \left| \int_{\bbR \setminus [-\delta, \delta]} e^{-kt^2 ze^{-i\pi/4}} b(t) dt \right| &\leq \int_{\bbR \setminus [-\delta, \delta]} e^{-kt^2 |z|/2} |b(t)| dt \,, \\
   &\leq C_{0} e^{2ck} \int_{\bbR \setminus [-\delta, \delta]} e^{-kt^2 |z|/2 + 2ck t^2} dt \,,\\
   & \leq C_{0} e^{2ck} \frac{e^{-k\left(\frac{|z|}{2} - 2c \right)\delta^2}}{\left(\frac{|z|}{2} - 2c \right)\delta} \,, 
   \end{aligned}
   \label{eq:inf-steepest-est}
    \end{equation}
 where the last inequality follows from the fact that $|z|> 4c$
 and~\cite[\href{https://dlmf.nist.gov/7.8.E3}{(7.8.3)}]{NIST:DLMF}.
 
 For the integral on $[-\delta, \delta]$, we have
 \begin{equation}	
\int_{-\delta}^{\delta} e^{-kt^2 ze^{-i\pi/4}} b(t) dt = \int_{-\delta}^{\delta} e^{-kt^2 ze^{-i\pi/4}} \sum_{j=0}^{N} b_{j} t^{2j}  dt  + \int_{-\delta}^{\delta}  e^{-kt^2 ze^{-i\pi/4}} r_{N}(t) dt \, ,
 \end{equation}
 where $|r_{N}(t)| \leq R_{N} |t|^{2N+2}$. Using the identity
 \begin{equation}
 \int_{-\infty}^{\infty} e^{-kt^2 ze^{-i\pi/4}} t^{2j} dt = \frac{\Gamma(j+1/2)}{(kze^{-i\pi/4})^{j+1/2}} \, ,
 \end{equation}
 and an estimate similar to equation~\eqref{eq:inf-steepest-est}, we get
 \begin{equation}
 \int_{-\delta}^{\delta} e^{-kt^2 ze^{-i\pi/4}} \sum_{j=0}^{N} b_{j} t^{2j}  dt = \sum_{j=0}^{N} b_{j} \frac{\Gamma(j+1/2))}{(kze^{-i\pi/4})^{j+1/2}} + O\left( e^{-\frac{|z|}{2} k \delta^2}\right) \,.
 \label{eq:bounded-steepest-est}
 \end{equation}
The remainder term $r_{n}(t)$ can be bounded in a similar manner so that
\begin{equation}
\label{eq:bounded-steepest-est-rem}
\left| \int_{-\delta}^{\delta} e^{-kt^2 ze^{-i\pi/4}} r_N(t) dt \right| \leq \frac{C_{N}}{|z|^{N+3/2}} \,, 
\end{equation}
for some $C_{N}>0$. The proof then follows by combining
equations~\eqref{eq:inf-steepest-est},~\eqref{eq:bounded-steepest-est},
and~\eqref{eq:bounded-steepest-est-rem}.
\end{proof}
In the following lemma, we prove asymptotic estimates for the kernels $k_{C}$, and $k_{D}$.

\begin{lemma}
\label{lem:asymp-est-kcd}
Suppose $q^{l,r}$ satisfy the hypothesis in Theorem~\ref{prop:suitable}. Then the corresponding kernels $k_{C}$ and $k_{D}$ satisfy the estimates~\eqref{eq:cc-est} and~\eqref{eq:cm-est} with $\alpha = 3/2$ and $1/2$ respectively.
\end{lemma}

\begin{proof}
As an example, we  prove the estimate for~$k_D$ when~$x_2,y_2\in \Gamma_R$. The other cases  follow by  nearly identical proofs.
From Theorem~\ref{prop:suitable}, it follows that there exists functions $A_{++,r}(\xi)$ and $A_{++,l}(\xi)$ such that
$\tilde{w}^{l,r}(\xi; x_{2}, y_{2}) = A_{++,l,r}(\xi) e^{\astar(\xi)(x_{2} + y_{2} - 2d)}$. This implies that
\begin{equation}
k_{D}(x_{2}, y_{2}) = \int_{\gamma_{F}} (A_{++,r}(\xi) - A_{++,l}(\xi)) e^{\astar(\xi)(x_{2} + y_{2} - 2d)} d\xi \, .
\end{equation}
Furthermore, Theorem~\ref{prop:suitable} also implies that $A_{++,r}(\xi) - A_{++,l}(\xi)$ satisfy the conditions in
Lemma~\ref{lem:steep_de}. The result then follows from Lemma~\ref{lem:steep_de} since $x_{2} + y_{2} - 2d \in \overline{Q}_{1}$.
   
Similarly, the kernel~$k_C$ has the form 
        \begin{equation}
        k_C(x_2,y_2)=e^{ik (x_2+y_2-2d)}\int_{\gamma_{F}} \xi^2 \lp A_{++,r}(\xi)-A_{++,l}(\xi)\rp e^{-(\astar+ik)(x_2+y_2-2d)}d\xi,
    \end{equation}
    which is of the form required by Lemma~\ref{lem:steep_de} with
    \begin{equation}
        A(\xi)=\xi^2 \lp A_{++,r}(\xi)-A_{++,l}(\xi)\rp.
    \end{equation}
    Since this has a second order zero at~$\xi=0$, the~$j=0$ term in~\eqref{eq:de_asym} vanishes. Repeating this argument for the other regions gives the desired result.
\end{proof}

Next we turn our attention to the asymptotic behavior of $w^{l,r}$. In the
remainder, for notational convenience, we will assume that $y_{1} = 0$, since
the kernels are only a function of $x_{1} - y_{1}$. Referring to
equation~\eqref{eq:IFT-appb}, we note that for any finite $x_{1}, y_{1}$, the
functions $w^{l,r}$ agree with the construction
in~\cite{epstein2023solvinga}. However, the contour $\gamma_{F}$ is not suitable
for analyzing the asymptotics as $|x_1|\to\infty$ as it involves values of $\xi$
with $\Im \xi$ of both signs.  For these calculations one needs to deform to the
contours $\Gamma^{\pm}_{\nu}$ for $\pm x_{1}>0$, introduced
in~\cite{epstein2023solvinga}, see Figure~\ref{fig:contour-gamma-nu}. These
deformations require crossing poles of $w^{l,r}(\xi;x_2,y_2),$ which leads to a
residue contribution composed of waveguide modes: for $\pm x_1>0$ we have with
$l\leftrightarrow -,\, r\leftrightarrow +,$
\begin{equation}\label{eqn5.21.103}
     w^{l,r}(x_1,x_2;0,y_2) =\int_{\Gamma^{\pm}_{\nu}} e^{i\xi x_1}\tilde w^{l,r}(\xi;x_2,y_2)d\xi+\sum_{j=1}^{N_{l,r}}v^{\pm}_j(x_2)v^{\pm}_j(y_2)e^{\pm i\xi^{\pm}_jx_1}.
\end{equation}
In this representation the contributions to the integral from non-real $\xi$
decay exponentially as $\pm x_1\to\infty.$ Where $\pm x_2\gg d$ the functions
$v^{\pm}_j(x_2)=c_{j,\pm}^{\pm}e^{\mp x_2\sqrt{(\xi^{\pm}_{j})^2-k^2}},$ and so
obviously have analytic continuations to $\GammaC,$ which decay exponentially as
$|\Re x_2|\to\infty,$ but do not decay as $|\Im x_2|\to \infty.$ This leads to
the restriction that the data belongs to a space
$\cC^{\alpha,\beta}\oplus\cC^{\alpha+\frac 12,\beta},$ with
{$\beta\ge 0$}, and the restriction that $y_{2}$ lie on a contour,
$\tGamma,$ that satisfies the slope condition in~\eqref{eq:slop-cond}.  In the
following lemma, we establish the asymptotic properties of $w^{l,r}$.

\begin{figure}[h]
    \centering
    \includegraphics[width = 0.6\textwidth]{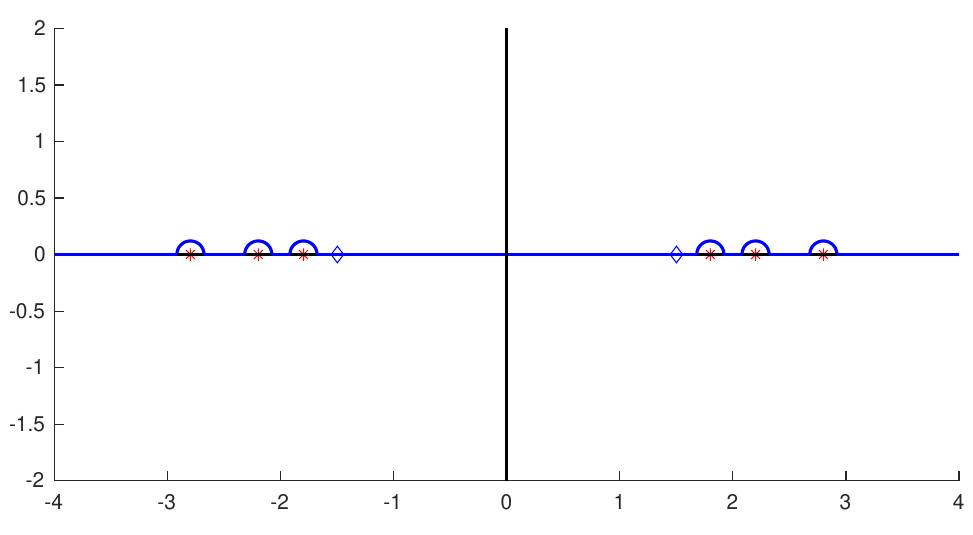}
    \caption{Plot in the $\xi$-plane showing $\Gamma_\nu^+$ in blue. The zeros of the Wronskian are marked by red stars, and the branch points are marked by diamonds. }\label{fig:contour-gamma-nu}
\end{figure}

\begin{lemma}\label{lemma:w_decay}
Suppose that $y_{2} \in \tGamma \subset \cG$ where $\tGamma$ satisfies the slope condition in~\eqref{eq:slop-cond}.
Then for any $n \in \mathbb{N}_{0}$,
    \begin{align}\label{eq:w_bd}
        |w^{l,r}(x_1,x_2;0,y_2)| + |\partial_{x_1}w^{l,r}(x_1,x_2;0,y_2)| \leq C_n\frac{1+|x_1|^n + |y_2|^n}{1+|y_2|^{2n+1}} 
    \end{align}
    uniformly in $\bx \in \mathbb{R}^2 \cap\{|x_2| \le \lt\}$ and $y_2 \in \tilde{\Gamma}$.
\end{lemma}
\begin{proof}
Note that various decay estimates for $w^{l,r} (x_1, x_2;0, y_2)$ with $x_1, x_2$ and $y_2$ real were already established in  \cite{epstein2023solvingb}. All that remains to be shown is that the estimates continue to hold when $y_{2} \in \tGamma$.
    Assume $x_1 > 0$ for concreteness. We recall \eqref{eqn5.21.103}, which states that
    \begin{align*}
        w(x_1,x_2;0,y_2) =
        w_0 (x_1, x_2; y_2) + w_{{\rm wg}} (x_1, x_2; y_2)
    \end{align*}
    where
    \begin{align}\label{eq:w2}
        w_0 (x_1, x_2; y_2) = \int_{\Gamma^{+}_{\nu}} e^{i\xi x_1}\tilde w(\xi;x_2,y_2){\rm d}\xi, \qquad w_{{\rm wg}} (x_1, x_2; y_2) = \sum_{j=1}^{N_{{\rm wg}}}v_j(x_2)v_j(y_2)e^{ i\xi_jx_1}
    \end{align}
    and we have dropped the superscripts $l,r$ for convenience. The waveguide contribution $w_{{\rm wg}}$ decays exponentially in $|\Re y_2|$ and thus clearly satisfies \eqref{eq:w_bd}.

    We next consider $w_0$.
    By the decay of $\tilde{w} (\cdot \; ; x_2, y_2)$ established in Theorem~\ref{prop:suitable}, $$\int_{\Gamma^{+}_{\nu}} |\tilde w(\xi;x_2,y_2)|{\rm d}\xi + \int_{\Gamma^{+}_{\nu}} |\xi \tilde w(\xi;x_2,y_2)|{\rm d}\xi$$ is a continuous function of $(x_2, y_2)$. Since $\Im \xi \ge 0$ for all $\xi \in \Gamma_\nu^+$, this means that $w_0 (x_1, x_2; y_2)$ and $\partial_{x_1} w_0 (x_1, x_2; y_2)$ are bounded uniformly in $x_1 \in \mathbb{R}$ over $(x_2, y_2)$ in compact sets. It remains to bound $w_0 (x_1, x_2; y_2)$ and $\partial_{x_1} w_0 (x_1, x_2; y_2)$ for $y_2 \in \tGamma \cap \{ |\Re y_2| > d\} \cap \{c|\Re y_2| \le |\Im y_2| \le C |\Re y_2|\}$.

    Using the shorthand $z := |x_2 - d| + \sqrt{(y_2 - d)^2}$, we invoke Theorem~\ref{prop:suitable} to write
    \begin{align*}
        \tilde{w} (\xi; x_2, y_2) = e^{-\astar (\xi) z} A(\xi, x_2), \qquad \xi \in \Gamma_\nu^+, \quad x_2 \in \mathbb{R}, \quad y_2 \in \tGamma \cap \{\Re y_2 > d\},
    \end{align*}
    where $A(\xi, x_2) = B(\astar (\xi), x_2)/\astar (\xi)$ with $B (\zeta, x_2)$ analytic in $\zeta \in \mathbb{C} \setminus \{\astar (\xi_j)\}$.
    We estimate the integral in \eqref{eq:w2} by separately handling the region where $\astar (\xi)$ is small. To this end, fix $0 < \delta < \min \{k, \xi_1 - \nu - k\}$ and let $\psi \in \mathcal{C}^\infty_c ((-k-\delta, -k+\delta) \cup (k-\delta, k+\delta))$ such that $\psi \equiv 1$ in $(-k-\delta/2, -k+\delta/2) \cup (k-\delta/2, k+\delta/2)$. This definition ensures that $\Gamma_\nu^+$ is real on the support of $\psi$, and we extend $\psi$ to be $0$ on the parts of $\Gamma_\nu^+$ that are not real. We write $w_0 = w_{0,0} + w_{0,1}$, where
    \begin{align*}
        w_{0,1} (x_1, x_2; y_2) := \int_{\Gamma^{+}_{\nu}} e^{i\xi x_1}\tilde w(\xi;x_2,y_2) \psi(\xi){\rm d}\xi.
    \end{align*}
    Looking first at $w_{0,0},$ since $\Re \astar (\xi) \ge \sqrt{\delta k - \delta^2/4} =: \eta$ and $0 \le \Im \astar (\xi) \le 2 \nu (\xi_p + \nu)$ for all $\xi \in \Gamma_\nu^+ \cap \supp (1 - \psi)$,
    \begin{align}\label{eq:w_02}
        |w_{0,0} (x_1, x_2; y_2)| \le \int_{\Gamma^{+}_{\nu}} e^{-\Re \astar (\xi) \Re y_2 + \Im \astar (\xi) \Im y_2} |A (\xi, x_2) (1-\psi(\xi))| {\rm d} \xi
        \le C e^{2\nu (\xi_p + \nu) \Im y_2-\eta\Re y_2}
    \end{align}
    uniformly in $\bx \in \mathbb{R}^2\cap \{x_1 > 0\} \cap\{|x_2| \le N\}$ and $y_2 \in \tGamma \cap \{\Re y_2 > d\}$.

    The remaining contribution is given by
    \begin{align*}
        w_{0,1} (x_1, x_2; y_2) = \int_{\Gamma_\nu^+}e^{i \xi x_1 - \astar (\xi) y_2}\frac{B(\astar (\xi), x_2)}{\astar (\xi)} \psi (\xi) {\rm d}\xi.
    \end{align*}
    After a change of variables in the integral, this becomes
    \begin{align*}
        w_{0,1} (x_1, x_2; y_2) &= 2\int_{0}^{\eta_+} \cos (\sqrt{\zeta^2 + k^2}x_1)e^{- \zeta y_2}\frac{B(\zeta, x_2)}{\sqrt{\zeta^2 + k^2}} \psi (\sqrt{\zeta^2 + k^2}) {\rm d}\zeta\\
        &\qquad +2i\int_{0}^{\eta_-} \cos (\sqrt{k^2 - \zeta^2}x_1)e^{i \zeta y_2}\frac{B(-i\zeta, x_2)}{\sqrt{k^2 - \zeta^2}} \psi (\sqrt{k^2 - \zeta^2}) {\rm d}\zeta,
    \end{align*}
    where we have assumed without loss of generality that $\psi$ is even, and have defined $\eta_+ := \sqrt{(k+\delta)^2- k^2}$ and $\eta_- := \sqrt{k^2 - (k-\delta)^2}$. It is clear that the first term on the right-hand side is bounded by $C/\Re y_2$, with the second bounded by $C/\Im y_2$.   
We note that almost identical arguments apply to the case for which $y_2 \in \tGamma \cap \{ \Re y_2 < -d\}$ and $x_1 < 0.$ Enforcing the comparability constraint $c |\Re y_2| \le |\Im y_2| \le C |\Re y_2|$, and choosing $\nu > 0$ sufficiently small, \eqref{eq:w_02} together with the above bounds imply that
    \begin{align*}
        |w_0 (x_1, x_2; y_2)| \le \frac{C}{1 + |y_2|}, \qquad \bx \in \mathbb{R}^2 \cap \{|x_2| \le \lt\}, \quad y_2 \in \tGamma.
    \end{align*}
    It is clear that the above arguments extend to $\partial_{x_1} w_0$, thus we have verified \eqref{eq:w_bd} when $n=0$.

    To extend this result to arbitrary $n$, we define $\varphi (\xi, \zeta, x_1, y_2) := i \xi x_1 - \zeta y_2$ and use that
    $$\partial_\xi \varphi (\xi, \astar (\xi), x_1, y_2) = ix_1 - \astar' (\xi) y_2 = \varphi (\astar (\xi), \xi, x_1, y_2)/\astar (\xi)$$
    to write
    \begin{align*}
        w_{0,1} (x_1, x_2; y_2) =
        \int_{\Gamma_\nu^+} \frac{B(\astar (\xi), x_2)\psi (\xi)}{\astar (\xi)}\left(\frac{\astar (\xi)}{\varphi (\astar (\xi), \xi, x_1, y_2)}\partial_\xi \right)^n e^{\varphi (\xi, \astar (\xi), x_1, y_2)}{\rm d}\xi.
    \end{align*}
    Now, observe that since $\astar (\mathbb{R}) \subset \mathbb{R} \cup i \mathbb{R}$, the function $\varphi (\astar (\xi), \xi, x_1, y_2)$ is nonzero for all $(\xi, x_1, y_2) \in \supp (\psi) \times \mathbb{R} \times Q_1$, and $\partial_\xi \varphi (\astar (\xi), \xi, x_1, y_2) = \varphi (\xi, \astar (\xi),x_1, y_2)/\astar (\xi)$ is locally integrable. We stress here that the comparability constraint implies that $\arg y_2$ is bounded away from zero in the above integral. Therefore, we can integrate by parts to obtain
    \begin{align*}
        w_{0,1} (x_1&, x_2; y_2) = (-1)^n \int_{\Gamma_\nu^+} e^{\varphi (\xi,\astar (\xi), x_1, y_2)}\left(\partial_\xi \frac{\astar (\xi)}{\varphi (\astar (\xi), \xi, x_1, y_2)}\right)^n \frac{B(\astar (\xi), x_2)\psi (\xi)}{\astar (\xi)} {\rm d} \xi,
    \end{align*}
    which by the local boundedness of the $\partial^n_\zeta B(\zeta, x_2)$ can be written as
    \begin{align*}
        w_{0,1} (x_1&, x_2; y_2) = \int_{\Gamma_\nu^+}e^{\varphi (\xi,\astar (\xi), x_1, y_2)}\frac{1}{\astar (\xi)} \sum_{j=0}^n \frac{\phi_j (\xi, x_1, y_2)}{\varphi^{n+j}(\astar (\xi), \xi, x_1, y_2)} {\rm d} \xi,
    \end{align*}
    where the $\phi_j$ are compactly supported in $\xi$ and satisfy $|\phi_j (\xi, x_1, y_2)| \le C (1 + |x_1|^j + |y_2|^j)$ uniformly in all variables. As above, we change variables in the integral to obtain
    \begin{align*}
        w_{0,1} (x_1, x_2; y_2) &= \sum_{\eps \in \{-1, 1\}} \Bigg\{ \int_0^{\eta_+}e^{\eps i\sqrt{\zeta^2 + k^2}x_1}e^{- \zeta y_2}\frac{1}{\sqrt{\zeta^2 + k^2}}\sum_{j=0}^n \frac{\phi_j (\eps\sqrt{\zeta^2 + k^2}, x_1, y_2)}{\varphi^{n+j} (\zeta, \eps \sqrt{\zeta^2 + k^2}, x_1, y_2)} {\rm d}\zeta\\
        &\qquad +i\int_{0}^{\eta_-} e^{\eps i\sqrt{k^2 - \zeta^2}x_1}e^{i \zeta y_2}\frac{1}{\sqrt{k^2 - \zeta^2}} \sum_{j=0}^n \frac{\phi_j (\eps \sqrt{k^2 - \zeta^2}, x_1, y_2)}{\varphi^{n+j} (-i \zeta, \eps\sqrt{k^2 - \zeta^2}, x_1, y_2)}{\rm d}\zeta \Bigg\}.
    \end{align*}
    Since $|1/\varphi (\zeta, \eps \sqrt{\zeta^2 + k^2}, x_1, y_2)| + |1/\varphi (-i\zeta, \eps \sqrt{k^2 - \zeta^2}, x_1, y_2)| \le C/|y_2|$ uniformly in $0 < \zeta < \eta_+$,
    \begin{align*}
        |w_{0,1} (x_1, x_2; y_2)| \le C \sum_{j=0}^n \frac{1+|x_1|^j+|y_2|^j}{|y_2|^{n+j}} \left( \int_0^{\eta_+} e^{-\zeta \Re y_2} {\rm d} \zeta +\int_0^{\eta_-} e^{-\zeta \Im y_2} {\rm d} \zeta \right) \le C\frac{1+|x_1|^n+|y_2|^n}{|y_2|^{2n+1}},
    \end{align*}
    where the last inequality follows from the comparability assumption $c |\Re y_2| \le |\Im y_2| \le C |\Re y_2|$.
    Given the exponential decay of $w_{0,0}$ in \eqref{eq:w_02} with $\nu > 0$ sufficiently small, we have now established that
    \begin{align*}
        |w_{0} (x_1, x_2; y_2)|\le C_n\frac{1+|x_1|^n+|y_2|^n}{1 + |y_2|^{2n+1}}, \qquad \bx \in \mathbb{R}^2 \cap \{|x_2| \le \lt\}, \quad y_2 \in \tGamma.
    \end{align*}
    It is again clear that the same estimates apply to $\partial_{x_1} w_0$, and the proof is complete.
\end{proof}

\bibliography{references}
\end{document}